\documentclass[11pt, a4paper,leqno]{amsart}
\usepackage{amsmath,amsthm,amscd,amssymb,amsfonts, amsbsy}
\usepackage{latexsym}
\usepackage{txfonts}
\usepackage{exscale}
\usepackage{mathrsfs}

\usepackage[colorlinks=true, pdfstartview=FitV, linkcolor=blue, citecolor=red, urlcolor=blue]{hyperref} %%% % Includes hyperrefs into pdf

\usepackage{color}

\calclayout
\allowdisplaybreaks

\def\Xint#1{\mathchoice
    {\XXint\displaystyle\textstyle{#1}}%
    {\XXint\textstyle\scriptstyle{#1}}%
    {\XXint\scriptstyle\scriptscriptstyle{#1}}%
    {\XXint\scriptscriptstyle\scriptscriptstyle{#1}}%
    \!\int}
    \def\XXint#1#2#3{{\setbox0=\hbox{$#1{#2#3}{\int}$}
    \vcenter{\hbox{$#2#3$}}\kern-.5\wd0}}
    \def\dashint{\Xint-}

\renewcommand{\chi}{{\bf 1}}

\theoremstyle{plain}
\newtheorem{theorem}[equation]{Theorem}
\newtheorem{lemma}[equation]{Lemma}

\newtheorem{proposition}[equation]{Proposition}

\theoremstyle{definition}

\theoremstyle{remark}
\newtheorem{remark}[equation]{Remark}

\numberwithin{equation}{section}

\def \R{ \mathbb{R} }
\def \N{ \mathbb{N} }

\def \iint{\int\!\!\!\int}
\def \Scal{ \mathcal{S} }
\def \hh{ \mathrm{H} }
\def \pp{ \mathrm{P} }
\def \Gcal { \mathcal{G} }
\def \Grm{ \mathrm{G} }

\def\div{\mathop{\rm div}}
\renewcommand{\Re}{{\rm Re}\,}

\DeclareMathOperator{\supp}{supp}

\begin{document}
\allowdisplaybreaks

\title[Weighted Hardy spaces associated with elliptic operators]{Weighted Hardy spaces associated with elliptic operators.
\\[0.3cm]
{\small Part I: Weighted norm inequalities for conical square functions}}

\author{Jos\'e Mar{\'\i}a Martell}
\address{Jos\'e Mar{\'\i}a Martell
\\
Instituto de Ciencias Matem\'aticas CSIC-UAM-UC3M-UCM
\\
Consejo Superior de Investigaciones Cient{\'\i}ficas
\\
C/ Nicol\'as Cabrera, 13-15
\\
E-28049 Madrid, Spain} \email{chema.martell@icmat.es}

\author{Cruz Prisuelos-Arribas}

\address{Cruz Prisuelos-Arribas
Instituto de Ciencias Matem\'aticas CSIC-UAM-UC3M-UCM
\\
Consejo Superior de Investigaciones Cient{\'\i}ficas
\\
C/ Nicol\'as Cabrera, 13-15
\\
E-28049 Madrid, Spain} \email{cruz.prisuelos@icmat.es}

\thanks{The research leading to these results has received funding from the European Research
Council under the European Union's Seventh Framework Programme (FP7/2007-2013)/ ERC
agreement no. 615112 HAPDEGMT. The first author was supported in part by MINECO Grant
MTM2010-16518, ICMAT Severo Ochoa project SEV-2011-0087. The second author was supported in part by
ICMAT Severo Ochoa project SEV-2011-0087.}

\date{June 24, 2014. \textit{Revised}: \today}



\begin{abstract}
This is the first part of a series of three articles. In this paper, we obtain weighted norm inequalities for different conical square functions associated with the Heat and the Poisson semigroups generated by a second order divergence form elliptic operator with bounded complex coefficients. We find classes of Muckenhoupt weights where the square functions are comparable and/or bounded. These classes are natural from the point of view of the ranges where the unweighted estimates hold. In doing that, we obtain sharp weighted change of angle formulas which allow us to compare conical square functions with different cone apertures in weighted Lebesgue spaces. A key ingredient in our proofs is a generalization of the Carleson measure condition which is more natural when estimating the square functions below $p=2$.
\end{abstract}

\maketitle

\tableofcontents

%%%%%%%%%%%%%%%%%%%%%%%%%%%%%%%%%%%
%%%%%%%%%%%%%%%%%%%%%%%%%%%%%%%%%%%
\section{Introduction}\label{section:intro}
%%%%%%%%%%%%%%%%%%%%%%%%%%%%%%%%%%%
%%%%%%%%%%%%%%%%%%%%%%%%%%%%%%%%%%%

In the last decade, after the solution of the Kato conjecture \cite{AHLMT}, there has been a big interest in developing a Calder\'on-Zygmund theory appropriate for the operators (functional calculus, Riesz transforms, square functions, etc.) that appear naturally associated with divergence form elliptic operators with complex bounded coefficients. In general, these operators cease to be classical Calder\'on-Zygmund operators, as their kernels do not have the required decay or smoothness. This causes, in particular, that their range of boundedness may no longer be the interval $(1,\infty)$ but some proper (small) bounded subinterval containing $p=2$. Auscher, in a very nice monograph (\cite{Auscher}), obtained a new Calder\'on-Zygmund theory adapted to singular ``non-integral'' operators arising from elliptic operators (see \cite{Auscher} for historic remarks and references). A key ingredient in the method is and idea used systematically in \cite{BK1} (see also \cite{HofmannMartell}): in place of using kernels, which do not have reasonable behavior, there is a representation of the operators in question in terms of the Heat semigroup $\{e^{-t\,L}\}_{t>0}$ (or its gradient) that has some integral decay measured in terms of the so-called ``off-diagonal'' or Gaffney type estimates. The bottom line of \cite{Auscher} is that the operators under consideration are bounded precisely in the ranges where either the semigroup or its gradient has a nice behavior.

After Auscher's fundamental monograph there has been quite a number of papers whose goal is to continue with the development of a generalized Calder\'on-Zygmund theory. We shall mention some that are relevant for the goal of the present work. Auscher and the first named author of this paper wrote a series of papers \cite{AuscherMartell:I, AuscherMartell:II, AuscherMartell:III} where the weighted theory was developed and where some appropriate classes of Muckenhoupt weights were found. While vertical square functions (i.e., usual Littlewood-Paley-Stein functionals) behave as expected with and without weights (see, resp., \cite{Auscher, AuscherMartell:III}), conical square functions have better ranges of boundedness in the unweighted case, even going beyond the intervals where the semigroup or its gradient has a nice behavior, see \cite{AuscherHofmannMartell}.

In harmonic analysis, and more in particular in the so-called Calder\'on-Zygmund theory, where the typical range of $L^p$-boundedness is the interval $(1,\infty)$, the natural endpoint spaces are the Hardy space $H^1(\R^n)$ for $p=1$ ($H^p(\R^n)$ for $p<1$) and the space of bounded mean oscillation functions $\mathrm{BMO}(\R^n)$ for $p=\infty$. For instance, it is well-known that the classical Riesz transform (associated with the Laplacian) is bounded from $H^1(\R^n)$ to $L^1(\R^n)$, and it becomes natural to study whether Riesz transforms associated with general elliptic operators behave well in Hardy spaces. Classical real-variable Hardy spaces in $\R^n$ have been deeply studied since the fundamental paper of Stein and Weiss, \cite{Stein-Weiss}, on systems of conjugate harmonic functions. The pioneering paper of Fefferman and Stein \cite{Fef-Ste} showed that, besides the intimate relation between Hardy spaces and
harmonic functions, Hardy spaces can be characterized in terms of general approximations of the identity or by general conical square functions (i.e., area functionals of Lusin type). This eventually led to some developments of these spaces without using their connection with the Laplacian, which could be used for general elliptic operators.  However, this is not the case: if we had that the associated Riesz transform maps continuously $H^1(\R^n)$ into $L^1(\R^n)$, then the interpolation with the $L^2(\R^n)$ boundedness (which is the Kato conjecture) would imply boundedness on $L^p(\R^n)$ for $1<p<2$ violating  some of the results in \cite{Auscher} (see \cite{HofmannMayboroda} for more details). In the same way as real-variable Hardy spaces were originally defined in connection with the Laplacian, in the last decade there has been a big interest in studying Hardy and other related spaces adapted to elliptic operators, see \cite{Auscher-McIntosh-Russ, Auscher-McIntosh-Morris, Auscher-Russ, BCKYY1, BCKYY2, Anh-Duong, Duong-Yan-BMO, Duong-Yan-H1,  HLMMY, HofmannMayboroda, HofmannMayborodaMcIntosh} and the references therein. Among them we highlight \cite{HofmannMayboroda, HofmannMayborodaMcIntosh} where Hardy  spaces, $\mathrm{BMO}$, and some other related spaces adapted to general divergence form elliptic operators were successfully developed.

The goal of this series of papers is to continue with the development of generalized Calder\'on-Zygmund theory for elliptic operators and study the corresponding weighted Hardy spaces. Classical (or Laplacian-adapted) weighted Hardy spaces were first introduced by J. Garc{\'\i}a-Cuerva \cite{GC-w-HP} (see also \cite{ST}). Our aim is to present a satisfactory Hardy space theory for general elliptic operators with bounded complex coefficients complementing the results in \cite{HofmannMayboroda, HofmannMayborodaMcIntosh}. Our spaces generalize those in \cite{BCKYY1, Anh-Duong, Liu-Song} in the Euclidean setting where the adapted weighted Hardy spaces are associated with a friendlier class of non-negative self-adjoint operators whose heat kernel satisfy Gaussian upper bounds. We also generalize \cite{BCKYY2} by considering molecules living in weighted spaces and also by being able to fully recover \cite{HofmannMayboroda, HofmannMayborodaMcIntosh}.

In the first part of the series, which is the present paper, we study the weighted norm inequalities for conical square functions. We establish boundedness and comparability in weighted Lebesgue spaces of different square functions  using the Heat and Poisson semigroups. In the second part, \cite{Mar-Pri-2}, we shall use  these square functions to define several weighted Hardy spaces $H^1_L(w)$. We show that they are one and the same in view of the fact that the square functions are comparable in the corresponding weighted spaces. We also show that Hardy spaces can be equivalently defined using molecules and/or non-tangential maximal functions. The study of $H^p_L(w)$ for other values of $p$ is in the third part \cite{Pri}.

In contrast with \cite{HofmannMayboroda, HofmannMayborodaMcIntosh}, where some of the unweighted estimates for the conical square function are taken off-the-shelf from \cite{Auscher}, our first difficulty consists in proving that conical square functions are bounded on weighted spaces for some classes of weights ``adapted'' to the unweighted range of boundedness. This was left open in \cite{AuscherHofmannMartell} since some of the existing arguments naturally split the boundedness into the cases  $p<2$ and $p>2$. That procedure, as learned from \cite{AuscherMartell:III}, is inefficient when adding weights: to obtain the right class of weights one has to be able to work with  the whole interval where the unweighted estimates hold. Splitting the interval would lead to some distortion in the class of weights. To illustrate this, let us recall that in \cite{AuscherHofmannMartell} it is shown that the square function $G_{0,H}$, defined in \eqref{square-H-2} below, is bounded on $L^p(\R^n)$  for every $p_-(L)<p<\infty$ where $p_-(L)$, introduced in \eqref{p-} below, is strictly smaller than $2$. Using the approach in \cite{AuscherHofmannMartell}, and ``stepping'' at $p=2$, this  square function is bounded on $L^p(w)$ for every $2<p<\infty$ and $w\in A_{p/2}$ (see the precise definitions below). However, as we shall see in Theorem \ref{thm:SF-Heat}, one has boundedness on $L^p(w)$ for every $p_-(L)<p<\infty$ and $w\in A_{p/p_-(L)}$, hence in a bigger range and a wider class of weights since $p_-(L)<2$. Moreover, the obtained class of weights is the natural one adapted to the unweighted range $(p_-(L),\infty)$, in view of  the version of the Rubio de Francia extrapolation theorem in \cite[Theorem 4.9]{AuscherMartell:I} or \cite[Theorem 3.31]{CruzMartellPerez}.  See also \cite{Anh-Duong, Liu-Song} for related issues. The goal of the present paper is to present a library of weighted norm inequalities for the different square functions that can be defined using the Heat or the Poisson semigroup. We look for two different types of estimates: the first type of estimates will give us comparison among the square functions, and the second boundedness. The main idea is to show that all these square functions can be controlled by either $\Scal_{\hh}$ or $\Gcal_{\hh}$ (see below), and independently to obtain that these are bounded on $L^p(w)$ for some range of $p$'s and for some class of Muckenhoupt weights.

To be more precise let us set our background hypotheses. Let $A$ be an $n\times n$ matrix of complex and
$L^\infty$-valued coefficients defined on $\R^n$. We assume that
this matrix satisfies the following ellipticity (or \lq\lq
accretivity\rq\rq) condition: there exist
$0<\lambda\le\Lambda<\infty$ such that
$$
\lambda\,|\xi|^2
\le
\Re A(x)\,\xi\cdot\bar{\xi}
\quad\qquad\mbox{and}\qquad\quad
|A(x)\,\xi\cdot \bar{\zeta}|
\le
\Lambda\,|\xi|\,|\zeta|,
$$
for all $\xi,\zeta\in\mathbb{C}^n$ and almost every $x\in \R^n$. We have used the notation
$\xi\cdot\zeta=\xi_1\,\zeta_1+\cdots+\xi_n\,\zeta_n$ and therefore
$\xi\cdot\bar{\zeta}$ is the usual inner product in $\mathbb{C}^n$. Note
that then
$A(x)\,\xi\cdot\bar{\zeta}=\sum_{j,k}a_{j,k}(x)\,\xi_k\,\bar{\zeta_j}$.
Associated with this matrix we define the second order divergence
form elliptic operator
$$
L f
=
-\div(A\,\nabla f),
$$
which is understood in the standard weak sense as a maximal-accretive operator on $L^2(\R^n,dx)$ with domain $\mathcal{D}(L)$ by means of a
sesquilinear form.

The operator $-L$ generates a $C^0$-semigroup, $\{e^{-t L}\}_{t>0}$, of contractions on $L^2(\mathbb{R}^n)$ which is called the Heat semigroup. Using this semigroup and the corresponding Poisson semigroup, $\{e^{-t\,\sqrt{L}}\}_{t>0}$, one can define different conical square functions which all have an expression of the form
$$
Q f(x)
=\left(\iint_{\Gamma(x)}|T_t f(y)|^2 \frac{dy \, dt}{t^{n+1}}\right)^{\frac{1}{2}},
\qquad
x\in\R^n.
$$
where $\Gamma(x): =\{(y,t)\in \R^{n+1}_+: |x-y|<t\}$ denotes the cone (of aperture 1) with vertex at $x\in\R^n$ (see Section \ref{section:TSpaces} for more details including  a discussion about the use of cones with different apertures).  More precisely, we introduce the following conical square functions written in terms of the Heat semigroup, $\{e^{-t L}\}_{t>0}$, (hence the subscript $\hh$): for every $m\in \mathbb{N}$,
\begin{align} \label{square-H-1}
\mathcal{S}_{m,\hh}f(x) & = \left(\iint_{\Gamma(x)}|(t^2L)^{m} e^{-t^2L}f(y)|^2 \frac{dy \, dt}{t^{n+1}}\right)^{\frac{1}{2}},
\end{align}
and, for every $m\in \mathbb{N}_0:=\mathbb{N}\cup\{0\}$,
\begin{align}
\mathrm{G}_{m,\hh}f(x)& =\left(\iint_{\Gamma(x)}|t\nabla_y(t^2L)^m e^{-t^2L}f(y)|^2 \frac{dy \, dt}{t^{n+1}}\right)^{\frac{1}{2}},
\label{square-H-2}\\[4pt]
\mathcal{G}_{m,\hh}f(x)&
=
\left(\iint_{\Gamma(x)}|t\nabla_{y,t}(t^2L)^m e^{-t^2L}f(y)|^2 \frac{dy \, dt}{t^{n+1}}\right)^{\frac{1}{2}}.
\label{square-H-3}
\end{align}

In the same manner, let us consider conical square functions associated with the Poisson semigroup, $\{e^{-t \sqrt{L}}\}_{t>0}$, (hence the subscript $\pp$):  given $K\in \mathbb{N}$,
\begin{align}
\mathcal{S}_{K,\pp}f(x)
&=
\left(\iint_{\Gamma(x)}|(t\sqrt{L}\,)^{2K} e^{-t\sqrt{L}}f(y)|^2 \frac{dy \, dt}{t^{n+1}}\right)^{\frac{1}{2}},
\label{square-P-1}
\end{align}
and for every $K\in \mathbb{N}_0$,
\begin{align}
\mathrm{G}_{K,\pp}f(x)
&=\left(\iint_{\Gamma(x)}|t\nabla_y (t\sqrt{L}\,)^{2K} e^{-t\sqrt{L}}f(y)|^2 \frac{dy \, dt}{t^{n+1}}\right)^{\frac{1}{2}},
\label{square-P-2}
\\[4pt]
\mathcal{G}_{K,\pp}f(x)
&=
\left(\iint_{\Gamma(x)}|t\nabla_{y,t}(t\sqrt{L}\,)^{2K} e^{-t\sqrt{L}}f(y)|^2 \frac{dy \, dt}{t^{n+1}}\right)^{\frac{1}{2}}.
\label{square-P-3}
\end{align}
Corresponding to the case $m=0$ or $K=0$, we simply write $\mathrm{G}_{\hh}f:=\mathrm{G}_{0,\hh}f$,
$\mathcal{G}_{\hh}f:=\mathcal{G}_{0,\hh}f$,  $\mathrm{G}_{\pp}f:=\mathrm{G}_{0,\pp}f$, and
$\mathcal{G}_{\pp}f:=\mathcal{G}_{0,\pp}f$. Besides, we set $\Scal_{\hh}f:=\Scal_{1,\hh}f$, and $\Scal_{\pp}f:=\Scal_{1,\pp}f$.

\medskip

In order to give the statements of our main results we need to introduce some notation. As in \cite{Auscher} and \cite{AuscherMartell:II}, we denote by $(p_-(L),p_+(L))$ the maximal open interval on which the Heat semigroup, $\{e^{-tL}\}_{t>0}$, is uniformly bounded on $L^p(\mathbb{R}^n)$:
\begin{align}\label{p-}
p_-(L) &:= \inf\left\{p\in(1,\infty): \sup_{t>0} \|e^{-t^2L}\|_{L^p(\mathbb{R}^n)\rightarrow L^p(\mathbb{R}^n)}< \infty\right\},
\\[4pt]
p_+(L)& := \sup\left\{p\in(1,\infty) : \sup_{t>0} \|e^{-t^2L}\|_{L^p(\mathbb{R}^n)\rightarrow L^p(\mathbb{R}^n)}< \infty\right\}.
\label{p+}
\end{align}
Note that in place of the semigroup $\{e^{-t L}\}_{t>0}$ we are using its rescaling $\{e^{-t^2 L}\}_{t>0}$. We do so since all the ``Heat'' square functions are written using the latter and also because in the context of the off-diagonal estimates discussed below it will simplify some computations.

Besides, for every $K\in\mathbb{N}_0$ let us set
$$
p_+(L)^{K,*}:=
\left\{
\begin{array}{ll}
\dfrac{p_+(L)\,n}{n-(2K+1)\,p_+(L)}, &\quad\mbox{ if}\quad(2K+1)\,p_+(L)<n,
\\[10pt]
\infty, &\quad\mbox{ if}\quad(2K+1)\,p_+(L)\ge n.
\end{array}
\right.
$$
Corresponding to the case $K=0$, we write $p_+(L)^{*}:=p_+(L)^{0,*}$.

We shall work with Muckenhoupt weights, $w$, which are locally integrable positive functions. We say that a weight $w\in A_1$ if, for every ball $B\subset \mathbb{R}^n$, there holds
$$
\dashint_Bw(x) \, dx\leq C w(y), \quad \textrm{for a.e. } y\in B,
$$
or, equivalently, $\mathcal{M} w\le C\,w$  a.e., where $\mathcal{M}$ denotes the Hardy-Littlewood
maximal operator over balls in $\R^n$.
For each $1<p<\infty$, we say that $w\in A_p$ if it satisfies
$$\left(\dashint_Bw(x) \, dx\right)\left(\dashint_Bw(x)^{1-p'} \, dx\right)^{p-1}
\leq C, \quad \forall B\subset \mathbb{R}^n.$$
In the previous expression $p'=p/(p-1)$ denotes the conjugate exponent of $p$.
The reverse H\"older classes are defined as follows: for each $1<s<\infty$,
$w\in RH_s$ if, for every ball $B\subset \mathbb{R}^n$, we have
$$\left(\dashint_Bw(x)^s \, dx\right)^{\frac{1}{s}}\leq C\dashint_Bw(x) \, dx.$$
For $s=\infty$, $w\in RH_{\infty}$ provided that there exists a constant $C$ such that for every ball $B\subset \mathbb{R}^n$
$$
w(y)\leq C\dashint_Bw(x) \, dx, \quad \textrm{for a.e. }y\in B.
$$
Notice that we have excluded the case $q = 1$ since the class $RH_1$ consists of all
weights, and that is the way $RH_1$ is understood in what follows.

We sum up some of the properties of these classes in the following result, see for instance
\cite{Duo}, \cite{GCRF85}, or \cite{Grafakos}. %For $(vii)$ see \cite{JN}.

\begin{proposition}\label{prop:weights}\
\begin{enumerate}
\renewcommand{\theenumi}{\roman{enumi}}
\renewcommand{\labelenumi}{$(\theenumi)$}
\addtolength{\itemsep}{0.2cm}

\item $A_1\subset A_p\subset A_q$ for $1\le p\le q<\infty$.

\item $RH_{\infty}\subset RH_q\subset RH_p$ for $1<p\le q\le \infty$.

\item If $w\in A_p$, $1<p<\infty$, then there exists $1<q<p$ such
that $w\in A_q$.

\item If $w\in RH_s$, $1<s<\infty$, then there exists $s<r<\infty$ such
that $w\in RH_r$.

\item $\displaystyle A_\infty=\bigcup_{1\le p<\infty} A_p=\bigcup_{1<s\le
\infty} RH_s$.

\item If $1<p<\infty$, $w\in A_p$ if and only if $w^{1-p'}\in
A_{p'}$.

\item For every $1<p<\infty$, $w\in A_p$ if and only if $\mathcal{M}$ is bounded
on $L^p(w)$. Also, $w\in A_1$ if and only if $\mathcal{M}$ is bounded from $L^1(w)$ into $L^{1,\infty}(w)$.

%\item If $1\le q\le \infty$ and $1\le s<\infty$, then $\displaystyle w\in A_q \cap RH_s$ if and only if $ w^{s}\in A_{s\,(q-1)+1}$.

\end{enumerate}
\end{proposition}

\medskip

For a weight $w\in A_{\infty}$, define
\begin{align}\label{rw}
r_w:=\inf\{1\leq r<\infty : w\in A_{r}\},
\qquad
s_w:=\inf\{1\leq s<\infty : w\in RH_{s'}\}.
\end{align}
Notice that according to our definition $s_w$ is the conjugated exponent of the one defined in \cite[Lemma 4.1]{AuscherMartell:I}.
Given $0\le p_0<q_0\le \infty$, $w\in A_{\infty}$, and as in \cite[Lemma 4.1]{AuscherMartell:I}, we have
\begin{align}\label{intervalrs}
\mathcal{W}_w(p_0,q_0):=\left\{p : p_0<p<q_0, w\in A_{\frac{p}{p_0}}\cap RH_{\left(\frac{q_0}{p}\right)'}\right\}
=
\left(p_0r_w,\frac{q_0}{s_w}\right).
\end{align}
If $p_0=0$ and $q_0<\infty$ it is understood that the only condition that stays is $w\in RH_{\left(\frac{q_0}{p}\right)'}$. Analogously, if $0<p_0$ and $q_0=\infty$ the only assumption is $w\in A_{\frac{p}{p_0}}$. Finally $\mathcal{W}_w(0,\infty)=(0,\infty)$.

Our first goal is to study the boundedness of the square functions presented in \eqref{square-H-1}--\eqref{square-P-3} on weighted spaces $L^p(w)$ where $w\in A_\infty$. Our first result establishes the boundedness of the square functions associated with the Heat semigroup. Notice that when $w\equiv 1$, which corresponds to the unweighted case, this result recovers the estimates in the range $(p_-(L),\infty)$ obtained in \cite{AuscherHofmannMartell}.

\begin{theorem}\label{thm:SF-Heat}
Let $w\in A_{\infty}$.
\begin{list}{$(\theenumi)$}{\usecounter{enumi}\leftmargin=1cm \labelwidth=1cm\itemsep=0.2cm\topsep=.2cm \renewcommand{\theenumi}{\alph{enumi}}}

\item $\Scal_{\hh}$, $\Grm_{\hh}$, and  $\Gcal_{\hh}$ are bounded on  $L^p(w)$ for all $p\in \mathcal{W}_w(p_-(L),\infty)$.

\item Given $m\in\mathbb{N}$, $\Scal_{m,\hh}$, $\Grm_{m, \hh}$, and $\Gcal_{m, \hh}$ are bounded on $L^p(w)$ for all $p\in \mathcal{W}_w(p_-(L),\infty)$.
\end{list}
Equivalently, all the previous square functions are bounded on $L^p(w)$ for every $p_-(L)<p<\infty$ and every  $w\in A_{\frac{p}{p_-(L)}}$.
\end{theorem}

The proof of Theorem \ref{thm:SF-Heat} is split into two steps. First, we prove $(a)$ (in doing that we only need to consider $\Scal_{\hh}$ and  $\Gcal_{\hh}$ since $G_{\hh}f\le \Gcal_{\hh}f$). Second, we shall show that the square functions in $(b)$ are all controlled by $\Scal_{\hh}$ in $L^p(w)$ for every $w\in A_\infty$ and $0<p<\infty$ (see Theorem \ref{theor:control-SF-Heat}). Gathering this and $(a)$, the proof of $(b)$ will be complete.

\medskip

Our second result deals with the boundedness of the square functions related to the Poisson semigroup. Here the formulation is more involved
since the ranges where these square functions are bounded, not only depend on $p_-(L)$ and the weight, but also on $p_+(L)$ and the parameter $K$.
We also notice that when $w\equiv 1$ we recover the estimates obtained in \cite{AuscherHofmannMartell}.

\begin{theorem}\label{thm:SF-Poisson}
Let $w\in A_{\infty}$.
\begin{list}{$(\theenumi)$}{\usecounter{enumi}\leftmargin=1cm \labelwidth=1cm\itemsep=0.2cm\topsep=.2cm \renewcommand{\theenumi}{\alph{enumi}}}

\item Given $K\in\mathbb{N}$, $\Scal_{K,\pp}$ is bounded on $L^p(w)$ for all $p\in \mathcal{W}_w(p_-(L),p_+(L)^{K,*})$.

\item Given $K\in\mathbb{N}_0$,  $\Gcal_{K,\pp}$ and $\Grm_{K,\pp}$ are bounded on $L^p(w)$ for all $p\in \mathcal{W}_w(p_-(L),p_+(L)^{K,*})$.
\end{list}
\end{theorem}

The proof of this result is as follows. We shall first show that each square function in $(a)$ and $(b)$ can be controlled by either $\Scal_{\hh}$ or $\Gcal_{\hh}$ in $L^p(w)$ for every $w\in A_\infty$ and $p\in \mathcal{W}_w(0, p_+(L)^{K,*})$  (see Theorem \ref{theor:control-SF-Poisson}). This, in concert with $(a)$ in Theorem \ref{thm:SF-Heat}, will easily lead to the desired estimates.

\medskip

We present the two promised results containing the control of the previous square functions by $\Scal_{\hh}$ and  $\Gcal_{\hh}$. In the first result we deal with the square functions defined in terms of the Heat semigroup.
\begin{theorem}\label{theor:control-SF-Heat}
Given an arbitrary $f\in L^2(\mathbb{R}^n)$ there hold:
\begin{list}{$(\theenumi)$}{\usecounter{enumi}\leftmargin=1cm \labelwidth=1cm\itemsep=0.2cm\topsep=.2cm \renewcommand{\theenumi}{\alph{enumi}}}

\item $\Grm_{m,\hh}f(x)\le \Gcal_{m, \hh}f(x)$, for every $x\in\mathbb{R}^n $ and  for all $m\in\mathbb{N}_0$.

\item Given $m\in\mathbb{N}$,  $\displaystyle \|\Scal_{m,\hh}f\|_{L^p(w)}\lesssim \|\Scal_{\hh}f\|_{L^p(w)}$, for all $w\in A_\infty$ and $0<p<\infty$.

\item Given $m\in\mathbb{N}$,  $\displaystyle \|\Gcal_{m,\hh}f\|_{L^p(w)}\lesssim \|\Scal_{\hh}f\|_{L^p(w)}$, for all $w\in A_\infty$ and $0<p<\infty$.
\end{list}
\end{theorem}

Finally, the following result establishes the control of the square functions associated with the Poisson semigroup.

\begin{theorem}\label{theor:control-SF-Poisson} Given an arbitrary $f\in L^2(\mathbb{R}^n)$ there hold:
\begin{list}{$(\theenumi)$}{\usecounter{enumi}\leftmargin=1cm \labelwidth=1cm\itemsep=0.2cm\topsep=.2cm \renewcommand{\theenumi}{\alph{enumi}}}

\item $\Grm_{K,\pp}f(x)\le \Gcal_{K, \pp}f(x)$, for every $x\in\mathbb{R}^n $ and  for all $K\in\mathbb{N}_0$.

\item Given $K\in\mathbb{N}$, $\displaystyle \|\Scal_{K,\pp}f\|_{L^p(w)}\lesssim \|\Scal_{\hh}f\|_{L^p(w)}$, for all $w\in A_\infty$ and $p\in\mathcal{W}_w(0,p_+(L)^{K,*})$.

\item $\displaystyle \|\Gcal_{\pp}f\|_{L^p(w)}\lesssim \|\Gcal_{\hh}f\|_{L^p(w)}$, for all $w\in A_\infty$ and $w\in\mathcal{W}_w(0,p_+(L)^{*})$.

\item Given $K\in\mathbb{N}$, $\displaystyle \|\Gcal_{K,\pp}f\|_{L^p(w)}\lesssim \|\Scal_{\hh}f\|_{L^p(w)}$, for all $w\in A_\infty$ and $p\in\mathcal{W}_w(0,p_+(L)^{K,*})$.
\end{list}
\end{theorem}

Let us observe that in $(b)$ and $(d)$ (and also $(c)$ with $K=0$), if $(2K+1)\,p_+(L)\ge n$ the corresponding estimates hold for every $w\in A_\infty$ and every $0<p<\infty$. Otherwise, if $(2K+1)\,p_+(L)< n$, each corresponding estimate holds for all $0<p<p_+(L)^{K,*}$ and $w\in RH_{(p_+(L)^{K,*}/p)'}$.

\medskip

The organization of the paper is as follows. In Section \ref{section:OD} we recall the off-diagonal estimates satisfied by the Heat and Poisson semigroups, as well as by the other related objects that define the square functions under study. In Section \ref{section:TSpaces}, we consider weighted estimates in the tent spaces introduced and developed in \cite{CoifmanMeyerStein}. Crucial to us are the change-of-angle formulas which are very useful for comparing square functions in weighted Lebesgue spaces with cones having different apertures and with a precise control of the variation of the angle. Another important tool is the introduction of a modified version of the Carleson measure condition suited to deal with estimates on $L^p$, for $p<2$. As explained above, this will be crucial when obtaining weighted estimates without splitting the argument into $p<2$ and $p>2$ as previously done in \cite{AuscherHofmannMartell}.
Finally, in Section \ref{section:main-results-SF} we prove our main results: Theorems \ref{thm:SF-Heat}, \ref{thm:SF-Poisson}, \ref{theor:control-SF-Heat}, and \ref{theor:control-SF-Poisson}.

%%%%%%%%%%%%%%%%%%%%%%%%%%%%%%%%%%%
%%%%%%%%%%%%%%%%%%%%%%%%%%%%%%%%%%%
\section{Off-diagonal estimates}\label{section:OD}
%%%%%%%%%%%%%%%%%%%%%%%%%%%%%%%%%%%
%%%%%%%%%%%%%%%%%%%%%%%%%%%%%%%%%%%

We briefly recall the notion of off-diagonal estimates. Let $\{T_t\}_{t>0}$ be a family of linear operators
and let $1\le p\leq q\le \infty$. We say that $\{T_t\}_{t>0}$ satisfies $L^p-L^q$ off-diagonal estimates of exponential type, denoted by $\{T_t\}_{t>0}\in \mathcal{F}_\infty(L^p\rightarrow L^q)$, if for all closed sets $E$, $F$, all $f$, and all $t>0$ we have
$$
\|T_{t}(f\,\chi_E)\,\chi_F\|_{L^q(\mathbb{R}^n)}
\leq
Ct^{-n\left(\frac{1}{p}-\frac{1}{q}\right)}e^{-c\frac{d(E,F)^2}{t^2}}\|f\,\chi_E\|_{L^p(\mathbb{R}^n)}.
$$
Analogously, given $\beta>0$, we say that $\{T_t\}_{t>0}$ satisfies $L^p-L^q$ off-diagonal estimates of polynomial type with order $\beta>0$, denoted by $\{T_t\}_{t>0}\in \mathcal{F}_{\beta}(L^p\rightarrow L^q)$ if for all closed sets $E$, $F$, all $f$, and all $t>0$ we have
$$
\|T_{t}(f\,\chi_E)\,\chi_F\|_{L^q(\mathbb{R}^n)}
\leq
Ct^{-n\left(\frac{1}{p}-\frac{1}{q}\right)}\left(1+\frac{d(E,F)^2}{t^2}
    \right)^{-\left(\beta+\frac{n}{2}\left(\frac{1}{p}-\frac{1}{q}\right)\right)}
\|f\,\chi_E\|_{L^p(\mathbb{R}^n)}.
$$

\medskip

The Heat and the Poisson semigroups satisfy respectively off-diagonal estimates of exponential and polynomial type. Before making this precise, let us recall the definition of $p_-(L)$ and $p_+(L)$ in \eqref{p-}--\eqref{p+} and introduce two more parameters related to the gradient of the Heat semigroup. Let $(q_-(L),q_+(L))$ be the maximal open interval on which the gradient of the Heat semigroup, i.e.  $\{t\nabla_y e^{-t^2L}\}_{t>0}$, is uniformly bounded on $L^p(\mathbb{R}^n)$:
\begin{align*}
q_-(L) &:= \inf\left\{p\in(1,\infty): \sup_{t>0} \|t\nabla_y e^{-t^2L} \|_{L^p(\mathbb{R}^n)\rightarrow L^p(\mathbb{R}^n)}< \infty\right\},
\\[4pt]
q_+(L)& := \sup\left\{p\in(1,\infty) : \sup_{t>0} \|t\nabla_y e^{-t^2L} \|_{L^p(\mathbb{R}^n)\rightarrow L^p(\mathbb{R}^n)}< \infty\right\}.
\end{align*}
From \cite{Auscher} (see also \cite{AuscherMartell:II}) we know that $p_-(L)=1$ and  $p_+(L)=\infty$ if $n=1,2$; and if $n\ge 3$ then $p_-(L)<\frac{2\,n}{n+2}$ and $p_+(L)>\frac{2\,n}{n-2}$. Moreover, $q_-(L)=p_-(L)$, $ q_+(L)^*\le p_+(L)$, and we always have $q_+(L)>2$, with $q_+(L)=\infty$ if $n=1$.

The importance of these parameters stems from the fact that, besides giving the maximal intervals on which either the Heat semigroup or its gradient is uniformly bounded, they characterize the maximal open intervals on which off-diagonal estimates of exponential type hold (see \cite{Auscher} and \cite{AuscherMartell:II}). More precisely, for every $m\in \N_0$, there hold
$$
\{(t^2L)^me^{-t^2L}\}_{t>0}\in \mathcal{F}_\infty(L^p-L^q) \quad \textrm{ for all} \quad p_-(L)<p\leq q<p_+(L)$$
and
$$
\{t\nabla_ye^{-t^2L}\}_{t>0}\in \mathcal{F}_\infty(L^p-L^q) \quad \textrm{ for all} \quad q_-(L)<p\leq q<q_+(L).
$$
From these off-diagonal estimates we show that, for every $m\in \N_0$,
\begin{align*}
\{(t\sqrt{L}\,)^{2m}e^{-t\sqrt{L}}\}_{t>0}, \
\in \mathcal{F}_{m+\frac{1}{2}}(L^p\rightarrow L^q),
\end{align*}
for all $p_-(L)<p\leq q< p_+(L)$, and
\begin{align*}
 &\{t\nabla_{y}(t^2L)^me^{-t^2L}\}_{t>0}, \ \{t\nabla_{y,t}(t^2L)^me^{-t^2L}\}_{t>0}\in \mathcal{F}_\infty(L^p\rightarrow L^q),
\\[4pt]
&  \{t\nabla_{y}(t\sqrt{L}\,)^{2m}e^{-t\sqrt{L}}\}_{t>0}\in \mathcal{F}_{m+1}(L^p\rightarrow L^q), \,
\{t\nabla_{y,t}(t\sqrt{L}\,)^{2m}e^{-t\sqrt{L}}\}_{t>0}\in \mathcal{F}_{m+\frac{1}{2}}(L^p\rightarrow L^q),
\end{align*}
for all $q_-(L)<p\leq q< q_+(L)$.

\medskip

To show these off-diagonal estimates we shall apply the following Lemma, whose proof follows \textit{mutatis mutandis} that of \cite[Lemma 2.3]{HofmannMartell}.

\begin{lemma}\label{lemma:composition}
Let $\{P_t\}_{t>0}$ and $\{Q_t\}_{t>0}$ be two families of linear operators. Given $1\le p\le q\le \infty$, assume that
$\{P_t\}_{t>0}\in\mathcal{F}_{\infty}(L^p\rightarrow L^q)$ and $\{Q_t\}_{t>0}\in\mathcal{F}_{\infty}(L^p\rightarrow L^p)$. Then, for all closed sets $E$, $F$, all $f$, and all $t, s>0$ we have
$$
\big\|(P_t\circ Q_s) (f\,\chi_E)\,\chi_F\big\|_{L^q(\mathbb{R}^{n})}
\leq
C t^{-n\left(\frac{1}{p}-\frac{1}{q}\right)}e^{-c\frac{d(E,F)^2}{\max\{t,s\}^2}} \|f\,\chi_E\|_{L^{p}(\R^n)}.$$
\end{lemma}

\medskip

To prove our claims, let us first consider
$$
t\nabla_y(t^2L)^me^{-t^2L}=C\frac{t}{\sqrt{2}}\nabla_ye^{-\frac{t^2}{2}L}\circ
\left(\frac{t^2}{2}L\right)^me^{-\frac{t^2}{2}L}.
$$
Taking $P_t=\frac{t}{\sqrt{2}}\nabla_ye^{-\frac{t^2}{2}L}$ and $Q_t=\left(\frac{t^2}{2}L\right)^me^{-\frac{t^2}{2}L}$ \ for all $t>0$, since
$\{P_t\}_{t>0}\in \mathcal{F}_\infty(L^p-L^q)$ and $\{Q_t\}_{t>0}\in \mathcal{F}_\infty(L^p-L^p)$, for all $q_-(L)<p\leq q<q_+(L)$, we conclude from Lemma \ref{lemma:composition} that $\{t\nabla_y(t^2L)^me^{-t^2L}\}_{t>0}\in \mathcal{F}_\infty(L^p-L^q)$, for all $q_-(L)<p\leq q<q_+(L)$.

\medskip

To prove that $\{t\nabla_{y,t}(t^2L)^me^{-t^2L}\}_{t>0}\in \mathcal{F}_\infty(L^p\rightarrow L^q)$, for all $q_-(L)<p\leq q<q_+(L)$, we just need to observe that
\begin{align*}
|t\nabla_{y,t}(t^2L)^me^{-t^2L}f(y)|&\lesssim |t\nabla_{y}(t^2L)^me^{-t^2L}f(y)|+
|(t^2L)^me^{-t^2L}f(y)|+|(t^2L)^{m+1}e^{-t^2L}f(y)|,
\end{align*}
and apply the off-diagonal estimates satisfied by each term.

\medskip

We next obtain the off-diagonal estimates of polynomial type satisfied by the operators related to the Poisson semigroup.
Following some ideas used in \cite[Lemma $5$.$1$]{HofmannMayboroda}, we shall combine the subordination formula
\begin{align}\label{FR}
e^{-t\sqrt{L}}f(y)=C\int_0^{\infty}\frac{e^{-u}}{\sqrt{u}}e^{-\frac{t^2L}{4u}}f(y) \, du,
\end{align}
with Minkowski's inequality and the off-diagonal estimates satisfied by $\{(t^2L)^me^{-t^2L}\}_{t>0}$ and by $\{t\nabla_y(t^2L)^me^{-t^2L}\}_{t>0}$.

To obtain that $\{(t\sqrt{L}\,)^{2m}e^{-t\sqrt{L}}\}_{t>0}\in \mathcal{F}_{m+\frac{1}{2}}(L^p\rightarrow L^q)$ for all $p_-(L)<p\leq q<p_+(L)$, take two closed sets $E$ and $F$, a function $f$ supported in $E$, and $t>0$. Apply \eqref{FR}, Minkowski's inequality, the off-diagonal estimates satisfied by $\{(tL)^me^{-tL}\}_{t>0}$, and change the variable $u$ into $ \left(1+d(E,F)^2/t^2\right)^{-1} u$:
\begin{align*}
&
\left(\int_F |(t\sqrt{L}\,)^{2m}e^{-t\sqrt{L}}f(y)|^q \, dy\right)^{\frac{1}{q}}
\\
&
\quad
=
C \left(\int_F \left|(t\sqrt{L}\,)^{2m}\int_0^{\infty}\frac{e^{-u}}{\sqrt{u}}e^{-\frac{t^2}{4u}L}f(y) \, du\right|^q \, dy\right)^{\frac{1}{q}}\\
&
\quad
\lesssim
\int_0^{\infty}e^{-u}u^{m+\frac{1}{2}}\left(\int_F \left|\left(\frac{t^2}{4u}L\right)^{m}e^{-\frac{t^2}{4u}L}f(y)\right|^q \, dy\right)^{\frac{1}{q}} \ \frac{du}{u}\\
&
\quad
\lesssim \int_0^{\infty}e^{-cu\left(1+\frac{d(E,F)^2}{t^2}\right)}u^{m+\frac{1}{2}}\left(\frac{t}{\sqrt{u}}\right)^{-n\left(\frac{1}{p}-\frac{1}{q}\right)} \ \frac{du}{u} \left(\int_E |f(y)|^p \, dy\right)^{\frac{1}{p}}
\\
&
\quad  =
C t^{-n\left(\frac{1}{p}-\frac{1}{q}\right)} \left(1+\frac{d(E,F)^2}{t^2}\right)^{-\left(m+\frac{1}{2}+\frac{n}{2}\left(\frac{1}{p}-\frac{1}{q}\right)\right)} \int_0^{\infty}e^{-cu}u^{m+\frac{1}{2}+\frac{n}{2}\left(\frac{1}{p}-\frac{1}{q}\right)} \ \frac{du}{u} \left(\int_E |f(y)|^p \, dy\right)^{\frac{1}{p}}
\\
&
\quad  =
C\,
t^{-n\left(\frac{1}{p}-\frac{1}{q}\right)} \left(1+\frac{d(E,F)^2}{t^2}\right)^{-\left(m+\frac{1}{2}+\frac{n}{2}\left(\frac{1}{p}-\frac{1}{q}\right)\right)}
\left(\int_E |f(y)|^p \, dy\right)^{\frac{1}{p}},
\end{align*}
where in the last equality we have used that $m\ge 0$ and that $p\le q$.

We next show that $\{t\nabla_{y,t}(t\sqrt{L}\,)^{2m}e^{-t\sqrt{L}}\}_{t>0}\in \mathcal{F}_{m+\frac{1}{2}}(L^p\rightarrow L^q)$ for all $q_-(L)<p\leq q<q_+(L)$. Apply subordination formula \eqref{FR}, and Minkowski's inequality to obtain
\begin{align*}
&
\left(\int_{F}|t\nabla_{y,t}(t\sqrt{L}\,)^{2m} e^{-t\sqrt{L}}f(y)|^q \, dy\right)^{\frac{1}{q}}
\\
&
\qquad
= C
\left(\int_{F}\left|t\nabla_{y,t}(t\sqrt{L}\,)^{2m}\int_0^{\infty} \frac{e^{-u}}{\sqrt{u}}e^{-\frac{t^2}{4u}L}f(y) \, du\right|^q \, dy\right)^{\frac{1}{q}}
\\
&
\qquad
\le C
\int_0^{\infty}e^{-u}u^{m+1}\left(\int_{F}\left|\frac{t}{2\sqrt{u}}
\nabla_{y,t}\left(\frac{t}{2\sqrt{u}}\sqrt{L}\right)^{2m} e^{-\frac{t^2}{4u}L}f(y)\right|^q \, dy\right)^{\frac{1}{q}} \ \frac{du}{u}.
\end{align*}
Note now that
\begin{multline*}
\left|\frac{t}{2\sqrt{u}}
\nabla_{y,t}\left(\frac{t}{2\sqrt{u}}\sqrt{L}\right)^{2m} e^{-\frac{t^2}{4u}L}f(y)\right|\approx
\left|\frac{t}{2\sqrt{u}}
\nabla_{y}\left(\frac{t}{2\sqrt{u}}\sqrt{L}\right)^{2m} e^{-\frac{t^2}{4u}L}f(y)\right|
\\
+
\left|u^{-\frac{1}{2}}
\left(\frac{t}{2\sqrt{u}}\sqrt{L}\right)^{2m} e^{-\frac{t^2}{4u}L}f(y)\right|
+
\left|u^{-\frac{1}{2}}\left(\frac{t}{2\sqrt{u}}\sqrt{L}\right)^{2(m+1)} e^{-\frac{t^2}{4u}L}f(y)\right|.
\end{multline*}
Then,
applying that, for all $K\in \N_0$, $\{t\nabla_y(t\sqrt{L})^{2K}e^{-t^2L}\}_{t>0},\{(t^2L)^Ke^{-t^2L}\}_{t>0}\in \mathcal{F}_{\infty}(L^p\rightarrow L^q)$, we have
\begin{align*}
&
\left(\int_{F}|t\nabla_{y,t}(t\sqrt{L}\,)^{2m} e^{-t\sqrt{L}}f(y)|^q \, dy\right)^{\frac{1}{q}}
\\
&
\qquad
\lesssim
t^{-n\left(\frac{1}{p}-\frac{1}{q}\right)} \int_0^{\infty}\left(u^{m+1+\frac{n}{2}\left(\frac{1}{p}-\frac{1}{q}\right)}+u^{m+\frac{1}{2}+\frac{n}{2}\left(\frac{1}{p}-\frac{1}{q}\right)} \right)e^{-c\left(1+\frac{d(E,F)^2}{t^2}\right)u}\ \frac{du}{u} \ \|f\|_{L^p(E)}
\\
&
\qquad
\leq
C\,t^{-n\left(\frac{1}{p}-\frac{1}{q}\right)}
\left(1+\frac{d(E,F)^2}{t^2}\right)^{-\left(m+\frac{1}{2}+\frac{n}{2}\left(\frac{1}{p} -\frac{1}{q}\right)\right)}\|f\|_{L^p(E)}.
\end{align*}

Finally to show that $\{t\nabla_{y}(t\sqrt{L}\,)^{2m}e^{-t\sqrt{L}}\}_{t>0}\in \mathcal{F}_{m+1}(L^p\rightarrow L^q)$ we proceed as above. Applying that, for all $K\in \N_0$, $\{t\nabla_y(t\sqrt{L})^{2K}e^{-t^2L}\}_{t>0}\in \mathcal{F}_{\infty}(L^p\rightarrow L^q)$, we have
\begin{align*}
&
\left(\int_{F}|t\nabla_{y}(t\sqrt{L}\,)^{2m} e^{-t\sqrt{L}}f(y)|^q \, dy\right)^{\frac{1}{q}}
= C
\left(\int_{F}\left|t\nabla_{y}(t\sqrt{L}\,)^{2m}\int_0^{\infty} \frac{e^{-u}}{\sqrt{u}}e^{-\frac{t^2}{4u}L}f(y) \, du\right|^q \, dy\right)^{\frac{1}{q}}
\\
&
\qquad
\le C
\int_0^{\infty}e^{-u}u^{m+1}\left(\int_{F}\left|\frac{t}{2\sqrt{u}}
\nabla_{y}\left(\frac{t}{2\sqrt{u}}\sqrt{L}\right)^{2m} e^{-\frac{t^2}{4u}L}f(y)\right|^q \, dy\right)^{\frac{1}{q}} \ \frac{du}{u}
\\
&
\qquad\qquad
\lesssim
t^{-n\left(\frac{1}{p}-\frac{1}{q}\right)} \int_0^{\infty}u^{m+1+\frac{n}{2}\left(\frac{1}{p}-\frac{1}{q}\right)}e^{-c\left(1+\frac{d(E,F)^2}{t^2}\right)u}\ \frac{du}{u} \ \|f\|_{L^p(E)}
\\
&
\qquad\qquad\qquad
\leq
C\,t^{-n\left(\frac{1}{p}-\frac{1}{q}\right)}
\left(1+\frac{d(E,F)^2}{t^2}\right)^{-\left(m+1+\frac{n}{2}\left(\frac{1}{p} -\frac{1}{q}\right)\right)}\|f\|_{L^p(E)}.
\end{align*}

%%%%%%%%%%%%%%%%%%%%%%%%%%%%%%%%%%%
%%%%%%%%%%%%%%%%%%%%%%%%%%%%%%%%%%%
\section{Tent spaces}\label{section:TSpaces}
%%%%%%%%%%%%%%%%%%%%%%%%%%%%%%%%%%%
%%%%%%%%%%%%%%%%%%%%%%%%%%%%%%%%%%%

\medskip

We start with some definitions. Let $\R_+^{n+1}$ denote the upper-half space, that is, the set of points $(y,t)\in \R^n\times \R$ with $t>0$. Given $\alpha>0$ and $x\in \mathbb{R}^n$ we define the cone of aperture  $\alpha$ with vertex at $x$ by
$$
\Gamma^{\alpha}(x):=\{(y,t)\in \R_+^{n+1} : |x-y|<\alpha t\}.
$$
When $\alpha=1$ we simply write $\Gamma(x)$. For a closed set $E$ in $\mathbb{R}^n$, set
$$
\mathcal{R}^{\alpha}(E):=\bigcup_{x\in E}\Gamma^{\alpha}(x).
$$
When $\alpha=1$ we simplify the notation by writing $\mathcal{R}(E)$ instead of $\mathcal{R}^1(E)$.

We also define the operator $\mathcal{A}^{\alpha}$, $\alpha>0$, (and simply write $\mathcal{A}$ when $\alpha=1$) by
\begin{align}\label{AF}
\mathcal{A}^{\alpha}F(x):=\left(\iint_{\Gamma^{\alpha}(x)}|F(y,t)|^2 \ \frac{dy \, dt}{t^{n+1}}\right)^{\frac{1}{2}}.
\end{align}

\subsection{Change of angles}

Related to the above operators we obtain Proposition \ref{prop:alpha}, which is a weighted version of  \cite[Proposition $4$]{CoifmanMeyerStein} and \cite{Auscherangles}, see also \cite{Lerner}.

\begin{proposition}[Change of angles]\label{prop:alpha}
Let $0< \alpha\leq \beta<\infty$.
\begin{list}{$(\theenumi)$}{\usecounter{enumi}\leftmargin=1cm
\labelwidth=1cm\itemsep=0.2cm\topsep=.2cm
\renewcommand{\theenumi}{\roman{enumi}}}

\item  For every $w\in A_r$, $1\le r<\infty$, there holds
$$
\|\mathcal{A}^{\beta}F\|_{L^p(w)}\
\leq
C \left(\frac{\beta}{\alpha}\right)^{\frac{nr}{p}}
 \|\mathcal{A}^{\alpha} F\|_{L^p(w)} \quad \textrm{for all} \quad 0<p\leq 2r.
$$

\item  For every $w\in RH_{s'}$, $1\le s<\infty$, there holds
$$
\|\mathcal{A}^{\alpha} F\|_{L^p(w)}\leq  C\left(\frac{\alpha}{\beta}\right)^{\frac{n}{sp}} \|\mathcal{A}^{\beta} F\|_{L^p(w)}\quad \text{for all} \quad \frac{2}{s}\leq p<\infty.
$$
\end{list}
\end{proposition}

In Remark \ref{remark-sharp-coa} below we shall show that the previous estimates are sharp: the exponents $nr/p$ in $(i)$ and $n/sp$ in $(ii)$ cannot be improved. This should be compared with \cite{Auscherangles} where the unweighted case was considered (see also \cite{Lerner}).

To prove this proposition we need the following extrapolation result:

\begin{lemma}\label{lemma:extrapol}
Let $\mathcal{F}$ be a given family of pairs $(f,g)$ of non-negative and not identically zero measurable functions on $\R^n$.
\begin{list}{$(\theenumi)$}{\usecounter{enumi}\leftmargin=1cm
\labelwidth=1cm\itemsep=0.2cm\topsep=.2cm
\renewcommand{\theenumi}{\alph{enumi}}}

\item Suppose that for some fixed exponent $p_0$, $1\le p_0<\infty$, and every weight $w\in A_{p_0}$,
\begin{equation}\label{extrapol-Ap:hyp}
\int_{\mathbb{R}^{n}}f(x)^{p_0}\,w(x)\,dx
\leq
C_{w}
\int_{\mathbb{R}^{n}}g(x)^{p_0}\,w(x)\,dx,
\qquad\forall\,(f,g)\in\mathcal{F}.
\end{equation}
Then, for all $1<p<\infty$ and for all $w\in A_p$,
\begin{equation}\label{extrapol-Ap:con}
\int_{\mathbb{R}^{n}}f(x)^{p}\,w(x)\,dx
\leq
C_{w,p}
\int_{\mathbb{R}^{n}}g(x)^{p}\,w(x)\,dx,
\qquad\forall\,(f,g)\in\mathcal{F}.
\end{equation}

\item Suppose that for some fixed exponent $q_0$, $1\le q_0<\infty$, and every weight $w\in RH_{q_0'}$,
\begin{equation}\label{extrapol-RHq:hyp}
\int_{\mathbb{R}^{n}}f(x)^{\frac1{q_0}}\,w(x)\,dx
\leq
C_{w}
\int_{\mathbb{R}^{n}}g(x)^{\frac1{q_0}}\,w(x)\,dx,
\qquad\forall\,(f,g)\in\mathcal{F}.
\end{equation}
Then, for all $1<q<\infty$ and for all $w\in RH_{q'}$,
\begin{equation}\label{extrapol-RHq:con}
\int_{\mathbb{R}^{n}}f(x)^{\frac1q}\,w(x)\,dx
\leq
C_{w,q}
\int_{\mathbb{R}^{n}}g(x)^{\frac1q}\,w(x)\,dx,
\qquad\forall\,(f,g)\in\mathcal{F}.
\end{equation}
\end{list}
\end{lemma}

Part $(a)$ is the so-called  Rubio de Francia extrapolation theorem (cf. \cite{GC-extrapol, RdF}) written in terms of pairs of functions rather than in terms of boundedness of operators. The reader is referred to \cite{CruzMartellPerez} for a complete account of this topic. There is, however, a subtle difference between $(a)$ and \cite[Theorem 3.9]{CruzMartellPerez}: in the latter both the hypothesis and the conclusions are assumed to hold for all pairs $(f,g)\in\mathcal{F}$ for which the left-hand sides are finite. Here we do not make such assumptions and, in particular, we do have that the infiniteness of the left-hand side will imply that of the right-hand one. This formulation is more convenient for our purposes and its proof becomes a simple consequence of \cite[Theorem 3.9]{CruzMartellPerez}. The extrapolation result in $(b)$ is not written explicitly in \cite{CruzMartellPerez}, but can be easily obtained using \cite[Theorem 4.9]{AuscherMartell:I} and \cite[Theorem 3.31]{CruzMartellPerez} (see also \cite[Proposition 2.3]{AuscherHofmannMartell} for a particular case).

\begin{proof}
We start with $(a)$. Given a family $\mathcal{F}$ as in the statement and an arbitrary large number $N>0$ we consider the new family
$$
\mathcal{F}_N
:=
\big\{(f_N,g): (f,g)\in\mathcal{F}, f_N:=f\chi_{\{x\in B(0,N):f(x)\leq N\}}
\big\}.
$$
Note that
\begin{equation}\label{extrap-LHS-finite}
\int_{\mathbb{R}^n}f_N(x)^r w(x)dx \leq  N^r w(B(0,N))<\infty, \quad \textrm{for all}\quad 0< r<\infty \quad \textrm{and} \quad w\in A_{\infty}.
\end{equation}
From \eqref{extrapol-Ap:hyp} and the fact that $f_N\le f$, we clearly obtain that the same estimate holds for every pair in $\mathcal{F}_N$ (with a constant uniform on $N$) with a left-hand side that is always finite by \eqref{extrap-LHS-finite}. Thus we can apply \cite[Theorem 3.9]{CruzMartellPerez} to $\mathcal{F}_N$ to conclude that \eqref{extrapol-Ap:con} holds for all pairs $(f_N,g)\in\mathcal{F}_N$ (with a constant uniform on $N$), since again the left-hand side is always finite by \eqref{extrap-LHS-finite}. To complete the proof we just need to invoke the Monotone Convergence Theorem.

We next obtain $(b)$. Let us fix $1<q<\infty$ and $w\in RH_{q'}$. As before we first work with $\mathcal{F}_N$.
Since $w\in RH_{q'}\subset A_\infty$, there exists $p_0$ such that $w\in A_{p_0}$. We set $p_+:=2q$, $r_0:=\frac{2q}{q_0}$
and pick $0<p_-<\min\left\{\frac{2q}{q_0},\frac{2}{p_0},2\right\}$. We  then have that $0<p_-<r_0\leq p_+$, and for all $w_0\in A_{\frac{r_0}{p_-}}\cap RH_{\left(\frac{p_+}{r_0}\right)'}\subset
RH_{\left(\frac{p_+}{r_0}\right)'}=RH_{q_0'}$,
\begin{multline}
\int_{\mathbb{R}^n}
\left(f_N(x)^{\frac{1}{2q}}\right)^{r_0}w_0(x)dx
=\int_{\mathbb{R}^n}f_N(x)^{\frac{1}{q_0}}w_0(x)dx\leq\int_{\mathbb{R}^n}f(x)^{\frac{1}{q_0}}w_0(x)dx
\\
\leq C\int_{\mathbb{R}^n}g(x)^{\frac{1}{q_0}}w_0(x)dx
=
C\int_{\mathbb{R}^n}\left(g(x)^{\frac{1}{2q}}\right)^{r_0}w_0(x)dx,
\end{multline}
with $C$ independent of $N$, and for every pair $(f_N,g)\in\mathcal{F}_N$. Note that for each pair the left-hand side is finite by \eqref{extrap-LHS-finite}.
Therefore, applying \cite[Theorem 4.9]{AuscherMartell:I} or \cite[Theorem 3.31]{CruzMartellPerez}, we obtain, for all $p_-<p<p_+$ and for all $\widetilde{w}\in A_{\frac{p}{p_-}}\cap RH_{\left(\frac{p_+}{p}\right)'}$,
 \begin{align}\label{extralem1}
\int_{\mathbb{R}^n}f_N(x)^{\frac{p}{2q}}\widetilde{w} (x)dx\leq C\int_{\mathbb{R}^n}g(x)^{\frac{p}{2q}} \widetilde{w}(x)dx,
\end{align}
with $C$ independent of $N$,
for every pair $(f_N,g)\in\mathcal{F}_N$.
Then, note that $p=2$ satisfies that $p_-<p<p_+$ and  also $w\in A_{p_0}\cap RH_{q'}\subset A_{\frac{2}{p_-}}\cap RH_{\left(\frac{p_+}{2}\right)'}$. Thus, we can apply \eqref{extralem1} with $p=2$ and $\widetilde{w}=w$ to obtain
 \begin{align*}
\int_{\mathbb{R}^n}f_N(x)^{\frac{1}{q}} w(x)dx\leq C\int_{\mathbb{R}^n}g(x)^{\frac{1}{q}}w (x)dx,
\end{align*}
with $C$ independent of $N$. Letting $N\rightarrow \infty$, the Monotone Convergence Theorem yields the desired estimate \eqref{extrapol-RHq:con}.
\end{proof}

\medskip

Before proving Proposition \ref{prop:alpha}  let us recall the following property satisfied by Muckenhoupt weights.  Given $1\le p,q<\infty$,
for every ball $B$ and every measurable set $E\subset B$,
\begin{align}\label{pesosineq:Ap}
\frac{w(E)}{w(B)}\ge [w]_{A_{p}}^{-1}\left(\frac{|E|}{|B|}\right)^{p}, \qquad\forall\,w\in A_p,
\end{align}
and
\begin{align}\label{pesosineq:RHq}
\frac{w(E)}{w(B)}
\leq
[w]_{RH_{q'}}\left(\frac{|E|}{|B|}\right)^{\frac{1}{q}}, \qquad\forall\,w\in RH_{q'}.
\end{align}

\begin{proof}[Proof of Proposition \ref{prop:alpha}, part $(i)$]

We first observe that if $0<\alpha \le \beta<\infty$ then $\mathcal{A}^\beta F(x)= \mathcal{A}^{{\beta}/{\alpha}} \widetilde{F}$, where $\widetilde{F}(x,t)=\alpha^{\frac{n}{2}}F(x,t/\alpha)$. Thus, we can reduce matters to obtaining that for every $\alpha\ge 1$ and for every $w\in A_r$, $1\le r<\infty$, there holds
\begin{equation}\label{change-alph-1}
\|\mathcal{A}^{\alpha}F\|_{L^p(w)}\
\leq
C \alpha ^{\frac{nr}{p}}
 \|\mathcal{A}F\|_{L^p(w)}, \quad \textrm{for all} \quad 0<p\leq 2r.
\end{equation}

We then prove \eqref{change-alph-1} by splitting the proof into three steps. We first obtain the case $p=2$ and $1\leq r<\infty$. From this,  we extrapolate concluding the desired estimate in the ranges $0<p\le 2\,r$ and $1<r<\infty$. Finally, we shall consider the case $r=1$ and $0<p<2$.

Fix from now on $\alpha>1$. For the first step, let $p=2$ and $w\in A_{r_0}$, $1\leq r_0<\infty$.  From \eqref{pesosineq:Ap}, we easily obtain
\begin{align}\label{Change-p=2:Ar}
\|\mathcal{A}^{\alpha}F\|_{L^2(w)}&=\left(\int_{\mathbb{R}^n}\int_0^{\infty}\int_{|x-y|<\alpha t}|F(y,t)|^2 \frac{dy \, dt}{t^{n+1}}w(x) dx\right)^{\frac{1}{2}}
\\
\nonumber & =
\left(\int_{\mathbb{R}^n}\int_0^{\infty}|F(y,t)|^2 w(B(y,\alpha t)) \frac{dy \, dt}{t^{n+1}}\right)^{\frac{1}{2}}\nonumber
\\
\nonumber & \lesssim \alpha^{\frac{nr_0}{2}}
\left(\int_{\mathbb{R}^n}\int_0^{\infty}|F(y,t)|^2 w(B(y,t)) \frac{dy \, dt}{t^{n+1}}\right)^{\frac{1}{2}}
\\
\nonumber &=
\alpha^{\frac{nr_0}{2}}\left(\int_{\mathbb{R}^n}\int_0^{\infty}\int_{|x-y|< t}|F(y,t)|^2 \frac{dy \, dt}{t^{n+1}}w(x) dx\right)^{\frac{1}{2}}
\\
\nonumber &=\alpha^{\frac{nr_0}{2}}\|\mathcal{A}F\|_{L^2(w)}.
\end{align}

We shall extrapolate from this inequality. To set the stage, take an arbitrary $1\le r_0<\infty$ and consider $\mathcal{F}$ the family of pairs
$
(f,g)=\big((\mathcal{A}^{\alpha}F)^\frac{2}{r_0}, \alpha^n\,(\mathcal{A}F)^\frac{2}{r_0}\big)$. Notice that \eqref{Change-p=2:Ar} immediately gives that for every $w\in A_{r_0}$
$$
\int_{\R^n} f(x)^{r_0}\,w(x)\,dx
=
\int_{\R^n} \mathcal{A}^{\alpha}F(x)^2\,w(x)\,dx
\le
C\,\alpha^{n\,r_0}\,\int_{\R^n} \mathcal{A}F(x)^2\,w(x)\,dx
=
C\,\int_{\R^n} g(x)^{r_0}\,w(x)\,dx,
$$
where $C$ does not depend on $\alpha$. Next, we apply $(a)$ in Lemma \ref{lemma:extrapol} to conclude that for every $1<r<\infty$ and for every $w\in A_{r}$
$$
\int_{\R^n} \mathcal{A}^{\alpha}F(x)^{\frac{2\,r}{r_0}}\,w(x)\,dx
=
\int_{\R^n} f(x)^{r}\,w(x)\,dx
\le
C\,\int_{\R^n} g(x)^{r}\,w(x)\,dx
=
C\,\alpha^{n\,r}\,\int_{\R^n} \mathcal{A}F(x)^{\frac{2\,r}{r_0}}\,w(x)\,dx,
$$
where $C$ does not depend on $\alpha$. From this, using that $1\le r_0<\infty$ is arbitrary, we conclude \eqref{change-alph-1} under the restriction $1<r<\infty$.

To complete the proof it remains to consider the case $r=1$, (i.e., $w\in A_1$) and $0<p<2$. Notice that if $\|\mathcal{A}F\|_{L^p(w)}=\infty$ the inequality follows immediately. So, we can assume that $\|\mathcal{A}F\|_{L^p(w)}<\infty$.

For a fixed $\lambda >0$, set
\begin{align*}
E_{\lambda}:=\{x\in \mathbb{R}^n: \mathcal{A}F(x)\leq \lambda\}, \qquad  O_{\lambda}:=\R^n\backslash E_{\lambda}=\{x\in \mathbb{R}^n: \mathcal{A}F(x)> \lambda\}.
\end{align*}
Then, for each $0<\gamma<1$, we also consider the set of global $\gamma$-density with respect to $E_{\lambda}$ defined by
$$
E_{\lambda}^*:=\left\{x\in \mathbb{R}^n : \frac{|E_{\lambda}\cap B|}{|B|}\geq \gamma, \ \forall B \ \textrm{centered at} \ x \right\}
$$
and denote its complement by
\begin{align}\label{O}
O_{\lambda}^*
=\left\{x\in \mathbb{R}^n:\exists\, r>0 \ \textrm{such that} \ \frac{|O_{\lambda}\cap B(x,r)|}{|B(x,r)|}>1-\gamma\right\}
=
\left\{x\in \mathbb{R}^n:\mathcal{M}(\chi_{O_{\lambda}})(x)>1-\gamma\right\},
\end{align}
where $\mathcal{M}$ is the centered Hardy-Littlewood maximal operator.

Note that if $x_k\to x$ then $\chi_{\Gamma(x_k)}(y,t)\rightarrow \chi_{\Gamma(x)}(y,t)$ for a.e.~$(y,t)\in \R^{n+1}_+$. This and the Fatou Lemma clearly imply that $E_{\lambda}$ is closed. We next show that, for each $0<\gamma<1$, $E_{\lambda}^*$ is a nonempty closed set contained in $E_{\lambda}$.  Notice that
the fact that
$\mathcal{M}:L^1(w)\rightarrow L^{1,\infty}(w)$, since $w\in A_{1}$, and our earlier assumption ($\|\mathcal{A}F\|_{L^p(w)}<\infty$) give
$$
w(O_{\lambda}^*)=w\left(\left\{x\in \mathbb{R}^n:\mathcal{M}(\chi_{O_{\lambda}})(x)>1-\gamma\right\}\right)\lesssim
\frac{1}{1-\gamma}w(O_{\lambda})
\le
\frac{1}{(1-\gamma)\,\lambda^p}\,\|\mathcal{A}F\|_{L^p(w)}^p
<\infty.
$$
This immediately implies that $E_{\lambda}^*$ cannot be  empty.

Next, we see that $E_{\lambda}^*\subset E_{\lambda}$, for all $0<\gamma<1$. This follows from the fact that $E_{\lambda}$ is closed: if
$x\not \in E_{\lambda}$, there exists $r>0$ such that $B(x,r)\cap E_{\lambda}=\emptyset$, and then $x\notin E_{\lambda}^*$.

 \medskip

Finally, we show that $E_{\lambda}^*$ is closed. Let $\{x_k\}_k\subset E_{\lambda}^*$ be such that $x_k
\rightarrow  x$. Take an arbitrary $r>0$ and define the functions  $f_{k}=\chi_{E_{\lambda}\cap B(x_k,r)}$ which satisfy $f_{k}\rightarrow \chi_{E_{\lambda}\cap B(x,r)}$ a.e. in $\R^n$. Note also that for $k$ large enough $f_{k}\leq \chi_{B(x,2r)}$ (since $x_k\in B(x,r)$). Thus, by the Dominated Convergence Theorem, we conclude that
\begin{align*}
|E_{\lambda}\cap B(x,r)|
=
\lim_{k\rightarrow \infty}\int_{\mathbb{R}^n}f_{k}(y) \, dy
=
\lim_{k\rightarrow \infty}
|E_{\lambda}\cap B(x_k,r)|.
\end{align*}
On the other hand, since $x_k\in E_{\lambda}^*$ we have that $|E_{\lambda}\cap B(x_k,r)|\ge \gamma|B(x_k,r)|=\gamma|B(x,r)|$. This in turn implies that for every  $r>0$
$$
\frac{|E_{\lambda}\cap B(x,r)|}{|B(x,r)|}\geq \gamma,
$$
which yields that $x\in E_{\lambda}^*$ and hence $E_{\lambda}^*$ is closed.

After these preparations, given  $(y,t)\in \mathcal{R}^{\alpha}(E_{\lambda}^*)$, there exists $\bar{x}\in E_{\lambda}^*$ such that $|\bar{x}-y|<\alpha t$. Therefore, for $z=y-\frac{t}{2}\frac{y-\bar{x}}{|y-\bar{x}|}$ we have that $B\left(z,\frac{t}{2}\right)\subset B(\bar{x},\alpha t)\cap B(y,t)$ and
\begin{multline*}
\left|B(\bar{x},\alpha t)\setminus B(y,t)\right|
\leq
\left|B(\bar{x},\alpha t)\setminus B\left(z,\frac{t}{2}\right)\right|
=
\left|B(\bar{x},\alpha t)\right|-\left|B\left(z,\frac{t}{2}\right)\right|
\\
=
\left|B(\bar{x},\alpha t)\right|\left(1-\frac{1}{2^n\alpha^n}\right)
=
c_{\alpha}
\left|B(\bar{x},\alpha t)\right|,
\end{multline*}
with $c_{\alpha}=\left(1-\frac{1}{2^n\alpha^n}\right)<1$. This and the fact that $\bar{x}\in E_{\lambda}^*$ yield
\begin{multline*}
\gamma|B(\bar{x},\alpha t)| \le
|E_{\lambda}\cap B(\bar{x},\alpha t)|
= |E_{\lambda}\cap B(\bar{x},\alpha t)\setminus B(y,t)|+|E_{\lambda}\cap B(\bar{x},\alpha t)\cap B(y,t)|\\
\leq
c_{\alpha}
\left|B(\bar{x},\alpha t)\right| +|E_{\lambda}\cap B(y,t)|.
\end{multline*}
Choosing $\gamma=\frac{1+c_{\alpha}}{2}$ we conclude that
\begin{align}\label{densidad}
|E_{\lambda}\cap B(y,t)|\geq \frac{1}{2^{n+1}\alpha^n}|B(\bar{x},\alpha t)|
=
\frac{1}{2^{n+1}\alpha^n}|B(y,\alpha t)| .
\end{align}
From this and \eqref{pesosineq:Ap}, we have for every  $(y,t)\in \mathcal{R}^{\alpha}(E_{\lambda}^*)$,
\begin{align}\label{omega}
\frac{w(E_{\lambda}\cap B(y,t))}{w(B(y,\alpha t))}
\geq
[w]^{-1}_{A_1} \frac{|E_{\lambda}\cap B(y,t)|}{|B(y,\alpha t)|}
\geq
\frac{1}{2^{n+1}\alpha^n[w]_{A_1}}.
\end{align}
We use this to show that
\begin{align}\label{ineq}
\int_{E_{\lambda}^*} \mathcal{A}^{\alpha}F(x)^2 w(x) \, dx
&= \int_{E_{\lambda}^*}\int_{0}^{\infty}\int_{\mathbb{R}^n} |F(y,t)|^2 \chi_{B(0,1)}\left(\frac{x-y}{\alpha t}\right) w(x) \frac{dy \, dt}{t^{n+1}} \, dx
\\
\nonumber &\leq \iint_{\mathcal{R}^{\alpha}(E_{\lambda}^*)} |F(y,t)|^2 \int_{B(y,\alpha t)} w(x) \, dx \frac{dy \, dt}{t^{n+1}}
\\
\nonumber &
\leq 2^{n+1}\alpha^n[w]_{A_1}\iint_{\mathcal{R}^{\alpha}(E_{\lambda}^*)} |F(y,t)|^2 \int_{B(y, t)\cap E_{\lambda}} w(x) \, dx \frac{dy \, dt}{t^{n+1}}
\\
\nonumber &
\leq 2^{n+1}\alpha^n[w]_{A_1}\int_{E_{\lambda}} \mathcal{A}F(x)^2 w(x) \, dx.
\end{align}
 Therefore, from \eqref{ineq}, \eqref{O}, and the fact that $\mathcal{M}:L^{1}(w)\rightarrow L^{1,\infty}(w)$ (because $w \in A_{1}$), we obtain
\begin{align*}
w(\{x: \mathcal{A}^{\alpha}F(x)> \lambda\})&\leq w(\{x\in O_{\lambda}^*:\mathcal{A}^{\alpha}F(x)> \lambda\})+w(\{x\in E_{\lambda}^*: \mathcal{A}^{\alpha}F(x)> \lambda\})
\\
&
\leq w(\{x: \mathcal{M}(\chi_{O_{\lambda}})(x)> 1-\gamma\}) + \frac{1}{\lambda^2}\int_{E_{\lambda}^*}\mathcal{A}^{\alpha}F(x)^{2} w(x) \, dx
\\
&
\lesssim \alpha^n [w]_{A_1}w(O_{\lambda}) + \alpha^n [w]_{A_1} \frac{1}{\lambda^2}\int_{E_{\lambda}}\mathcal{A}F(x)^2 w(x) \, dx
\\
&
=
\alpha^n [w]_{A_1}w(\{x: \mathcal{A}F(x)>\lambda\}) +\alpha^n [w]_{A_1}\frac{1}{\lambda^2}\int_{E_{\lambda}}\mathcal{A}F(x)^2 w(x) \, dx.
\end{align*}
Using this and that $0<p<2$ it follows that
\begin{align*}
&\|\mathcal{A}^{\alpha}F\|_{L^p(w)}^p
=
\int_{0}^{\infty}p\,\lambda^{p} \, w(\{x: \mathcal{A}^{\alpha}F(x)> \lambda\})\, \frac{d\lambda}{\lambda}
\\
&
\quad\lesssim
\alpha^n [w]_{A_1}\bigg(\int_{0}^{\infty}p\,\lambda^{p} \, w(\{x: \mathcal{A}F(x)> \lambda\})\, \frac{d\lambda}{\lambda}
+
\int_{0}^{\infty}p\lambda^{p-2}\int_{E_{\lambda}} \mathcal{A}F(x)^2w(x) \, dx \, \frac{d\lambda}{\lambda}\bigg)
\\
&
\quad\leq
\alpha^n [w]_{A_1}
\bigg(
\|\mathcal{A}F\|_{L^p(w)}^p
+
\int_{\mathbb{R}^n} \mathcal{A}F(x)^2\int_{\mathcal{A}F(x)}^{\infty}p\lambda^{p-2} \, \frac{d\lambda}{\lambda} w(x) \, dx\bigg)\\
&
\quad=
C\,\alpha^n[w]_{A_1} \|\mathcal{A}F\|_{L^p(w)}^p.
\end{align*}
This completes the proof of $(i)$.
\end{proof}

\begin{proof}[Proof of Proposition \ref{prop:alpha}, part $(ii)$]
As before, we can reduce matters to showing that for every $\alpha\ge 1$ and for every $w\in RH_{s'}$, $1\le s<\infty$, there holds
\begin{equation}\label{change-alph-2}
\|\mathcal{A}F\|_{L^p(w)}\
\leq
C \alpha ^{-\frac{n}{sp}}
 \|\mathcal{A}^\alpha F\|_{L^p(w)}, \quad \textrm{for all} \quad \frac{2}{s}\leq p<\infty.
\end{equation}
We show this estimate considering three cases: $p=2$ and $1\leq s<\infty$,
$2/s\leq p<\infty$ and $1<s<\infty$, and  $s=1$ and $2<p<\infty$.

We start by taking $p=2$ and $w\in RH_{s_0'}$ with $1\leq s_0<\infty$. We proceed as in \eqref{Change-p=2:Ar} and use \eqref{pesosineq:RHq}  to obtain
\begin{multline}\label{Change-p=2:RHs}
\|\mathcal{A}F\|_{L^2(w)}
=
\left(\int_{\mathbb{R}^n}\int_0^{\infty}|F(y,t)|^2 w(B(y, t)) \frac{dy \, dt}{t^{n+1}}\right)^{\frac{1}{2}}
\\
\lesssim \alpha^{-\frac{n}{2s_0}}
\left(\int_{\mathbb{R}^n}\int_0^{\infty}|F(y,t)|^2 w(B(y,\alpha t)) \frac{dy \, dt}{t^{n+1}}\right)^{\frac{1}{2}}
=\alpha^{-\frac{n}{2s_0}}\|\mathcal{A}^{\alpha} F\|_{L^2(w)}.
\end{multline}

For the second case we shall extrapolate from \eqref{Change-p=2:RHs}. Take an arbitrary $1\le s_0<\infty$ and consider $\mathcal{F}$ the family of pairs
$
(f,g)=\big((\mathcal{A}F)^{2\,s_0}, \alpha^{-n}\,(\mathcal{A}^{\alpha} F)^{2\,s_0}\big)$. Notice that \eqref{Change-p=2:RHs} immediately gives that, for every $w\in RH_{s_0'}$,
$$
\int_{\R^n} f(x)^\frac1{s_0}\,w(x)\,dx
=
\int_{\R^n} \mathcal{A}F(x)^2\,w(x)\,dx
\le
C\,\alpha^{-\frac{n}{s_0}}\,\int_{\R^n} \mathcal{A}^\alpha F(x)^2\,w(x)\,dx
=
C\,\int_{\R^n} g(x)^\frac1{s_0}\,w(x)\,dx,
$$
where $C$ does not depend on $\alpha$. Next, we apply $(b)$ in Lemma \ref{lemma:extrapol} to conclude that, for every $1<s<\infty$ and for every $w\in RH_{s'}$,
$$
\int_{\R^n} \mathcal{A}F(x)^{\frac{2\,s_0}{s}}\,w(x)\,dx
=
\int_{\R^n} f(x)^{\frac1{s}}\,w(x)\,dx
\le
C\,\int_{\R^n} g(x)^{\frac1{s}}\,w(x)\,dx
=
C\,\alpha^{-\frac{n}{s}}\,\int_{\R^n} \mathcal{A}^{\alpha}  F(x)^{\frac{2\,s_0}{s}}\,w(x)\,dx,
$$
where $C$ does not depend on $\alpha$. From this, using that $1\le s_0<\infty$ is arbitrary we conclude \eqref{change-alph-2} under the restriction $1<s<\infty$.

Finally, we show \eqref{change-alph-2} for all $2<p<\infty$ and $w\in RH_{\infty}$ (i.e., $s=1$).
Without loss of generality, we may assume that $\alpha>32$ (for $1\le \alpha\le 32$ we just use that $\mathcal{A}F\le \mathcal{A}^\alpha F$). Let us also assume that $\|\mathcal{A}^\alpha F\|_{L^p(w)}<\infty$. Otherwise, there is nothing to prove. Besides, since $w\in RH_{\infty} $ there exists $r>1$, which can be assumed to satisfy $r\ge p/2$, such that $w\in A_{r}$. Then we can apply part $(i)$ with $\beta=6\sqrt{n}\alpha$ and obtain that
\begin{equation}\label{A-A}
\|\mathcal{A}^{6\sqrt{n}\alpha}F\|_{L^p(w)}
\le
C
\left(\frac{6\sqrt{n}\alpha}{\alpha}\right)^{\frac{n\,r}{p}}
\|\mathcal{A}^{\alpha}F\|_{L^p(w)}
=
C\|\mathcal{A}^{\alpha}F\|_{L^p(w)}<\infty,
\end{equation}
where $C$ does not depend on $\alpha$.

After these observations, for every $\lambda>0$, consider the set
$$
O_{\lambda}:=\{x\in \mathbb{R}^n:\mathcal{A}^{6\sqrt{n}\alpha}F(x)>\lambda\}.
$$
We shall show that
\begin{align}\label{AF-Awe}
w(\{x\in \mathbb{R}^n : \mathcal{A}F(x)>2\lambda\})
\lesssim
\frac{\alpha^{-n}}{\lambda^2}\int_{O_{\lambda}}|\mathcal{A}^{6\sqrt{n}\alpha}F(x)|^2w(x)dx.
\end{align}
Note that the previous estimate is trivial when $O_\lambda=\emptyset$: both sides vanish since $\mathcal{A}F\leq \mathcal{A}^{6\sqrt{n}\alpha}F$. We may then assume that  $O_\lambda\neq\emptyset$. From the arguments in the proof of $(i)$ we clearly have that $O_\lambda$ is open. Also \eqref{A-A} and Chebychev's inequality give that $w(O_\lambda)<\infty$, which in turn yields that $O_\lambda\subsetneq \R^n$.  We can then take a Whitney decomposition of $O_{\lambda}$ (cf. \cite[Chapter VI]{St70}): there exists a family of  closed cubes  $\{Q_j\}_{j\in \N}$ with disjoint interiors so that
\begin{equation}\label{Whitney}
O_{\lambda}
=
\bigcup_{j\in \N}Q_j,  \qquad \textrm{diam}(Q_j)\leq d(Q_j,\R^n\setminus O_{\lambda})\leq 4 \textrm{diam}(Q_j),
\qquad
\sum_j \chi_{Q_j^*}\le 12^n\,\chi_{O_\lambda},
\end{equation}
where $Q_j^*:=\frac98Q_j$.

On the other hand, since $\mathcal{A}F\leq \mathcal{A}^{6\sqrt{n}\alpha}F$, we have that
\begin{align}\label{rhwhi}
w(\{x\in \mathbb{R}^n : \mathcal{A}F(x)>2\lambda\})
= w(\{x\in O_\lambda: \mathcal{A}F(x)>2\lambda\})
=\sum_{j\in \N}
w(\{x\in Q_j : \mathcal{A}F(x)>2\lambda\}).
\end{align}
Fix $j\in\N$ and, for every $x\in Q_j$, write
$$
\mathcal{A}F(x)
\le
G_j(x)+H_j(x)
:=
\left(\int_{\frac{\ell(Q_j)}{\alpha}}^{\infty}\int_{B(x,t)}|F(y,t)|^2\frac{dy \, dt}{t^{n+1}}\right)^{\frac{1}{2}}
+
\left(\int_{0}^{\frac{\ell(Q_j)}{\alpha}}\int_{B(x,t)}|F(y,t)|^2\frac{dy \, dt}{t^{n+1}}\right)^{\frac{1}{2}}.
$$
Pick $x_j\in \mathbb{R}^n\setminus O_{\lambda}$ such that $d(x_j,Q_j)\leq 4 \textrm{diam}(Q_j)$. Notice that for every $x\in Q_{j}$ and $t\ge \ell(Q_j)/\alpha$ we have that
$B(x,t)\subset B(x_j,6\sqrt{n}\alpha t)$.
Then,
\begin{align*}
G_j(x)^2=\int_{\frac{\ell(Q_j)}{\alpha}}^{\infty}\int_{B(x,t)}|F(y,t)|^2\frac{dy \, dt}{t^{n+1}}
\leq
\int_{\frac{\ell(Q_j)}{\alpha}}^{\infty}\int_{B(x_j,6\sqrt{n}\alpha t)}|F(y,t)|^2\frac{dy \, dt}{t^{n+1}}
\leq \mathcal{A}^{6\sqrt{n}\alpha}F(x_j)^2\leq \lambda^2,
\end{align*}
where we have used that $x_j\in \mathbb{R}^n\setminus O_{\lambda}$ in the last inequality. Using this and that $w\in RH_{\infty}$, we have
\begin{align*}
w(\{x\in Q_j : \mathcal{A}F(x)>2\lambda\})
&\le
w(\{x\in Q_j : H_j(x)>\lambda\})
\\
&
\leq
\frac{1}{\lambda^2}\int_{Q_j} H_j(x)^2w(x)dx\\
%&=
%\frac{1}{\lambda^2}\int_{Q_j}\int_0^{\frac{\ell(Q_j)}{\alpha}}\int_{B(x,t)}|F(y,t)|^2\frac{dy \, dt}{t^{n+1}}w(x)dx
%\\
&\leq
\frac{1}{\lambda^2} \iint_{\mathcal{R}(Q_j)}\chi_{(0,\alpha^{-1}\ell(Q_j))}(t)|F(y,t)|^2 w(B(y,t))\frac{dy \, dt}{t^{n+1}}
\\
&\lesssim
\frac{\alpha^{-n}}{\lambda^2}\iint_{\mathcal{R}(Q_j)}\chi_{(0,\alpha^{-1}\ell(Q_j))}(t)|F(y,t)|^2 w(B(y,32^{-1}\alpha t))\frac{dy \, dt}{t^{n+1}}
\\
&\leq
\frac{\alpha^{-n}}{\lambda^2}\int_{Q_j^*}\int_0^{\infty}\int_{B(x,32^{-1}\alpha t)}|F(y,t)|^2 \frac{dy \, dt}{t^{n+1}}w(x)dx
\\
&\leq
%\frac{\alpha^{-n}}{\lambda^2}\int_{Q_j^*}\int_0^{\infty}\int_{B(x,6\sqrt{n}\alpha t)}|F(y,t)|^2 \frac{dy \, dt}{t^{n+1}}w(x)dx
%\\
%&=
\frac{\alpha^{-n}}{\lambda^2}\int_{Q_j^*}\mathcal{A}^{6\sqrt{n}\alpha}F(x)^2w(x)dx.
\end{align*}
Then, by  \eqref{rhwhi} and the bounded overlap of the family $\{Q_j^*\}_{j\in\N}$, we conclude \eqref{AF-Awe}:
\begin{align*}
w(\{x\in \mathbb{R}^n : \mathcal{A}F(x)>2\lambda\})
\lesssim
\frac{\alpha^{-n}}{\lambda^2}\sum_{j\in \N}\int_{Q_j^*}|\mathcal{A}^{6\sqrt{n}\alpha}F(x)|^2w(x)dx
\lesssim
\frac{\alpha^{-n}}{\lambda^2}\int_{O_{\lambda}}|\mathcal{A}^{6\sqrt{n}\alpha}F(x)|^2w(x)dx.
\end{align*}
This, the fact that $2<p<\infty$, and \eqref{A-A} give
\begin{multline*}
\|\mathcal{A}F\|_{L^p(w)}^p
=
2^p\int_{0}^{\infty}p\,\lambda^{p} \, w\left\{x: \mathcal{A}F(x)> 2\lambda\right\}\, \frac{d\lambda}{\lambda}
\lesssim
\alpha^{-n}\int_{0}^{\infty}\lambda^{p-2}
\int_{O_{\lambda}}\big(\mathcal{A}^{6\sqrt{n}\alpha}F(x)\big)^2w(x)dx\frac{d\lambda}{\lambda}
\\
\lesssim \alpha^{-n}\int_{\mathbb{R}^n}\big(\mathcal{A}^{6\sqrt{n}\alpha}F(x)\big)^2\int_{0}^{\mathcal{A}^{6\sqrt{n}\alpha}F(x)}\lambda^{p-2}
\frac{d\lambda}{\lambda}w(x)dx
\lesssim \alpha^{-n}\|\mathcal{A}^{6\sqrt{n}\alpha}F\|_{L^p(w)}^p
\lesssim \alpha^{-n}\|\mathcal{A}^{\alpha}F\|_{L^p(w)}^p
.
\end{multline*}
This completes the proof.
\end{proof}

As announced before, we next discuss the sharpness of Proposition \ref{prop:alpha}.
\begin{remark}\label{remark-sharp-coa}
Let us consider the weights $w_{\theta}(x)=|x|^{-\theta}$. It is standard to show that $w_\theta \in A_r$ if and only if $-n(r-1)<\theta<n$ (with the possibility of taking $\theta=0$ when $r=1$). Besides, $w_{\theta}\in RH_{s'}$ if and only if $-\infty<\theta<\frac{n}{s'}$ (with the possibility of taking $\theta=0$ when $s=1$). We shall use this family of weights to show that the exponents obtained in Proposition \ref{prop:alpha} parts $(i)$ and $(ii)$ are sharp.

We proceed as in \cite{Auscherangles}, where the unweighted case was considered. Set $B:=B(0,\frac{1}{4})$ and $a(y,t):=\chi_{B}(y)\chi_{[\frac{1}{2},1]}(t)$. It is straightforward to show that
$$
\mathcal{A} a(x)\leq C \chi_{5B}(x),\quad \forall\, x\in \R^n ,
\qquad \textrm{and}\qquad
\mathcal{A} a(x)\geq C,  \quad \forall\, x\in B,
$$
and, for every $\alpha\ge 1$,
$$
\mathcal{A}^{\alpha}a(x)\leq C \chi_{(4\alpha+1)B}(x), \quad \forall\, x\in \R^n ,
\qquad \textrm{and}\qquad  \mathcal{A}^{\alpha} a(x)\geq C,\quad  \forall\, x\in (2\alpha-1)B.
$$
Hence,
\begin{equation}\label{cog-sharp:a}
\|\mathcal{A} a\|_{L^p(w_{\theta})}\approx 1 \qquad\quad \textrm{and} \qquad\quad
 \|\mathcal{A}^{\alpha} a\|_{L^p(w_{\theta})}\approx \alpha^{\frac{n-\theta}{p}},
\end{equation}
where the implicit constants may depend on $\theta$ but are independent of $\alpha$.

To see that the exponent in part $(i)$ is sharp, assume by way of contradiction, that there exists $0<\varrho<\frac{nr}{p}$ such that for all $\alpha\ge 1$, $w\in A_r$, $1\le r<\infty$, and $0<p\le 2r$ there holds
\begin{equation}\label{cog-sharp1}
\|\mathcal{A}^{\alpha} F\|_{L^p(w)}\leq C_{w} \alpha^{\frac{nr}{p}-\varrho}\|\mathcal{A}F\|_{L^p(w)}.
\end{equation}
Take $1\le r<\infty$ $0<p\le 2r$, and set $\theta:=-n(r-1)+\frac{\varrho p}{2}$. Note that $-n(r-1)<\theta<n$ and therefore $w_\theta\in A_r$. Applying \eqref{cog-sharp:a} and \eqref{cog-sharp1}, there exists $C_\theta$ so that for every $\alpha>1$ there holds
$$
\alpha^{\frac{nr}{p}-\frac{\varrho}2}
=
\alpha^{\frac{n-\theta}{p}}\approx
\|\mathcal{A}^{\alpha} a\|_{L^p(w_{\theta})}
\leq
C_{\theta} \alpha^{\frac{nr}{p}-\varrho}\|\mathcal{A} a\|_{L^p(w_{\theta})}
\approx
C_\theta\,\alpha^{\frac{nr}{p}-\varrho},
$$
where the implicit constants may depend on $\theta$ but are independent of $\alpha$. This clearly leads to a contradiction since $\alpha^{\frac{nr}{p}-\frac{\varrho}2}\gg \alpha^{\frac{nr}{p}-\varrho}$ when $\alpha\to\infty$.

\medskip

We next see that the exponent in part $(ii)$ is sharp. Again we proceed by way of contradiction: let us assume that there exists $\varrho>0$ such that
for all $\alpha\ge 1$, $w\in RH_{s'}$, $1\le s<\infty$, and $\frac2s\le p<\infty$ there holds
\begin{equation}\label{cog-sharp2}
\|\mathcal{A}F\|_{L^p(w)}\leq C_{w} \alpha^{-\frac{n}{sp}-\varrho}\|\mathcal{A}^\alpha F\|_{L^p(w)}.
\end{equation}
Take $1\le s<\infty$, $\frac2s\le p<\infty$, and pick $\theta:=\frac{n}{s'}-\frac{\varrho p}{2}$. Observe that $-\infty <\theta<\frac{n}{s'}$ and therefore $w_\theta\in RH_{s'}$. Applying \eqref{cog-sharp:a} and \eqref{cog-sharp2}, there exists $C_\theta$ so that for every $\alpha>1$ there holds
$$
1
\approx
\|\mathcal{A} a\|_{L^p(w_{\theta})}
\leq
C_{\theta} \alpha^{-\frac{n}{sp}-\varrho}\|\mathcal{A}^{\alpha} a\|_{L^p(w_{\theta})}
\approx
C_\theta\, \alpha^{-\frac{n}{sp}-\varrho+\frac{n-\theta}{p}}
=
C_\theta\, \alpha^{-\frac{\varrho}{2}},
$$
where the implicit constants may depend on $\theta$ but are independent of $\alpha$. Note that the right-hand side tends to $0$ as $\alpha\to\infty$ and this readily leads to a contradiction.
\end{remark}

Proposition \ref{prop:alpha} gives us a way to compare the norms of $\mathcal{A}^\alpha F$ in $L^p(w)$ for different angles $\alpha$. In that result, the emphasis is on the class of weights: fixed a class of weights ($A_r$ in $(a)$ or $RH_{s'}$ in $(b)$), we estimate the change of angles in $L^p(w)$ for some range of $p$'s. In some other situations it may be interesting to give formulas where the emphasis is on the exponent $p$. This is contained in the following result whose elementary proof  follows from Proposition \ref{prop:alpha} and  is left to the interested reader:

\begin{proposition}
Let $w\in A_{\infty}$, $0< \alpha\leq \beta<\infty$ and $0<p<\infty$. There hold:
\begin{list}{$(\theenumi)$}{\usecounter{enumi}\leftmargin=1cm
\labelwidth=1cm\itemsep=0.2cm\topsep=.2cm
\renewcommand{\theenumi}{\roman{enumi}}}

\item
 $\displaystyle \|\mathcal{A}^{\beta}F\|_{L^p(w)}\leq C\left(\frac{\beta}{\alpha}\right)^{\frac{nr}{p}}
 \|\mathcal{A}^{\alpha} F\|_{L^p(w)}$, for $r>\max\{\frac{p}{2},r_w\}$, and for $r=\max\{\frac{p}{2},r_w\}$ if $r_w<\frac{p}{2}$ or $w\in A_1$.

\item $\displaystyle \|\mathcal{A}^{\alpha} F\|_{L^p(w)}\leq
C\left(\frac{\alpha}{\beta}\right)^{\frac{n}{sp}}\|\mathcal{A}^{\beta}F\|_{L^p(w)}$, for $\frac{1}{s}<\min\{\frac{p}{2},\frac{1}{s_w}\}$, and for $\frac{1}{s}=\min\{\frac{p}{2},\frac{1}{s_w}\}$ if $\frac{p}{2}<\frac{1}{s_w}$ or $w\in RH_{\infty}$.

\end{list}

\end{proposition}

\bigskip

%%%%%%%%%%%%%%%%%%%%%%%%%%%%%%%%%%%%%%%%%%%%%%%%%%%%%%%%%%%%%%%%%%%%%%%%%%%%%%%%%
%%%%%%%%%%%%%%%%%%%%%%%%%%%%%%%%%%%%%%%%%%%%%%%%%%%%%%%%%%%%%%%%%%%%%%%%%%%%%%%%%
%%%%%%%%%%%%%%%%%%%%%%%%%%%%%%%%%%%%%%%%%%%%%%%%%%%%%%%%%%%%%%%%%%%%%%%%%%%%%%%%%

Related to the change of angles, we establish the following result, which will be required in the proof of Theorem \ref{theor:control-SF-Poisson}.

\medskip

\begin{proposition}\label{prop:Q}
Let  $1\le q \leq s<\infty$, $w\in RH_{s'}$, and $0\leq \alpha\leq 1$. Then, for every $t>0$, we have
\begin{align}\label{G-alpha}
\int_{\mathbb{R}^n}\left(\int_{B(x,\alpha t)}|h(y,t)| \, dy \right)^{\frac{1}{q}}w(x)dx\lesssim \alpha^{\frac{n}{s}}
\int_{\mathbb{R}^n}\left(\int_{B(x, t)}|h(y,t)| \, dy \right)^{\frac{1}{q}}w(x)dx.
\end{align}
\end{proposition}

\begin{proof}
We fix $t>0$, $0< \alpha \le 1$, and $1\le q<\infty$. Set
$$
G^{\alpha}(x,t):=\left(\int_{B(x,\alpha t)}|h(y,t)| \, dy\right)^{\frac{1}{q}}.
$$
For $\alpha=1$, we simply write $G(x,t)$.
Then, from \eqref{pesosineq:RHq}, for all  $1\le s_0<\infty$ and $w\in RH_{s_0'}$, we have
\begin{multline}\label{G-alpha:extr}
\int_{\mathbb{R}^n}G^{\alpha}(x,t)^{q}w(x)dx
=
\int_{\mathbb{R}^n}|h(y,t)|\, w(B(y,\alpha t))\,dy
\\
\lesssim \alpha^{\frac{n}{s_0}}
\int_{\mathbb{R}^n}|h(y,t)|\,w(B(y,t))\, dy
=
\alpha^{\frac{n}{s_0}}
\int_{\mathbb{R}^n}G(x,t)^{q} \ w(x)dx.
\end{multline}
This gives \eqref{G-alpha} for $q=1$, and thus we may assume that $q>1$.
We shall extrapolate from \eqref{G-alpha:extr}. Take an arbitrary $1\le s_0<\infty$ and consider $\mathcal{F}$ the family of pairs
$
(f,g)=\big( G^\alpha(\cdot,t)^{q\,s_0}, \alpha^{n}\,G(\cdot,t) ^{q\,s_0}\big)$. Notice that \eqref{G-alpha:extr} immediately gives that, for every $w\in RH_{s_0'}$,
$$
\int_{\R^n} f(x)^\frac1{s_0}\,w(x)\,dx
=
\int_{\R^n} G^{\alpha}(x,t)^{q} \,w(x)\,dx
\le
C\,\alpha^{\frac{n}{s_0}}\,\int_{\R^n} G(x,t)^{q} \,w(x)\,dx
=
C\,\int_{\R^n} g(x)^\frac1{s_0}\,w(x)\,dx,
$$
where $C$ does not depend on $\alpha$. Next, we apply $(b)$ in Lemma \ref{lemma:extrapol} to conclude that, for every $1<s<\infty$ and for every $w\in RH_{s'}$,
$$
\int_{\R^n} G^{\alpha}(x,t)^{\frac{q\,s_0}{s}}\,w(x)\,dx
=
\int_{\R^n} f(x)^{\frac1{s}}\,w(x)\,dx
\le
C\,\int_{\R^n} g(x)^{\frac1{s}}\,w(x)\,dx
=
C\,\alpha^{\frac{n}{s}}\,\int_{\R^n} G(x,t)^{\frac{q\,s_0}{s}}\,w(x)\,dx,
$$
where $C$ does not depend on $\alpha$. From this, if $1<q\le s<\infty$ we can take $s_0=s/q$ and conclude \eqref{G-alpha} as desired.
\end{proof}

\subsection{A new version of the Carleson measure condition}
Let us recall the following maximal operator from \cite{CoifmanMeyerStein}
\begin{align*}
\mathcal{C}F(x):=\sup_{B\ni x}\left(\frac{1}{|B|}\int_0^{r_B}\int_{B}|F(y,t)|^2\frac{dy \, dt}{t}\right)^{\frac{1}{2}}.
\end{align*}
Recall that $\mathcal{C} F\in L^\infty(\R^n)$ means that $|F(y,t)|^2\frac{dy \, dt}{t}$ is a Carleson measure in $\R^{n+1}_+$.

Given $0<p<\infty$, we now introduce a new maximal operator
\begin{align}\label{CF}
\mathcal{C}_{p}F(x_0)=\sup_{B\ni x_0}\left(\frac{1}{|B|}\int_{B}\left(\int_{0}^{r_B}\int_{B(x,t)}|F(y,t)|^2 \frac{dy \, dt}{t^{n+1}}\right)^{\frac{p}{2}} \, dx \right)^{\frac{1}{p}},
\end{align}
where the supremum is taken over all balls $B\subset \mathbb{R}^n$  and where $r_B$ denotes the corresponding radius.

This operator is a version of $\mathcal{C}$ which will be very useful for our purposes. Indeed, for $p=2$, we shall see that $\mathcal{C} F \approx \mathcal{C}_2F$.
First, applying Fubini we have
\begin{multline*}
\mathcal{C}_2F(x_0)
=
\sup_{B\ni x_0}\left(\frac{1}{|B|}\int_{B}\int_{0}^{r_B}\int_{B(x,t)}|F(y,t)|^2 \frac{dy \, dt}{t^{n+1}}\, dx \right)^{\frac{1}{2}}
\\
\leq \sup_{B\ni x_0}\left(\frac{1}{|B|}\int_{2B}\int_{0}^{r_B}|F(y,t)|^2\left(\int_{B(y,t)} 1 \, dx\right) \frac{dy \, dt}{t^{n+1}} \right)^{\frac{1}{2}}
\\
\lesssim \sup_{B\ni x_0}\left(\frac{1}{|2B|}\int_0^{2r_B}\int_{2B}|F(y,t)|^2 \frac{dy \, dt}{t} \right)^{\frac{1}{2}}=\mathcal{C} F(x_0).
\end{multline*}
For the reverse inequality, there holds
\begin{multline*}
\mathcal{C} F(x_0)
=\sup_{B\ni x_0}\left(\frac{1}{|B|}\int_0^{r_B}\int_{B}|F(y,t)|^2 \frac{dy \, dt}{t} \right)^{\frac{1}{2}}
\\
\lesssim \sup_{B\ni x_0}\left(\frac{1}{|B|}\int_{0}^{r_B}\int_{B}|F(y,t)|^2\left(\int_{B(y,t)}1 \, dx\right) \frac{dy \, dt}{t^{n+1}} \right)^{\frac{1}{2}}
\\
\lesssim \sup_{B\ni x_0}\left(\frac{1}{|2B|}\int_{2B}\int_{0}^{2r_B}\int_{B(x,t)}|F(y,t)|^2\frac{dy \, dt}{t^{n+1}} \, dx \right)^{\frac{1}{2}}
=
\mathcal{C}_2F(x_0). \end{multline*}

Our next result shows how $\mathcal{C}_{p_0}$ and $\mathcal{A}$ compare to each other. The case $p_0=2$ and $w\equiv 1$ appears in \cite{CoifmanMeyerStein}
and as a result one sees that $\mathcal{A}F$ and $\mathcal{C}F\approx\mathcal{C}_2F$ are comparable in $L^p(\R^n)$ for every $2<p<\infty$. Our result gives comparability of $\mathcal{A}F$ and $\mathcal{C}_{p_0}F$ in the range $p_0<p<\infty$ and, in particular, if $p_0<2$ we can go below $p=2$.

\begin{proposition}\label{prop:maximal}\
\begin{list}{$(\theenumi)$}{\usecounter{enumi}\leftmargin=1cm \labelwidth=1cm\itemsep=0.2cm\topsep=.2cm \renewcommand{\theenumi}{\alph{enumi}}}
\item If\, $0<p_0,p<\infty$, $w \in A_{\infty}$ and $F\in L^2_{\rm loc}(\R^{n+1}_+)$ then
\begin{align*}
\|\mathcal{A}F\|_{L^p(w)}\lesssim \|\mathcal{C}_{p_0}F\|_{L^p(w)}.
\end{align*}

\item If\, $0< p_0< p<\infty$ and $w\in A_{\frac{p}{p_0}}$ then
\begin{align*}
\|\mathcal{C}_{p_0}F\|_{L^p(w)}\lesssim \|\mathcal{A}F\|_{L^p(w)}.
\end{align*}
\end{list}
\end{proposition}

\begin{proof}
The proof of $(a)$ uses a good-$\lambda$ argument. Then, it requires to know that the quantity to be hidden is a priori finite. To guarantee this we divide the proof into two steps. The first step consists in proving $(a)$ for all  $F\in L^2(\R^{n+1}_+)$ such that, for some $N>1$, $\supp F\subset K_N:=\chi_{B(0,N)}(y)\chi_{(N^{-1},N)}(t)$. In the second step we shall consider general functions $F\in L^2_{\rm loc}(\R^{n+1}_+)$ and define, $F_N:=F\chi_{K_N}$, $N\ge 1$. Clearly $F_N\in L^2(\R^{n+1}_+)$ and $\supp F_N\subset K_N$, and hence we can apply step $1$ to $F_N$. By a limiting argument we shall obtain the desired estimate for $F$.

\medskip

{\em Step $1$}: Take $F\in L^2(\R^{n+1}_+)$ such that, for some $N>1$, $\supp F\subset K_N$, and note that under this assumption
$\|\mathcal{A}F\|_{L^p(w)}<\infty$. Indeed, $\supp \mathcal{A}F\subset B(0,2N)$, and then
$$
\|\mathcal{A}F\|_{L^p(w)}\leq N^{\frac{n+1}{2}}\|F\|_{L^2(\R^{n+1}_+)}w(B(0,2N))^{\frac{1}{p}}<\infty.
$$

We claim that it is enough to prove that there exist $\alpha>1$ and a constant $c$ such that for all $0<\gamma\leq 1$ and $0<\lambda<\infty$ we have
\begin{align}\label{good-lambda:A-Cp:w}
w(\{x\in \mathbb{R}^n: \mathcal{A}F(x)>2\lambda,\mathcal{C}_{p_0}F(x)\leq \gamma\lambda\})\leq c\gamma^{c_{w}}w(\{x\in \mathbb{R}^n: \mathcal{A}^{\alpha}F(x)>\lambda\}).
\end{align}
Assuming this momentarily it follows that
\begin{align*}
&
w(\{x\in \mathbb{R}^n: \mathcal{A}F(x)>2\lambda\})
\\
&
\qquad \leq w(\{x\in \mathbb{R}^n: \mathcal{A}F(x)>2\lambda,\mathcal{C}_{p_0}F(x)\leq \gamma\lambda\})+ w(\{x\in \mathbb{R}^n: \mathcal{C}_{p_0}F(x)> \gamma\lambda\})
\\
&
\qquad\leq c\gamma^{c_w}w(\{x\in \mathbb{R}^n: \mathcal{A}^{\alpha}F(x)>\lambda\}) + w(\{x\in \mathbb{R}^n: \mathcal{C}_{p_0}F(x)> \gamma\lambda\}).
\end{align*}
This easily gives
\begin{align*}
\|\mathcal{A}F\|_{L^p(w)}^p\leq C_{\gamma,p} \|\mathcal{C}_{p_0}F\|_{L^p(w)}^p+c\gamma^{c_{w}}\|\mathcal{A}^{\alpha}F\|_{L^p(w)}^p.
\end{align*}
From Proposition \ref{prop:alpha} we know that $\|\mathcal{A}^{\alpha}F\|_{L^p(w)}\leq c(\alpha,p)\|\mathcal{A}F\|_{L^p(w)}$. Then, by choosing $\gamma$ small enough so that $c\gamma^{c_w}c(\alpha,p)^p<1$, and since $\|\mathcal{A}F\|_{L^p(w)}<\infty$, we easily conclude that
\begin{align*}
\|\mathcal{A}F\|_{L^p(w)}\lesssim \|\mathcal{C}_{p_0}F\|_{L^p(w)}.
\end{align*}

To complete the proof it remains to show  \eqref{good-lambda:A-Cp:w}. We argue as in \cite{CoifmanMeyerStein}. Write $O_\lambda=\{x\in \mathbb{R}^n:\mathcal{A}^{\alpha}F(x)>\lambda\}$. We may assume that $w(O_\lambda)<\infty$ (otherwise, there is nothing to prove) and this in turn implies that $O_\lambda \subsetneq \R^n$. Without loss of generality we can also suppose that $O_\lambda\neq\emptyset$ (otherwise, both terms in \eqref{good-lambda:A-Cp:w} vanish, since $\mathcal{A}^{\alpha}F\geq \mathcal{A}F$ because $\alpha>1$, and again the proof is trivial). Note finally that $O_\lambda$ is open, fact that can be proved much as in the proof of Proposition \ref{prop:alpha}. We can then take a Whitney decomposition  of $O_\lambda$ (cf. \cite[Chapter VI]{St70}): there exists a family of  closed cubes  $\{Q_j\}_{j\in \N}$ with disjoint interiors satisfying \eqref{Whitney}.
In particular, for each $j\in\N$ we can pick $x_j\in \R^n\setminus O_\lambda$ such that $d(x_j,Q_j)\leq 4 \textrm{diam}(Q_j)$. Furthermore, since $\alpha>1$ we have $\mathcal{A}^{\alpha}F\geq \mathcal{A}F$ and
\begin{multline*}
w(\{x\in \R^n: \mathcal{A}F(x)>2\lambda,\mathcal{C}_{p_0}F(x)\leq \gamma\lambda\})
=
w(\{x\in O_\lambda: \mathcal{A}F(x)>2\lambda,\mathcal{C}_{p_0}F(x)\leq \gamma\lambda\})
\\
=
\sum_{j\in\N} w(\{x\in Q_j: \mathcal{A}F(x)>2\lambda,\mathcal{C}_{p_0}F(x)\leq \gamma\lambda\}).
\end{multline*}
Thus, to show \eqref{good-lambda:A-Cp:w}, it is enough to prove
\begin{align}\label{good-lambda:A-Cp:dx}
|\{x\in Q_j: \mathcal{A}F(x)>2\lambda,\mathcal{C}_{p_0}F(x)\leq \gamma\lambda\}|\leq c\gamma^{p_0}|Q_j|,
\end{align}
which, together with $w \in A_{\infty}$ (cf. \eqref{pesosineq:RHq}), would imply
\begin{align*}
w(\{x\in Q_j: \mathcal{A}F(x)>2\lambda,\mathcal{C}_{p_0}F(x)\leq \gamma\lambda\})\leq c\gamma^{c_w}w(Q_j),
\end{align*}
and summing in $j$ we would get \eqref{good-lambda:A-Cp:w}.

Let us now fix $j\in \N$ and obtain \eqref{good-lambda:A-Cp:dx}. There is nothing to prove if the set on its left-hand side is empty. Thus, we assume that there exists $\bar{x}_j\in \{x\in Q_j: \mathcal{A}F(x)>2\lambda,\mathcal{C}_{p_0}F(x)\leq \gamma\lambda\}$. Let $B_j$ be the ball such that $Q_j\subset B_j$ with $2r_{B_j}=\textrm{diam}(Q_j)$. Then,
$d(x_j,Q_j)\leq 8r_{B_j}$ and $Q_j\subset \overline{B(x_j, 10r_{B_j})}$.

We now write
\begin{align*}
F(x,t)
=
F_{1,j}(x,t)+F_{2,j}(x,t)
:=
F(x,t)\,\chi_{[r_{B_j},\infty)}(t)+
F(x,t)\,\chi_{(0,r_{B_j})}(t).
\end{align*}
In particular, $\mathcal{A}F(x)\leq \mathcal{A}F_{1,j}(x)+\mathcal{A}F_{2,j}(x)$. Easy calculations lead to obtain that for every $\alpha\ge 11$ there holds
\begin{align}\label{AF1}
\mathcal{A}F_{1,j}(x)^2
&
=
\int_{r_{B_j}}^{\infty}\int_{|x-y|<t}|F(y,t)|^2 \frac{dy \, dt}{t^{n+1}}
\leq
\int_{0}^{\infty}\int_{|x_j-y|<\alpha t}|F(y,t)|^2 \frac{dy \, dt}{t^{n+1}}
=
\mathcal{A}^{\alpha}F(x_j)^2 \leq \lambda^2,
\end{align}
where in the last inequality we have used the fact that $x_j\in \R^n\setminus O_\lambda$.
On the other hand, by our choice of $\bar{x}_j\in Q_j\subset B_j$, it follows that
\begin{align}\label{AF2}
\frac{1}{|B_j|}\int_{B_j}\mathcal{A}F_{2,j}(x)^{p_0} \, dx
&=
\frac{1}{|B_j|}\int_{B_j}\left(\int_0^{r_{B_j}}\int_{B(x,t)} |F(y,t)|^2 \frac{dy \, dt}{t^{n+1}}\right)^{\frac{p_0}{2}}dx
\leq
\mathcal{C}_{p_0}F(\bar{x}_j)^{p_0}
\leq
(\gamma\lambda)^{p_0}.
\end{align}
Using \eqref{AF1}, Chebychev's  inequality, and \eqref{AF2} we conclude \eqref{good-lambda:A-Cp:dx}:
\begin{multline*}
|\{x\in Q_j: \mathcal{A}F(x)>2\lambda, C_{p_0}F(x)\leq \gamma \lambda\}|\leq |\{x\in Q_j: \mathcal{A}F_{2,j}(x)>\lambda\}|
\\
\le
\frac1{\lambda^{p_0}}\int_{Q_j}\mathcal{A}F_{2,j}(x)^{p_0} \, dx
\leq \gamma^{p_0}|B_j|
\leq c\gamma^{p_0}|Q_j|.
\end{multline*}
This completes the proof of Step $1$.

\medskip

{\em Step $2$}: Take $F\in L^2_{\rm loc}(\R^{n+1}_+)$ and define, for every $N>1$, $F_N:=F\chi_{K_N}$. Then, since $F_N\in L^2(\R^{n+1}_+)$ and $\supp F_N\subset K_N$, we can apply Step $1$ and obtain that
$$
\|\mathcal{A}F_N\|_{L^p(w)}\lesssim \|\mathcal{C}_{p_0}F_N\|_{L^p(w)}\leq \|\mathcal{C}_{p_0}F\|_{L^p(w)},
$$
where the implicit constant is uniform on $N$.
Finally since $F_N\nearrow F$ in $\R^{n+1}_+$, the Monotone Convergence Theorem yields the desired estimate. This finishes the proof of $(a)$.

\medskip

We next turn to prove $(b)$.
For every $x_0\in \mathbb{R}^n$ and any ball $B\subset \mathbb{R}^n$ such that $x_0\in B$, we have
\begin{align*}
\left(\dashint_B\left(\int_0^{r_B}\int_{B(x,t)}|F(y,t)|^2\frac{dy \, dt}{t^{n+1}}\right)^{\frac{p_0}{2}} \, dx\right)^{\frac{1}{p_0}}
\leq \left(\dashint_B|\mathcal{A}F(x)|^{p_0} \, dx\right)^{\frac{1}{p_0}}
\leq
\mathcal{M}_{p_0}(\mathcal{A}F)(x_0),
\end{align*}
where for any function, $h$, $\mathcal{M}_{p_0} h(x):= \mathcal{M}\big(|h|^{p_0}\big)(x)^{1/p_0}$. Taking the supremum over all balls containing $x_0$, we conclude that $\mathcal{C}_{p_0}F(x_0)\leq \mathcal{M}_{p_0}(\mathcal{A}F)(x_0)$.
Besides, since  $\mathcal{M}_{p_0}:L^{p}(w)\rightarrow L^{p}(w)$ (because $w\in A_{\frac{p}{p_0}}$ and $p>p_0$) we finally conclude that
$$
\|\mathcal{C}_{p_0}F\|_{L^{p}(w)}\leq \|\mathcal{M}_{p_0}(\mathcal{A}F)\|_{L^{p}(w)}\lesssim  \|\mathcal{A}F\|_{L^{p}(w)}.
$$
This completes the proof.% of Proposition \ref{prop:maximal}.
\end{proof}

We conclude this section by stating some easy consequences of the previous results for other tent spaces. Prior to formulating the resulting estimates, we define, for each $0<q<\infty$, the following operators
\begin{align*}
\mathcal{A}_q^{\alpha}F(x):=\left(\iint_{\Gamma^{\alpha}(x)}|F(x,t)|^q \ \frac{dy \, dt}{t^{n+1}}\right)^{\frac{1}{q}},
\qquad
\mathcal{C}_qF(x):=\sup_{B\ni x}\left(\frac{1}{|B|}\int_0^{r_B}\int_{B}|F(y,t)|^q\frac{dy \, dt}{t}\right)^{\frac{1}{q}},
\end{align*}
and
\begin{align*}
\mathcal{C}_{q,p_0}F(x):=\sup_{B\ni x}\left(\dashint_B\left(\int_0^{r_B}\int_{B(x,t)}|F(y,t)|^q\frac{dy \, dt}{t^{n+1}}\right)^{\frac{p_0}{q}} \, dx\right)^{\frac{1}{p_0}}.
\end{align*}
Much as before we have that $\mathcal{C}_q F\approx \mathcal{C}_{q,q}F$. Besides, we obtain the following analogues of Propositions \ref{prop:alpha} and \ref{prop:maximal}.

\begin{proposition}
Let $0<q<\infty$ and $0< \alpha\leq \beta<\infty$.
\begin{list}{$(\theenumi)$}{\usecounter{enumi}\leftmargin=1cm
\labelwidth=1cm\itemsep=0.2cm\topsep=.2cm
\renewcommand{\theenumi}{\roman{enumi}}}

\item For every $w\in A_r$, $1\le r<\infty$, there holds
$$
\|\mathcal{A}^{\beta}_q F\|_{L^p(w)}\
\leq
C \left(\frac{\beta}{\alpha}\right)^{\frac{nr}{p}}
\|\mathcal{A}^{\alpha}_q F\|_{L^p(w)}, \quad \textrm{for all} \quad 0<p\leq qr.
$$

\item For every $w\in RH_{s'}$, $1\le s<\infty$, there holds
$$
\|\mathcal{A}^{\alpha}_q F\|_{L^p(w)}\leq  C\left(\frac{\alpha}{\beta}\right)^{\frac{n}{sp}} \|\mathcal{A}^{\beta}_q F\|_{L^p(w)},\quad \text{for all} \quad \frac{q}{s}\leq p<\infty.
$$
\end{list}
\end{proposition}

\begin{proposition}\
\begin{list}{$(\theenumi)$}{\usecounter{enumi}\leftmargin=1cm \labelwidth=1cm\itemsep=0.2cm\topsep=.2cm \renewcommand{\theenumi}{\alph{enumi}}}
\item If $0<p_0,p<\infty$, $w \in A_{\infty}$, and $F\in L^q_{\rm loc}(\R^{n+1}_+)$ then
\begin{align*}
\|\mathcal{A}_q F\|_{L^p(w)}\lesssim \|\mathcal{C}_{q,p_0}F\|_{L^p(w)}.
\end{align*}

\item If $0< p_0< p<\infty$ and $w\in A_{\frac{p}{p_0}}$ then
\begin{align*}
\|\mathcal{C}_{q,p_0}F\|_{L^p(w)}\lesssim \|\mathcal{A}_q F\|_{L^p(w)}.
\end{align*}
\end{list}
\end{proposition}

The proofs of these results follow immediately from Propositions \ref{prop:alpha} and \ref{prop:maximal}, and the equalities
$$
\mathcal{A}_q^{\alpha}F(x)=\mathcal{A}^{\alpha}(|F|^{\frac{q}{2}})(x)^{\frac{2}{q}} \qquad \textrm{and} \qquad \mathcal{C}_{q,p_0}F(x)=\mathcal{C}_{\frac{2\,p_0}{q}} \big(|F|^{\frac{q}{2}}\big)(x)^{\frac{2}{q}}.
$$

%%%%%%%%%%%%%%%%%%%%%%%%%%%%
%%%%%%%%%%%%%%%%%%%%%%%%%%%%
\section{Proofs of the main results }\label{section:main-results-SF}
%%%%%%%%%%%%%%%%%%%%%%%%%%%%
%%%%%%%%%%%%%%%%%%%%%%%%%%%%
In this section we shall prove Theorems \ref{thm:SF-Heat}, \ref{thm:SF-Poisson} , \ref{theor:control-SF-Heat}, and \ref{theor:control-SF-Poisson}.

\medskip

Let us first fix some notation. Given a ball $B\in \mathbb{R}^n$, and unless otherwise specified, we write $x_B$ and $r_B$ to denote respectively its center and its radius, so that $B=B(x_B,r_B)$. For every $\lambda>0$ let $\lambda B=B(x_B,\lambda\,B)$ be the ball concentric with $B$ whose radius is $\lambda\,r_B$. Finally, we write
$$
C_1(B):=4B,\qquad C_j(B):=2^{j+1}B\setminus 2^jB,\quad j\ge 2.
$$

\subsection{Proof of Theorem \ref{thm:SF-Heat}}

\subsubsection{Proof of Theorem \ref{thm:SF-Heat}, part $(a)$}

%%%%%%%%%%%%%%%%%%%%%%%%%%%%%%%%%%%%%%%%%%%%%%%%%%%%%%%%%%%%%%%%%%%%%%%%%%%%%%
%%%%%%%%%%%%%%%%%%%%%%%%%%%%%%%%%%%%%%%%%%%%%%%%%%%%%%%%%%%%%%%%%%%%%%%%%%%%%%%

Let us start by introducing more notation.
From now on, $\mathcal{Q}_t$ denotes $t^2Le^{-t^2L}$, $t\nabla_{y}e^{-t^2L}$, or $t\nabla_{y,t}e^{-t^2L}$ in such a way that, if we write
$$
\widetilde{\mathcal{A}}f(x):=\left(\iint_{\Gamma(x)}|\mathcal{Q}_t f(y)|^2 \frac{dy \, dt}{t^{n+1}}\right)^{\frac{1}{2}},
$$
then $\widetilde{\mathcal{A}}f$ is respectively $\Scal_{\hh}f$, $\Grm_{\hh}f$, or $\Gcal_{\hh}f$.

The boundedness of $\widetilde{\mathcal{A}}$ follows from the combination of Proposition \ref{prop:maximal} and the following auxiliary result.
%In order to show the boundedness of $\widetilde{\mathcal{A}}$ we need an auxiliary result which gives us an estimate for the maximal operator $\mathcal{C}_{p}$ introduced before  when acting on any of the $\mathcal{Q}_t$ as above.

\begin{proposition}\label{prop:Cpo-Mp0-S-heat}
Let $\mathcal{Q}_t$ denote $t^2Le^{-t^2L}$, $t\nabla_{y}e^{-t^2L}$, or $t\nabla_{y,t}e^{-t^2L}$. If we set
$$
\widetilde{\mathcal{C}}_{p_0}f(x):=\sup_{B\ni x}\left(\frac{1}{|B|}\int_{B}\left(\int_{0}^{r_B}\int_{B(x,t)}|\mathcal{Q}_t f(y)|^2 \frac{dy \, dt}{t^{n+1}}\right)^{\frac{p_0}{2}} \, dx\right)^{\frac{1}{p_0}},
$$
then, for every $p_-(L)<p_0\le 2$, there holds
\begin{equation}\label{Cpo-Mp0-S-heat}
\widetilde{\mathcal{C}}_{p_0}f(x)\lesssim \mathcal{M}_{p_0}f(x), \qquad x\in \mathbb{R}^n.
\end{equation}
\end{proposition}

Assuming this result momentarily we prove Theorem \ref{thm:SF-Heat}, part $(a)$. Note that taking $F(y,t)=\mathcal{Q}_t f(y)$ in \eqref{AF} and in \eqref{CF} we have that $\widetilde{\mathcal{A}}f(x)=\mathcal{A}F(x)$ and $\widetilde{\mathcal{C}}_{p_0}f(x)=\mathcal{C}_{p_0}F(x)$. Thus \eqref{Cpo-Mp0-S-heat}, in concert with $(a)$ in  Proposition \ref{prop:maximal},  implies that, for every $0<p<\infty$ and $w\in A_\infty$,
\begin{align*}
 \|\widetilde{\mathcal{A}}f\|_{L^p(w)}\lesssim \|\widetilde{\mathcal{C}}_{p_0}f\|_{L^p(w)}\lesssim \|\mathcal{M}_{p_0}f\|_{L^p(w)}, \quad \textrm{for all} \quad p_-(L)<p_0\le 2,
\end{align*}
provided $\mathcal{Q}_tf\in L^2_{\rm loc}(\R_+^{n+1})$. Hence, the above estimate holds for all functions $f\in L^\infty_c(\R^n)$.
Next, fix $w\in A_\infty$ and $p\in \mathcal{W}_w(p_-(L),\infty)$. Then, there exists $p_-(L)<p_0\le 2$ (close enough to $p_-(L)$), such that $w\in A_{\frac{p}{p_0}}$. Therefore,  $\mathcal{M}_{p_0}$ is bounded on $L^p(w)$ and consequently the previous estimate leads to
\begin{align}\label{est:Atilde}
\|\widetilde{\mathcal{A}}f\|_{L^p(w)}
\le
C\,\|f\|_{L^p(w)},\qquad\forall\, f\in L^\infty_c(\R^n).
\end{align}
A routine density argument allows one to extend this estimate to all functions in $L^p(w)$.

Let us notice that \eqref{Cpo-Mp0-S-heat} with $p_0=2$ appears implicit in \cite[p. 5479]{AuscherHofmannMartell}. Having used that estimate we would have obtained \eqref{est:Atilde} for every $2<p<\infty$ and $w\in A_{p/2}$. However, using $\mathcal{C}_{p_0}$ with $p_0$ very close to $p_-(L)$
allows to obtain better estimates: \eqref{est:Atilde} holds for every $p_-(L)<p<\infty$ and $w\in A_{p/p_-(L)}$.

We are left with the proof of Proposition \ref{prop:Cpo-Mp0-S-heat}, in  which we shall use the following unweighted estimates for the conical square functions that we are currently considering.

\begin{proposition}\label{prop:Sh}
The square functions $\Scal_{\hh}$,  $\Grm_{\hh}$, and  $\Gcal_{\hh}$ are bounded on $L^p(\mathbb{R}^n)$ for every $p_-(L)<p\le 2$.
\end{proposition}

Let us note that the boundedness of $\Grm_{\hh}$ has been established in \cite[Section 6.2]{Auscher}. On the other hand, one can  easily see that $\Gcal_{\hh}\lesssim \Grm_{\hh}+\Scal_{\hh}$, and therefore we only have to consider $\Scal_{\hh}$. In turn, this operator will be handled by
using a Calder\'on-Zygmund type result from \cite{AuscherMartell:III} after the proof of Proposition \ref{prop:Cpo-Mp0-S-heat}.

We would like to observe that, a posteriori, Theorem \ref{thm:SF-Heat}, part $(a)$, applied with $w\equiv 1$, implies that $\Scal_{\hh}$,  $\Grm_{\hh}$, and  $\Gcal_{\hh}$ are also bounded on $L^p(\mathbb{R}^n)$ for every $2\le p<\infty$ (and therefore in the range $p_-(L)< p<\infty$). The case $\Grm_{\hh}$ was obtained in \cite[Theorem 3.1, part (2)]{AuscherHofmannMartell}.

\begin{proof}[Proof of Proposition \ref{prop:Cpo-Mp0-S-heat}]
Fix $p_-(L)<p_0\le 2$ and  $x_0\in\R^n$. Take an arbitrary ball $B\ni x_0$ and split $f$  into its local and global parts:
$f=f_{\rm loc}+f_{\rm glob}:=f\chi_{4B}+f\chi_{\R^n\setminus 4B}$.

For $f_{\rm loc}$, we use that $\widetilde{\mathcal{A}}$ is bounded on $L^{p_0}(\mathbb{R}^n)$ by Proposition \ref{prop:Sh}:
\begin{multline*}
\left(\frac{1}{|B|}\int_{B}\left(\int_{0}^{r_B}\int_{B(x,t)}|\mathcal{Q}_tf_{\rm loc}(y)|^2 \frac{dy \, dt}{t^{n+1}}\right)^{\frac{p_0}{2}} \, dx \right)^{\frac{1}{p_0}}
\le
\left(\frac{1}{|B|}\int_{\mathbb{R}^n} \widetilde{\mathcal{A}} f_{\rm loc}(x)^{p_0}\,dx\right)^{\frac{1}{p_0}}
\\
\lesssim \left(\frac{1}{|B|}\int_{\mathbb{R}^n}|f_{\rm loc}(x)|^{p_0} \, dx\right)^{\frac{1}{p_0}}
\lesssim \left(\frac{1}{|4B|}\int_{4B}|f(x)|^{p_0} \, dx\right)^{\frac{1}{p_0}}
\lesssim \mathcal{M}_{p_0}f(x_0).
\end{multline*}
As for $f_{\rm glob}$, since $\{\mathcal{Q}_t\}_{t>0}\in \mathcal{F}_\infty(L^{p_0}\rightarrow L^2)$ ---where we recall that $\mathcal{Q}_t$ is $t^2Le^{-t^2L}$, $t\nabla_{y}e^{-t^2L}$, or $t\nabla_{y,t}e^{-t^2L}$--- and since $\supp f_{\rm glob}\subset \R^n\setminus 4B$, we have
\begin{align*}
&\left(\frac{1}{|B|}\int_{B}\left(\int_{0}^{r_B}\int_{B(x,t)}|\mathcal{Q}_tf_{\rm glob}(y)|^2 \frac{dy \, dt}{t^{n+1}}\right)^{\frac{p_0}{2}} \, dx \right)^{\frac{1}{p_0}}
\lesssim
\sum_{j\ge 2} \left(\int_{0}^{r_B}\int_{2B}|\mathcal{Q}_t(f\,\chi_{C_j(B)})(y)|^2\frac{dy \, dt}{t^{n+1}}\right)^{\frac{1}{2}}
\\
&\qquad\qquad\lesssim \sum_{j\geq 2} \left(\int_{0}^{r_B}\left(\int_{2^{j+1}B}|f(y)|^{p_0} \, dy\right)^{\frac{2}{p_0}}t^{-2\,n(\frac{1}{p_0}-\frac12)}e^{-c\frac{4^jr_B^2}{t^2}}\frac{dt}{t^{n+1}}\right)^{\frac{1}{2}}
\\
&\qquad\qquad\lesssim \mathcal{M}_{p_0}f(x_0) \sum_{j\geq 2} \left(\int_{0}^{r_B}(2^jr_B)^{\frac{2n}{p_0}}t^{-\frac{2n}{p_0}}e^{-c\frac{4^jr_B^2}{t^2}} \frac{dt}{t}\right)^{\frac{1}{2}}
\lesssim \mathcal{M}_{p_0}f(x_0).
\end{align*}
Gathering the estimates obtained for $f_{\rm loc}$ and for $f_{\rm glob}$ we conclude that
$$
\left(\frac{1}{|B|}\int_{B}\left(\int_{0}^{r_B}\int_{B(x,t)}|\mathcal{Q}_tf(y)|^2 \frac{dy \, dt}{t^{n+1}}\right)^{\frac{p_0}{2}} \, dx \right)^{\frac{1}{p_0}}
\lesssim \mathcal{M}_{p_0}f(x_0).
$$
Taking the supremum over all balls $B$ such that $x_0\in B$ we readily conclude the desired estimate.
\end{proof}

\begin{proof}[Proof of Proposition \ref{prop:Sh}]
As explained above we only need to consider the operator $\Scal_{\hh}$. It is well-known that $\Scal_{\hh}$ is bounded on $L^2(\R^n)$. Fix then
$p_-(L)<p<2$ and take $p_-(L)<p_0<p<2$. We shall apply \cite[Theorem 2.4]{AuscherMartell:III} (see also \cite{Auscher} and \cite{AuscherMartell:I}).
We claim that given $f\in L^{\infty}_c(\mathbb{R}^n)$ with $\supp f\subset B\subset \mathbb{R}^n$, the following estimates hold
\begin{align}\label{241}
\left(\dashint_{C_j(B)}\left|\Scal_{\hh}(I-e^{-r_B^2L})^Mf\right|^{p_0} \, dx\right)^{\frac{1}{p_0}}
\leq g(j)\left(\dashint_{B}\left|f\right|^{p_0} \, dx\right)^{\frac{1}{p_0}}, \qquad j\geq 2,
\end{align}
and
\begin{align}\label{242}
\left(\dashint_{C_j(B)}\left|I-(I-e^{-r_B^2L})^Mf\right|^2 \, dx\right)^{\frac{1}{2}}
\leq g(j)\left(\dashint_{B}\left|f\right|^{p_0} \, dx\right)^{\frac{1}{p_0}},\qquad j\geq 1,
\end{align}
with $g(j)=C\,2^{-j\,(2M+\frac{n}{p_0})}$. Assuming this momentarily and taking $M$ large enough in such a way that
$\sum_{j\ge 1} g(j)2^{jn}<\infty$, \cite[Theorem 2.4]{AuscherMartell:III} implies that $\Scal_{\hh}$ is of weak-type $(p_0,p_0)$ and, by Marcinkiewicz's interpolation Theorem, bounded on $L^p(\R^n)$, which is our goal.

In view of the previous considerations we need to obtain \eqref{241} and \eqref{242}. Fix a ball $B$. For $f\in L^{\infty}_c(\mathbb{R}^n)$ with $\supp f\subset B$, we first prove \eqref{241}. Define $A_{r_B^2}:=(I-e^{-r_B^2L})^M$ and by Fubini (or see \cite[Lemma 1]{CoifmanMeyerStein}) conclude that
\begin{align*}
&\left(\dashint_{C_j(B)}\left|\Scal_{\hh} A_{r_B^2}f(x)\right|^{p_0} \, dx\right)^{\frac{1}{p_0}}
\leq
\left(\dashint_{C_j(B)}\left|\Scal_{\hh} A_{r_B^2}f(x)\right|^2 \, dx\right)^{\frac{1}{2}}
\\
&
\qquad\qquad\lesssim
|2^jB|^{-\frac{1}{2}}\left(\iint_{\mathcal{R}(\overline{C_j(B)})}\left|t^2Le^{-t^2L}
A_{r_B^2}f(y)\right|^2\frac{dy \, dt}{t}\right)^{\frac{1}{2}}
\\
&
\qquad\qquad\lesssim
|2^jB|^{-\frac{1}{2}}\left(\int_{\mathbb{R}^n\setminus 2^{j-1}B}\int_{0}^{\infty}\left|t^2Le^{-t^2L}
A_{r_B^2}f(y)\right|^2\frac{dt \, dy}{t}\right)^{\frac{1}{2}}
\\
&
\hskip4cm
+ |2^jB|^{-\frac{1}{2}}\left(\int_{2^{j-1}\,B}\int_{2^{j-1} r_B}^{\infty}
\left|t^2Le^{-t^2L}
A_{r_B^2}f(y)\right|^2\frac{dt \, dy}{t}\right)^{\frac{1}{2}}
\\
&
\qquad\qquad=: |2^jB|^{-\frac{1}{2}}(I+II).
\end{align*}
We estimate each term in turn. Before that, let us remind the following off-diagonal estimate obtained in \cite[p. 504]{HofmannMartell}:
\begin{equation}\label{off-diff-semi}
\left\|\frac{s^2}{t^2}\big( e^{-s^2L}-e^{-(s^2+t^2)L}\big)(f\,\chi_E)\right\|_{L^2(F)}
\le
C\,  e^{-c\,\frac{d(E,F)^2}{s^2}}\,\|f\|_{L^2(E)}, \qquad 0<t\le s,
\end{equation}
with $C$ independent of $t$ and $s$. This and Lemma \ref{lemma:composition} imply that for every $M\ge 1$ there exists $C$ such that for every $0<t\le s$ there holds
\begin{equation}\label{off-diff-semi:2}
\left\|s^2 L\,e^{-s^2 L}\left(\frac{s^2}{t^2}\right)^M\big(e^{-s^2L}-e^{-(s^2+t^2)L}\big)^M(f\,\chi_E)\right\|_{L^2(F)}
\le
C\,  s^{-n\left(\frac{1}{p_0}-\frac{1}{2}\right)} e^{-c\,\frac{d(E,F)^2}{s^2}}\,\|f\|_{L^{p_0}(E)}.
\end{equation}

After these preparations we estimate $II$.   Doing the change of variables $t=\sqrt{M+1}\,s$ and using \eqref{off-diff-semi:2}, easy calculations lead to obtain \begin{align*}
II
&
\lesssim
\left(\int_{c 2^{j}r_B}^{\infty}
\left\|s^2 L e^{-s^2L}
\big(e^{-s^2 L}-e^{-(s^2+r_B^2)L}\big)^Mf\right\|_{L^2(2^{j-1}\,B)}^2\frac{ds}{s}\right)^{\frac{1}{2}}
\\
&
\lesssim
\left(\int_{c2^jr_B}^{\infty} \left(\frac{r_B}{s}\right)^{4M} s^{-2n\left(\frac{1}{p_0}-\frac{1}{2}\right)}\frac{ds}{s}\right)^{\frac{1}{2}}\,\|f\|_{L^{p_0}(B)}
\\
&
\lesssim
2^{-j\,(2M+\frac{n}{p_0})}\,|2^jB|^{\frac12}\,\left(\dashint_{B}|f(x)|^{p_0} \, dx\right)^{\frac{1}{p_0}}.
\end{align*}

Let us next estimate $I$. We proceed as in \cite{HofmannMartell} or \cite[p. 53-56]{HofmannMayboroda}. Change variables as before to obtain
\begin{multline*}
I
\lesssim
\left(\int_{0}^{r_B}
\left\|
s^2Le^{-(M+1)s^2L} \big(I-e^{-r_B^2L}\big)^Mf
\right\|_{L^2(\mathbb{R}^n\setminus 2^{j-1}B)}^2\frac{dt}{t}
\right)^{\frac{1}{2}}
\\
+
\left(\int_{r_B}^{\infty}
\left\|s^2 L e^{-s^2L}
\big(e^{-s^2 L}-e^{-(s^2+r_B^2)L}\big)^M f \right\|_{L^2(\mathbb{R}^n\setminus 2^{j-1}B)}^2\frac{dt}{t}\right)^{\frac{1}{2}}
=: I_1+I_2.
\end{multline*}
For $I_2$, employ \eqref{off-diff-semi:2} and conclude that
\begin{align*}
I_2
&
\lesssim
\left(\int_{r_B}^{\infty} \left(\frac{r_B}{s}\right)^{4M} s^{-2n\left(\frac{1}{p_0}-\frac{1}{2}\right)} e^{-c\,\frac{4^j\,r_B^2}{s^2}}\frac{ds}{s}\right)^{\frac{1}{2}}\,\|f\|_{L^{p_0}(B)}
\lesssim
2^{-j\,(2M+\frac{n}{p_0})}\,|2^jB|^{\frac12}\,\left(\dashint_{B}|f(x)|^{p_0} \, dx\right)^{\frac{1}{p_0}}.
\end{align*}
For $I_1$, expand $\big(I-e^{-r_B^2L}\big)^M$ and use  the $L^{p_0}-L^2$ off-diagonal estimates satisfied by the Heat semigroup:
\begin{align*}
I_1
&
\lesssim
\left(\int_{0}^{r_B}
\left\|
s^2Le^{-(M+1)s^2L} f\right\|_{L^2(\mathbb{R}^n\setminus 2^{j-1}B)}^2\frac{ds}{s}
\right)^{\frac{1}{2}}
\\
&\quad
+
\sup_{1\le k\le M}\left(\int_{0}^{r_B}
\left\|
s^2Le^{-\big((M+1)s^2+kr_B^2\big)L}f
\right\|_{L^2(\mathbb{R}^n\setminus 2^{j-1}B)}^2\frac{ds}{s}
\right)^{\frac{1}{2}}
\\
&\lesssim
\left(\int_{0}^{r_B} s^{-2n\left(\frac{1}{p_0}-\frac{1}{2}\right)} e^{-c\,\frac{4^j\,r_B^2}{s^2}}\frac{ds}{s}
\right)^{\frac{1}{2}}\|f\|_{L^{p_0}(B)}
\\
&\quad
+
\sup_{1\le k\le M}\left(\int_{0}^{r_B}
\left(\frac{s^2}{(M+1)s^2+kr_B^2}\right)^{2}\, \big((M+1)s^2+kr_B^2\big)^{-n\left(\frac{1}{p_0}-\frac{1}{2}\right)} e^{-c\,\frac{4^j\,r_B^2}{(M+1)s^2+kr_B^2}}
\frac{ds}{s}
\right)^{\frac{1}{2}}\|f\|_{L^{p_0}(B)}
\\
\\
&\lesssim
e^{-c\,4^j}\,|2^jB|^{\frac12}\left(\dashint_{B}|f(x)|^{p_0} \, dx\right)^{\frac{1}{p_0}}
+
e^{-c\,4^j}\,|2^jB|^{\frac12}\left(\dashint_{B}|f(x)|^{p_0} \, dx\right)^{\frac{1}{p_0}}\,
\left(\int_{0}^{r_B}
\left(\frac{s^2}{r_B^2}\right)^{2}\, \frac{ds}{s}
\right)^{\frac{1}{2}}
\\
&\lesssim
e^{-c\,4^j}\,|2^jB|^{\frac12}\left(\dashint_{B}|f(x)|^{p_0} \, dx\right)^{\frac{1}{p_0}}.
\end{align*}
Gathering all the estimates that we have obtained we complete the proof of  \eqref{241}:
\begin{align*}
\left(\dashint_{C_j(B)}\left|\Scal_{\hh} A_{r_B^2}f(x)\right|^{p_0} \, dx\right)^{\frac{1}{p_0}}
\leq
C
2^{-j\,(2M+\frac{n}{p_0})}\,\left(\dashint_{B}|f(x)|^{p_0} \, dx\right)^{\frac{1}{p_0}}.
\end{align*}

\medskip

To prove \eqref{242}, we use that $\{e^{-t^2L}\}_{t>0}\in \mathcal{F}_\infty(L^{p_0}\rightarrow L^2)$ and for every $j\ge 1$
\begin{align*}
\left(\dashint_{C_j(B)}|I-(I-e^{-r_B^2L})^Mf|^2\right)^{\frac{1}{2}}
\leq \sum_{k=1}^{M}C_{k,M}\left(\dashint_{C_j(B)}|e^{-kr_B^2L}f|^2\right)^{\frac{1}{2}}
\lesssim e^{-c4^j} \left(\dashint_B |f(x)|^{p_0}\,dx\right)^{\frac{1}{p_0}}.
\end{align*}
\end{proof}

\subsubsection{Proof of Theorem \ref{thm:SF-Heat}, part $(b)$}

Take $w\in A_{\infty}$, $m\in \N$, and $f\in L^{\infty}_c(\mathbb{R}^n)$, and apply Theorem \ref{theor:control-SF-Heat} (see below for its proof). Then, for all $0<p<\infty$,
$$
\|\Grm_{m,\hh}f\|_{L^p(w)}\lesssim \|\Scal_{\hh}f\|_{L^p(w)}, \quad \|\Gcal_{m,\hh}f\|_{L^p(w)}\lesssim \|\Scal_{\hh}f\|_{L^p(w)},\quad \textrm{and}\quad \|\Scal_{m,\hh}f\|_{L^p(w)}\lesssim \|\Scal_{\hh}f\|_{L^p(w)}.
$$
Now, use Theorem \ref{thm:SF-Heat} part $(a)$ to conclude that for all $p\in \mathcal{W}_w(p_-(L),\infty)$
$$\|\Grm_{m,\hh}f\|_{L^p(w)}\lesssim \|f\|_{L^p(w)}, \quad \|\Gcal_{m,\hh}f\|_{L^p(w)}\lesssim \|f\|_{L^p(w)},\quad \textrm{and}\quad \|\Scal_{m,\hh}f\|_{L^p(w)}\lesssim \|f\|_{L^p(w)},$$
for all $f\in L^{\infty}_c(\mathbb{R}^n)$. By a standard density argument these estimates easily extend to all functions $f\in L^p(w)$.\qed

%%%%%%%%%%%%%%%%%%%%%%%%%%%%%%%%%%%%%%%%%%%%%%%%%%%%%%%%%%%%%%
%%%%%%%%%%%%%%%%%%%%%%%%%%%%%%%%%%%%%%%%%%%%%%%%%%%%%%%%%%%%%%

\subsection{Proof of Theorem \ref{thm:SF-Poisson}}

%%%%%%%%%%%%%%%%%%%%%%%%%%%%%%%%%%%%%%%%%%%%%%%%%%%%%%%%%%%%%%
%%%%%%%%%%%%%%%%%%%%%%%%%%%%%%%%%%%%%%%%%%%%%%%%%%%%%%%%%%%%%%

\subsubsection{Proof of Theorem \ref{thm:SF-Poisson}, part $(a)$}
From Theorem \ref{theor:control-SF-Poisson} part $(b)$ (see below for its proof), given $w\in A_{\infty}$, we have for all $K\in \N$, $p\in  \mathcal{W}_w(0,p_+(L)^{K,*})$, and $f\in L^{\infty}_c(\mathbb{R}^n)$
$$
\|\Scal_{K,\pp}f\|_{L^p(w)}\lesssim \|\Scal_{\hh}f\|_{L^p(w)}.
$$
Hence, applying Theorem \ref{thm:SF-Heat} part $(a)$, we obtain, for all $p\in \mathcal{W}_w(p_-(L),p_+(L)^{K,*})$,
$$
\|\Scal_{K,\pp}f\|_{L^p(w)}\lesssim \|f\|_{L^p(w)}, \qquad f\in L^\infty_c(\R^n).
$$
A density argument allows us to complete the proof.
\qed
%%%%%%%%%%%%%%%%%%%%%%%%%%%%%%%%%%%%%%%%%%%%%%%%%%%%%%%%%%%%%
%%%%%%%%%%%%%%%%%%%%%%%%%%%%%%%%%%%%%%%%%%%%%%%%%%%%%%%%%%%%%

\subsubsection{Proof of Theorem \ref{thm:SF-Poisson}, part $(b)$}
Take $w\in A_{\infty}$ and apply Theorem \ref{theor:control-SF-Poisson} parts $(a)$, $(c)$, and $(d)$ (the proof of this result is given below) to obtain, for all $K\in \N$, $p\in \mathcal{W}_w(0,p_+(L)^{K,*})$, and $f\in L^{\infty}_c(\R^n)$
$$
\|\Grm_{K,\pp}f\|_{L^p(w)}\lesssim\|\Scal_{\hh}f\|_{L^p(w)}\quad \textrm{and}\quad \|\Gcal_{K,\pp}f\|_{L^p(w)}\lesssim \|\Scal_{\hh}f\|_{L^p(w)},
$$
and
$$
\|\Grm_{\pp}f\|_{L^p(w)}\lesssim\|\Gcal_{\hh}f\|_{L^p(w)}\quad \textrm{and}\quad \|\Gcal_{\pp}f\|_{L^p(w)}\lesssim \|\Gcal_{\hh}f\|_{L^p(w)}.$$
Now, apply  Theorem \ref{thm:SF-Heat}, part $(a)$, and conclude, by a density argument, that
for all $K\in \N_0$, $p\in \mathcal{W}_w(p_-(L),p_+(L)^{K,*})$, and $f\in L^p(w)$, there hold
$$\|\Grm_{K,\pp}f\|_{L^p(w)}\lesssim\|f\|_{L^p(w)}\quad \textrm{and}\quad \|\Gcal_{K,\pp}f\|_{L^p(w)}\lesssim \|f\|_{L^p(w)}.$$
\qed

%%%%%%%%%%%%%%%%%%%%%%%%%%%%%%%%%%%%%%%%%%%%%%%%%%%%%%%%%%%%%%
%%%%%%%%%%%%%%%%%%%%%%%%%%%%%%%%%%%%%%%%%%%%%%%%%%%%%%%%%%%%%%

\subsection{Proof of Theorem \ref{theor:control-SF-Heat}}

We first note that part $(a)$ is trivial.

%\subsubsection{Proof of Theorem \ref{theor:control-SF-Heat}, part $(a)$}
%It follows immediately from the following fact:
%$$
%\big|t\nabla_y(t^2L)^me^{-t^2L}f(y)\big|
%\leq
%\big|t\nabla_{y,t}(t^2L)^me^{-t^2L}f(y)\big|, \qquad t>0,\quad y\in \mathbb{R}^n, \quad \textrm{and} \quad m\in \N_{0}.$$
%\qed
%

\subsubsection{Proof of Theorem \ref{theor:control-SF-Heat}, part $(b)$}

For $m=1$ there is nothing to prove. So, take $m\in \N$ such that $m\geq 2$ and consider
$$
T_{\frac{t^2}{2}}:=\left(\tfrac{t^2}{2}L\right)^{m-1} e^{-\frac{t^2}{2}L}.
$$
Fix $0<p<\infty$ and $w\in A_{\infty}$. Pick $r\geq \max\{\frac{p}{2},r_w\}$ so that $w\in A_{r}$ and $0<p\leq 2r$. Then, from $\{(t^2L)^me^{-t^2L}\}_{t>0}\in \mathcal{F}_\infty(L^{2}\rightarrow L^2)$ and applying
 Proposition \ref{prop:alpha} in the next-to-last inequality, we have
\begin{align*}
\|\Scal_{m,\hh}f\|_{L^p(w)}
&
\lesssim \left(\int_{\mathbb{R}^n}\left(\iint_{\Gamma(x)}\left|T_{\frac{t^2}{2}}\left( \tfrac{t^2}{2}Le^{-\frac{t^2}{2}L}f\right)(y)\right|^2 \frac{dy \, dt}{t^{n+1}}\right)^{\frac{p}{2}} w(x) \, dx \right)^{\frac{1}{p}}
\\
& \lesssim  \sum_{j\geq 1} \left(\int_{\mathbb{R}^n}\left(\int_0^{\infty}\int_{B(x,t)}\left|T_{\tfrac{t^2}{2}}\left( \big(\tfrac{t^2}{2}Le^{-\frac{t^2}{2}L}f\big) \chi_{C_j(B(x,t))}\right)(y)\right|^2 \frac{dy \, dt}{t^{n+1}}\right)^{\frac{p}{2}} w(x) \, dx \right)^{\frac{1}{p}}
\\
& \lesssim \sum_{j\geq 1} e^{-c4^j} \left(\int_{\mathbb{R}^n}\left(\int_0^{\infty}\int_{B(x,2^{j+1}t)}\left|\tfrac{t^2}{2}L e^{-\frac{t^2}{2}L}f(y)\right|^2 \frac{dy \, dt}{t^{n+1}}\right)^{\frac{p}{2}} w(x) \, dx \right)^{\frac{1}{p}}
\\
& \lesssim \sum_{j\geq 1} e^{-c4^j} \left(\int_{\mathbb{R}^n}\left(\int_0^{\infty}\int_{B(x,2^{j+1}\sqrt{2}t)}\left|t^2L e^{-t^2L}f(y)\right|^2 \frac{dy \, dt}{t^{n+1}}\right)^{\frac{p}{2}} w(x) \, dx \right)^{\frac{1}{p}}
\\
&\lesssim \sum_{j\geq 1}2^{j\frac{nr}{p}} e^{-c4^j}\left(\int_{\mathbb{R}^n}\left(\int_0^{\infty}\int_{B(x,t)}\left|t^2L e^{-t^2L}f(y)\right|^2 \frac{dy \, dt}{t^{n+1}}\right)^{\frac{p}{2}} w(x) \, dx \right)^{\frac{1}{p}}
\\
&\lesssim \|\Scal_{\hh}f\|_{L^p(w)}. \qed
\end{align*}

\subsubsection{Proof of Theorem \ref{theor:control-SF-Heat}, part $(c)$}

%%%%%%%%%%%%%%%%%%%%%%%%%%%%%%%%%%%%%%%%%%%%%%%%%%%%%%%%%%%%%%
%%%%%%%%%%%%%%%%%%%%%%%%%%%%%%%%%%%%%%%%%%%%%%%%%%%%%%%%%%%%%%
Take $m\in \N$ and consider
$$
A_{\frac{t^2}{2}}:=\tfrac{t}{\sqrt{2}}\nabla_{y,t} e^{-\frac{t^2}{2}L}
\quad \textrm{and} \quad B_{\frac{t^2}{2},m}:=\left(\tfrac{t^2}{2}L\right)^me^{-\frac{t^2}{2}L}.
$$
Fix $0<p<\infty$ and $w\in A_{\infty}$. Pick $r\geq\max\{\frac{p}{2},r_w\}$ so that $w\in A_{r}$ and $0<p\leq 2r$. Then,
applying the $L^2-L^2$ off-diagonal estimates satisfied by $\{t\nabla_{y,t}e^{-t^2L}\}_{t>0}$ and Proposition \ref{prop:alpha}, we obtain
\begin{align*}
\|\Gcal_{m,\hh}f\|_{L^p(w)}
&
\lesssim \left(\int_{\mathbb{R}^n}\left(\int_0^{\infty}\int_{B(x,t)}\left|A_{\frac{t^2}{2}}
B_{\frac{t^2}{2},m}f(y)\right|^2 \frac{dy \, dt}{t^{n+1}}\right)^{\frac{p}{2}} \ w(x) \, dx\right)^{\frac{1}{p}}
\\
&   \lesssim \sum_{j\geq 1} \left(\int_{\mathbb{R}^n}\left(\int_0^{\infty}\int_{B(x,t)}\left|A_{\frac{t^2}{2}}
\left(\big(B_{\frac{t^2}{2},m}f\big)\chi_{C_j(B(x,t))}\right)(y)\right|^2 \frac{dy \, dt}{t^{n+1}}\right)^{\frac{p}{2}} \ w(x) \, dx\right)^{\frac{1}{p}}
\\
& \lesssim \sum_{j\geq 1} e^{-c4^j} \left(\int_{\mathbb{R}^n}\left(\int_0^{\infty}\int_{B(x,2^{j+1}t)}\left|B_{\frac{t^2}{2},m}f(y)\right|^2 \frac{dy \, dt}{t^{n+1}}\right)^{\frac{p}{2}} \ w(x) \, dx\right)^{\frac{1}{p}}
\\
& \lesssim  \sum_{j\geq 1} e^{-c4^j} \left(\int_{\mathbb{R}^n}\left(\int_0^{\infty}\int_{B(x,2^{j+1}\sqrt{2}t)}\left|
B_{t^2,m}f(y)\right|^2 \frac{dy \, dt}{t^{n+1}}\right)^{\frac{p}{2}} \ w(x) \, dx\right)^{\frac{1}{p}}
\\
&\lesssim  \sum_{j\geq 1}e^{-c4^j} 2^{j\frac{nr}{p}} \|\Scal_{m,\hh}f\|_{L^p(w)}
\\
&\lesssim  \|\Scal_{\hh}f\|_{L^p(w)},
\end{align*}
where in the last inequality we have used part $(b)$. \qed

%%%%%%%%%%%%%%%%%%%%%%%%%%%%%%%%%%%%%%%%%%%%%%%%%%%%%%%%%%%%%%
%%%%%%%%%%%%%%%%%%%%%%%%%%%%%%%%%%%%%%%%%%%%%%%%%%%%%%%%%%%%%%

\subsection{Proof of Theorem \ref{theor:control-SF-Poisson}}

We first note that part $(a)$ is trivial.

%\subsubsection{Proof of Theorem \ref{theor:control-SF-Poisson}, part $(a)$}
%It is enough to observe the following:
%$$
%\left|t\nabla_y(t\sqrt{L})^{2K}e^{-t\sqrt{L}}f(y)\right|\leq \left|t\nabla_{y,t}(t\sqrt{L})^{2K}e^{-t\sqrt{L}}f(y)\right|
% , \qquad t>0,\quad y\in \mathbb{R}^n, \quad \textrm{and} \quad K\in \N_{0}.
%$$
%\qed

\subsubsection{Proof of Theorem \ref{theor:control-SF-Poisson}, part $(b)$}

In view of Theorem \ref{theor:control-SF-Heat}, part $(b)$, and Lemma \ref{lemma:extrapol}, it is enough to show that
\begin{align}\label{SKP}
\|\Scal_{K,\pp}f\|_{L^p(w)}\lesssim\|\Scal_{K,\hh}f\|_{L^p(w)},
\end{align}
for all $w\in A_{\infty}$, $K\in \N$, and $p\in \mathcal{W}_{w}(0,p_+(L)^{K,*})$, with $1\le p$.
Set
$$
B_{t,K}:=\left(t^2L\right)^{K}
e^{-t^2L},
$$
and apply the subordination formula \eqref{FR} and Minkowski's inequality:
\begin{align*}
\left\|\Scal_{K,\pp}f\right\|_{L^p(w)}
&  \lesssim \left(\int_{\mathbb{R}^n}\left(\iint_{\Gamma(x)}\left|(t^2L)^{K}\int_0^{\infty}
e^{-u}u^{\frac{1}{2}}
e^{-\frac{t^2}{4u}L}f(y) \ \frac{du}{u}\right|^2 \ \frac{dy \, dt}{t^{n+1}}\right)^{\frac{p}{2}} \ w(x) \, dx\right)^{\frac{1}{p}}
\\
&   \lesssim \int_0^{\infty}e^{-u}u^{\frac{1}{2}}\left(\int_{\mathbb{R}^n}
\left(\int_0^{\infty}\int_{B(x,t)}\left|(t^2L)^{K}e^{-\frac{t^2}{4u}L}f(y)\right|^2 \ \frac{dy \, dt}{t^{n+1}}\right)^{\frac{p}{2}} w(x) \, dx\right)^{\frac{1}{p}}\ \frac{du}{u}
\\
& \lesssim\int_0^{\frac{1}{4}}e^{-u}u^{K+\frac{1}{2}}\left(\int_{\mathbb{R}^n}
\left(\int_0^{\infty}\int_{B(x,t)}\left|B_{\frac{t}{2\sqrt{u}},K}f(y)\right|^2 \ \frac{dy \, dt}{t^{n+1}} \right)^{\frac{p}{2}} \ w(x) \, dx\right)^{\frac{1}{p}} \ \frac{du}{u}
\\
&\qquad +\int_{\frac{1}{4}}^{\infty}e^{-u}u^{K+\frac{1}{2}}\left(\int_{\mathbb{R}^n}
\left(\int_0^{\infty}\int_{B(x,t)}\left|B_{\frac{t}{2\sqrt{u}},K}f(y)\right|^2 \ \frac{dy \, dt}{t^{n+1}} \right)^{\frac{p}{2}} \ w(x) \, dx\right)^{\frac{1}{p}} \ \frac{du}{u}
\\
& =:I+II.
\end{align*}
For $I$, fix $2<\widetilde{q}<\infty$ and apply Jensen's inequality to the integral in $y$. Then,
\begin{multline}\label{estI:1}
I
\lesssim \int_0^{\frac{1}{4}}u^{K+\frac{1}{2}} \ \left(\int_{\mathbb{R}^n}
\left(\int_0^{\infty}\left(\int_{B(x,t)}\left|B_{\frac{t}{2\sqrt{u}},K}f(y)\right|^{\widetilde{q}} dy \right)^{\frac{2}{\widetilde{q}}} \ \frac{dt}{t^{\frac{2n}{\widetilde{q}}+1}} \right)^{\frac{p}{2}} \ w(x) \, dx\right)^{\frac{1}{p}} \frac{du}{u}
\\
=:
\int_0^{\frac{1}{4}}u^{K+\frac{1}{2}} \ \left(\int_{\mathbb{R}^n} \mathcal{J}(u,x)^p\,w(x)\,dx\right)^{\frac{1}{p}} \frac{du}{u}
\
.
\end{multline}
Fix $0<u<\frac14$. Note that since $1<\frac{\widetilde{q}}{2}<\infty$, for $\alpha:=2\sqrt{u}\in (0,1]$ and $q:=\frac{\widetilde{q}}{2}$ we can apply Proposition \ref{prop:Q} and conclude, for all $1<\frac{\widetilde{q}}{2}\leq s<\infty$  and $w_0\in RH_{s'}$,
\begin{align}\label{KKS}
\int_{\mathbb{R}^n}\mathcal{J}(u,x)^2w_0(x)dx
&=
\int_0^{\infty}\int_{\mathbb{R}^n}\left(\int_{B(x,2\sqrt{u}\frac{t}{2\sqrt{u}})}\left|B_{\frac{t}{2\sqrt{u}},K}f(y)\right|^{\widetilde{q}} dy \right)^{\frac{2}{\widetilde{q}}} \ w_0(x)dx\frac{dt}{t^{\frac{2n}{\widetilde{q}}+1}}
\\
\nonumber
&\lesssim u^{\frac{n}{2s}}
\int_{\mathbb{R}^n}\int_0^{\infty}\left(\int_{B(x,\frac{t}{2\sqrt{u}})}\left|B_{\frac{t}{2\sqrt{u}},K}f(y)\right|^{\widetilde{q}} dy \right)^{\frac{2}{\widetilde{q}}}\frac{dt}{t^{\frac{2n}{\widetilde{q}}+1}} \ w_0(x)dx
\\
\nonumber
&\lesssim u^{\frac{n}{2s}-\frac{n}{\widetilde{q}}}
\int_{\mathbb{R}^n} \int_0^{\infty} \left(\int_{B(x,t)}\left|B_{t,K}f(x)\right|^{\widetilde{q}} dy \right)^{\frac{2}{\widetilde{q}}} \frac{dt}{t^{\frac{2n}{\widetilde{q}}+1}}w_0(x)dx
\\
\nonumber
&=: u^{\frac{n}{2s}-\frac{n}{\widetilde{q}}}
\int_{\mathbb{R}^n} \mathcal{T}(x)^2\ w_0(x)dx,
\end{align}
where in the last inequality we have changed the variable $t$ into $2\sqrt{u}t$.
Assuming further that $2<\widetilde{q}<p_+(L)$ and applying $L^2-L^{\widetilde{q}}$ off-diagonal estimates and Proposition \ref{prop:alpha}, we can bound the last integral above as follows
\begin{align*}
\int_{\mathbb{R}^n} \mathcal{T}(x)^2\ w_0(x)dx
&\lesssim
\int_{\mathbb{R}^n} \int_0^{\infty} \left(\int_{B(x,t)}\left|e^{-\frac{t^2}{2}L}B_{\frac{t}{\sqrt{2}},K}f(x)\right|^{\widetilde{q}} dy \right)^{\frac{2}{\widetilde{q}}} \frac{dt}{t^{\frac{2n}{\widetilde{q}}+1}}w_0(x)dx\nonumber
\\
&\lesssim \sum_{j\geq 1}e^{-c4^j}
\int_{\mathbb{R}^n} \int_0^{\infty} \int_{B(x,2^{j+1}t)}\left|B_{\frac{t}{\sqrt{2}},K}f(x)\right|^2 \frac{dy \, dt}{t^{n+1}}w_0(x)dx\nonumber
\\
&\lesssim \sum_{j\geq 1}e^{-c4^j}
\int_{\mathbb{R}^n} \int_0^{\infty} \int_{B(x,2^{j+1}\sqrt{2}t)}\left|B_{t,K}f(x)\right|^2 \frac{dy \, dt}{t^{n+1}}w_0(x)dx\nonumber
\\
&\lesssim \sum_{j\geq 1}2^{jnr}e^{-c4^j}
\|\Scal_{K,\hh}f\|_{L^2(w_0)}^2\nonumber
\\
&\lesssim
\|\Scal_{K,\hh}f\|_{L^2(w_0)}^2,
\end{align*}
where $r>r_{w_{0}}$. This and \eqref{KKS} yield, for all $2<\widetilde{q}<p_+(L)$, $\frac{\widetilde{q}}{2}\leq s<\infty$, and $w_0\in RH_{s'}$,
\begin{align*}
\int_{\mathbb{R}^n}\left(\mathcal{J}(u,x)^{2s}\right)^{\frac{1}{s}}w_0(x)dx\lesssim
\int_{\mathbb{R}^n}\left(u^{s\gamma(s,\widetilde{q})}\Scal_{K,\hh}f(x)^{2s}\right)^{\frac{1}{s}}w_0(x)dx,
\end{align*}
where $\gamma(s,\widetilde{q}):=\frac{n}{2s}-\frac{n}{\widetilde{q}}.$
Next, we apply Lemma \ref{lemma:extrapol}, part $(b)$, for $q_0:=s$ and for the pairs of functions $\big(\mathcal{J}(u,x)^{2s},u^{s\gamma(s,\widetilde{q})}\Scal_{K,\hh}f(x)^{2s}\big)$. Hence,  for all $2<\widetilde{q}<p_+(L)$, $\frac{\widetilde{q}}{2}\leq s<\infty$, $1<q<\infty$, and $\widetilde{w}\in RH_{q'}$,
\begin{align}\label{KKS1}
\int_{\mathbb{R}^n}\mathcal{J}(u,x)^{\frac{2s}{q}} \widetilde{w}(x)dx \lesssim u^{\frac{s\gamma(s,\widetilde{q})}{q}}
\int_{\mathbb{R}^n}\Scal_{K,\hh}f(x)^{\frac{2s}{q}} \widetilde{w}(x)dx.
\end{align}

\medskip

We now distinguish two cases. Assume first that $n\leq (2K+1)p_+(L)$. Under this assumption, for every $0<p<\infty$ and $w\in A_{\infty}$, we take $s>s_{w}\max\left\{\frac{p}{2},1\right\}$, $\max\left\{2,\frac{2sp_+(L)}{p_+(L)+2s}\right\}<\widetilde{q}<\min\left\{p_+(L),2s\right\}$, (if $p_+(L)=\infty$ take $\widetilde{q}:=2s$),
and $q:=\frac{2s}{p}.$ Then, we have that $2<\widetilde{q}<p_+(L)$, $\frac{\widetilde{q}}{2}\leq s<\infty$, $1\leq s_{w}<q<\infty$, and $w\in RH_{q'}$. Hence, applying \eqref{KKS1}, we obtain
\begin{align}\label{estI:2}
\int_{\mathbb{R}^n}\mathcal{J}(u,x)^{p} w(x)dx\lesssim u^{\frac{np}{4s}-\frac{np}{2\widetilde{q}}}
\int_{\mathbb{R}^n}\Scal_{K,\hh}f(x)^{p} w(x)dx.
\end{align}
Besides, note that from our choices of $s$ and $\widetilde{q}$, we have that
\begin{align*}
K+\frac{1}{2}+\frac{n}{4s}-\frac{n}{2\widetilde{q}}>K+\frac{1}{2}-\frac{n}{2p_+(L)}\geq 0.
\end{align*}
Consequently, plugging \eqref{estI:2} into \eqref{estI:1} we obtain
\begin{align}\label{FT}
I\lesssim \int_{0}^{\frac{1}{4}} u^{K+\frac{1}{2}+\frac{n}{4s}-\frac{n}{2\widetilde{q}}}\frac{du}{u}\,
\|\Scal_{K,\hh}f\|_{L^p(w)}
\lesssim \|\Scal_{K,\hh}f\|_{L^p(w)}.
\end{align}
Consider now the case $n>(2K+1)p_+(L)$. Fix $w\in A_{\infty}$ and $p\in \mathcal{W}_w(0,p_+(L)^{K,*})$. Then $w\in RH_{\left(\frac{p_+(L)^{K,*}}{p}\right)'}$ and $0<p<\frac{p_+(L)n}{s_w(n-(2K+1)p_+(L))}$. Therefore,
it is possible to pick $\varepsilon_1>0$ small enough and $2<\widetilde{q}<p_+(L)$ so that
$$
 0<p<\frac{\widetilde{q}n}{s_{w}(1+\varepsilon_1)(n-(2K+1)\widetilde{q})}.
$$
Besides, since $
\widetilde{q}<\widetilde{q}n/ (n-(2K+1)\widetilde{q})$ there also exists $\varepsilon_2>0$ so that
$$
\widetilde{q}<\frac{\widetilde{q}n}{(1+\varepsilon_2)(n-(2K+1)\widetilde{q})}.
$$
Take $\varepsilon_0:=\min\{\varepsilon_1,\varepsilon_2\}$, $s:=\frac{\widetilde{q}n}{2(1+\varepsilon_0)(n-(2K+1)\widetilde{q})}$,
and $q:=\frac{2s}{p}$. Then our choices guarantee that $2<\widetilde{q}<p_+(L)$, $\frac{\widetilde{q}}{2}\leq s<\infty$, $1\leq s_{w}<q<\infty$, and $w\in RH_{q'}$. Therefore, we can apply \eqref{KKS1} and obtain
\begin{align}\label{estI:3}
\int_{\mathbb{R}^n}\mathcal{J}(u,x)^{p} w(x)dx\lesssim u^{\frac{np}{4s}-\frac{np}{2\widetilde{q}}}
\int_{\mathbb{R}^n}\Scal_{K,\hh}f(x)^{p} w(x)dx.
\end{align}
Again, our choices of $s$ and $\widetilde{q}$ imply
\begin{align*}
K+\frac{1}{2}+\frac{n}{4s}-\frac{n}{2\widetilde{q}}
=
\varepsilon_0\left(\frac{n}{2\widetilde{q}}-K-\frac{1}{2}\right)>\varepsilon_0\left(\frac{n}{2p_+(L)}-K-\frac{1}{2}\right)>0.
\end{align*}
This, \eqref{estI:1}, and \eqref{estI:3} give
\begin{align*}
I\lesssim \int_{0}^{\frac{1}{4}} u^{K+\frac{1}{2}+\frac{n}{4s}-\frac{n}{2\widetilde{q}}}\frac{du}{u}\,
\|\Scal_{K,\hh}f\|_{L^p(w)}
\lesssim \|\Scal_{K,\hh}f\|_{L^p(w)}.
\end{align*}
Gathering this estimate and \eqref{FT}, we conclude that, for all $w\in A_{\infty}$ and $p\in \mathcal{W}_w(0,p_+(L)^{K,*})$ with $1\le p$,
$$
I\lesssim \|\Scal_{K,\hh}f\|_{L^p(w)}.
$$

\medskip

We finally estimate $II$. Take $F(y,t)=B_{t,K}f(y)$ and pick $r\geq\max\{\frac{p}{2},r_w\}$ so that $w\in A_r$ and $1\le p\leq 2r$. Hence, we have
\begin{align*}
II
\lesssim\int_{\frac{1}{4}}^{\infty}e^{-u}u^{K+\frac{1}{2}-\frac{n}{4}} \|\mathcal{A}^{2\sqrt{u}}F\|_{L^p(w)} \frac{du}{u}
\lesssim\int_{\frac{1}{4}}^{\infty}e^{-u}u^{K+\frac{1}{2}+\frac{nr}{p}-\frac{n}{4}} \|\mathcal{A}F\|_{L^p(w)} \frac{du}{u}
\lesssim \|\Scal_{K,\hh}f\|_{L^p(w)}.
\end{align*}
This estimate, together with the one obtained for $I$ and the observations made at the beginning of the proof, allows us to finish the proof. \qed

\subsubsection{Proof of Theorem \ref{theor:control-SF-Poisson}, parts $(c)$ and $(d)$}

%%%%%%%%%%%%%%%%%%%%%%%%%%%%%%%%%%%%%%%%%%%%%%%%%%%%%%%%%%%%%%
%%%%%%%%%%%%%%%%%%%%%%%%%%%%%%%%%%%%%%%%%%%%%%%%%%%%%%%%%%%%%%

We first invoke \cite[Lemma $3.5$]{AuscherHofmannMartell}:  for every $K\in \N_0$, $f\in L^2(\mathbb{R}^n)$ and $x\in \mathbb{R}^n$, there holds
\begin{align}\label{GKP}
\Gcal_{K,\pp}f(x)&\lesssim K\left(\int_0^{\infty}\int_{B(x,2t)}|(t^2L)^K e^{-t^2L}f(y)|^2 \frac{dy \, dt}{t^{n+1}}\right)^{\frac{1}{2}}
\\
  \nonumber
&\qquad\quad + \left(\int_0^{\infty}\int_{B(x,2t)}|t\nabla_{y,t}(t^2L)^K e^{-t^2L}f(y)|^2 \frac{dy \, dt}{t^{n+1}}\right)^{\frac{1}{2}}
 \\
 \nonumber
 &\qquad\quad+\left(\int_0^{\infty}\int_{B(x,2t)}|(t^2L)^K(e^{-t\sqrt{L}}
 -e^{-t^2L})f(y)|^2 \frac{dy \, dt}{t^{n+1}}\right)^{\frac{1}{2}}.
\end{align}
The first and second term in the right-hand side of the above inequality will be easily controlled in $L^p(w)$, applying Proposition \ref{prop:alpha}, by $\Scal_{K,\hh}f$ and $\Gcal_{K,\hh}f$ respectively. So, we just need to deal with the third term. To this end we define
\begin{align*}
\mathfrak{G}_{K,\pp}f(x):=\left(\int_0^{\infty}\int_{B(x,2t)}|(t^2L)^K(e^{-t\sqrt{L}}
 -e^{-t^2L})f(y)|^2 \frac{dy \, dt}{t^{n+1}}\right)^{\frac{1}{2}}.
\end{align*}
We claim that for all $K\in \N_0$, $w\in A_{\infty}$, and $p\in \mathcal{W}_w(0,p_+(L)^{K,*})$, the following estimate holds:
\begin{align}\label{est:gfrak}
\|\mathfrak{G}_{K,\pp}f\|_{L^p(w)}\lesssim \|\Scal_{K+1,\hh}f\|_{L^p(w)} .
\end{align}
Assuming this momentarily and applying Proposition \ref{prop:alpha} to the first two terms in the right-hand side of \eqref{GKP}
we conclude, for all $K\in \N_0$, $w\in A_{\infty}$, and $p\in \mathcal{W}_{w}(0,p_+(L)^{K,*})$,
\begin{align}\label{2GKP}
\|\Gcal_{K,\pp}f\|_{L^p(w)}\lesssim K\|\Scal_{K,\hh}f\|_{L^p(w)}+\|\Gcal_{K,\hh}f\|_{L^p(w)}+\|\Scal_{K+1,\hh}f\|_{L^p(w)}.
\end{align}
For $K\in \N$, apply Theorem\ref{theor:control-SF-Heat} parts $(b)$ and $(c)$. This proves part $(d)$. To obtain part $(c)$, we take $K=0$
in \eqref{2GKP}. Note that clearly $\Scal_{\hh}f\le \frac12\Gcal_{\hh}f$ and therefore
$$
\|\Gcal_{\pp}f\|_{L^p(w)}\lesssim \|\Gcal_{\hh}f\|_{L^p(w)}+\|\Scal_{\hh}f\|_{L^p(w)}\lesssim\|\Gcal_{\hh}f\|_{L^p(w)}.
$$

To complete the proof we need to obtain \eqref{est:gfrak}. Note that in view of Lemma \ref{lemma:extrapol} it is enough to prove \eqref{est:gfrak} for $p\geq 1$. Given $1\leq p<\infty$, apply the subordination formula \eqref{FR} and Minkowski's inequality:
\begin{multline*}
\|\mathfrak{G}_{K,\pp}f\|_{L^p(w)}
\lesssim \int_0^{\infty}e^{-u}u^{\frac{1}{2}}\left(\int_{\mathbb{R}^n} \left(\int_0^{\infty}\!\!\int_{B(x,2t)}|(t^2L)^K(e^{-\frac{t^2}{4u}L}
 -e^{-t^2L})f(y)|^2 \frac{dy \, dt}{t^{n+1}}\right)^{\frac{p}{2}}w(x)dx\right)^{\frac{1}{p}}\frac{du}{u}
\\
=: \int_0^{\infty}e^{-u}u^{\frac{1}{2}} F(u)\frac{du}{u}
\le \int_0^{\frac{1}{4}}u^{\frac{1}{2}}F(u)\frac{du}{u} +\int_{\frac{1}{4}}^{\infty}e^{-u}u^{\frac{1}{2}}F(u)\frac{du}{u}
=:I+II.
\end{multline*}

Fix $0<u<\frac{1}{4}$, and note that
\begin{align*}
\big|(e^{-\frac{t^2}{4u}L}
 -e^{-t^2L})f\big|
 \lesssim
 \int_{t}^{\frac{t}{2\sqrt{u}}}\big|r^2Le^{-r^2L}f\big|\frac{dr}{r}.
\end{align*}
We set $H_K(y,r):=(r^2L)^{K+1}  e^{-r^2L}f(y)$. Using the previous estimate and applying Minkowski's and Jensen's inequalities, it follows that
\begin{align*}
F(u)&\lesssim \left(\int_{\mathbb{R}^n} \left(\int_0^{\infty}\left(\int_{t}^{\frac{t}{2\sqrt{u}}}\left(\int_{B(x,2t)}|(t^2L)^K
 r^2L
 e^{-r^2L}f(y)|^{2}dy\right)^{\frac{1}{2}} \ \frac{dr}{r}\right)^2\frac{ \, dt}{t^{n+1}}\right)^{\frac{p}{2}}w(x)dx\right)^{\frac{1}{p}}
 \\
 &\lesssim u^{-\frac{1}{4}}\left(\int_{\mathbb{R}^n} \left(\int_0^{\infty}\int_{t}^{\frac{t}{2\sqrt{u}}}\int_{B(x,2t)}
| H_K(y,r)|^2 \Big(\frac{t}{r}\Big)^{4\,K}dy \ \frac{dr}{r^2}\frac{ \, dt}{t^{n}}\right)^{\frac{p}{2}} w(x)dx\right)^{\frac{1}{p}}
 \\
  &\lesssim u^{-\frac{1}{4}}\left(\int_{\mathbb{R}^n} \left(\int_0^{\infty}\int_{2\sqrt{u}r}^{r}\int_{B(x,2t)}
 | H_K(y,r)|^2 \Big(\frac{t}{r}\Big)^{4\,K}dy \frac{dt}{t^{n}}\frac{dr}{r^2}\right)^{\frac{p}{2}}
w(x)dx\right)^{\frac{1}{p}}.
\end{align*}
Take $2<\widetilde{q}<\infty$, apply Jensen's inequality to the integral in $y$, and change the variable $t$ into $rt$ to obtain
\begin{align*}
F(u) & \lesssim u^{-\frac{1}{4}}
\left(\int_{\mathbb{R}^n} \left(\int_0^{\infty}\int_{2\sqrt{u}r}^r
  \left(\int_{B(x,2t)} | H_K(y,r)|^{\widetilde{q}} dy\right)^{\frac{2}{\widetilde{q}}}  \Big(\frac{t}{r}\Big)^{4\,K} \frac{dt}{t^{\frac{2n}{\widetilde{q}}}}\frac{dr}{r^{2}}\right)^{\frac{p}{2}}
 w(x)dx\right)^{\frac{1}{p}}
 \\
 &\lesssim u^{-\frac{1}{4}}
\left(\int_{\mathbb{R}^n} \left(\int_0^{\infty}\int_{2\sqrt{u}}^{1}
  \left(\int_{B(x,2rt)}|H_K(y,r)|^{\widetilde{q}}dy\right)^{\frac{2}{\widetilde{q}}} t^{4K} \frac{dt}{t^{\frac{2n}{\widetilde{q}}}}\frac{dr}{r^{\frac{2n}{\widetilde{q}}+1}}\right)^{\frac{p}{2}}
 w(x)dx\right)^{\frac{1}{p}}
\\
&=:
 u^{-\frac{1}{4}}
\left(\int_{\mathbb{R}^n}
\widehat{H}(x,u)^p
w(x)dx\right)^{\frac{1}{p}}
 .
 \end{align*}
%To simplify the notation, we set
%\begin{align*}
%\widetilde{H}_K(x)&:= \left(\int_0^{\infty}
%  \left(\int_{B(x,2r)}|H_K(y,r)|^{\widetilde{q}}dy\right)^{\frac{2}{\widetilde{q}}} \frac{dr}{r^{\frac{2n}{\widetilde{q}}+1}}\right)^{\frac{1}{2}},
%  \\
%\widehat{H}(x,u)&:= \left(\int_0^{\infty}\int_{2\sqrt{u}}^{1}
%  \left(\int_{B(x,2rt)}|H_K(y,r)|^{\widetilde{q}}dy\right)^{\frac{2}{\widetilde{q}}} t^{4K} \frac{dt}{t^{\frac{2n}{\widetilde{q}}}}\frac{dr}{r^{\frac{2n}{\widetilde{q}}+1}}\right)^{\frac{1}{2}}.
%\end{align*}
Note that $1<\frac{\widetilde{q}}{2}$. Then for  $\alpha:=t\in (0,1)$ and $q:=\frac{\widetilde{q}}{2}$ we can apply Proposition \ref{prop:Q} and obtain for all $\frac{\widetilde{q}}{2}\leq s<\infty$
and $w_0\in RH_{s'}$
\begin{align}\label{HG}
\int_{\mathbb{R}^n}\widehat{H}(x,u)^2 w_0(x) dx
&=\int_{2\sqrt{u}}^1\int_{0}^{\infty}\int_{\mathbb{R}^n}\left(\int_{B(x,2rt)}|H_K(y,r)|^{\widetilde{q}} dy\right)^{\frac{2}{\widetilde{q}}}w_0(x)dx\frac{dr}{r^{\frac{2n}{\widetilde{q}}+1}}t^{4K} \frac{dt}{t^{\frac{2n}{\widetilde{q}}}}
\\
\nonumber
&\lesssim \int_{2\sqrt{u}}^1t^{4K-\frac{2n}{\widetilde{q}}+\frac{n}{s}+1} \frac{dt}{t}\int_{0}^{\infty}\int_{\mathbb{R}^n}\left(\int_{B(x,2r)}|H_K(y,r)|^{\widetilde{q}} dy\right)^{\frac{2}{\widetilde{q}}}w_0(x)dx\frac{dr}{r^{\frac{2n}{\widetilde{q}}+1}}
\\
\nonumber
&=C(u,s,\widetilde{q})\int_{\mathbb{R}^n}\widetilde{H}_K(x)^2w_0(x)dx,
\end{align}
where
\begin{align*}
C(u,s,\widetilde{q}):=
\int_{2\sqrt{u}}^1 t^{\theta}\frac{dt}{t}
:=
\int_{2\sqrt{u}}^1t^{4K-\frac{2n}{\widetilde{q}}+\frac{n}{s}+1}\frac{dt}{t}
\end{align*}
and
$$
\widetilde{H}_K(x):= \left(\int_0^{\infty}
  \left(\int_{B(x,2r)}|H_K(y,r)|^{\widetilde{q}}dy\right)^{\frac{2}{\widetilde{q}}} \frac{dr}{r^{\frac{2n}{\widetilde{q}}+1}}\right)^{\frac{1}{2}}.
$$
If we further assume  $\widetilde{q}< p_+(L)$, then  $\{(r^2L)^{K+1}e^{-r^2L}\}_{r>0}\in\mathcal{F}_\infty(L^2\rightarrow L^{\widetilde{q}})$. Use that  $H_K(y,r)=2^{K+1}e^{-\frac{r^2}2 L}\,H_K\big(y,\tfrac{r}{\sqrt{2}}\big)$ and  apply Proposition \ref{prop:alpha} to obtain
\begin{align*}
\int_{\mathbb{R}^n}\widetilde{H}_K(x)^2w_0(x)dx&\lesssim \sum_{j\geq 1}e^{-c4^j}\int_{\mathbb{R}^n}\int_0^{\infty}\int_{B(x,2^{j+2}r)}\big|H_K\big(y,\tfrac{r}{\sqrt{2}}\big)\big|^2
\frac{dy \, dr}{r^{n+1}}w_0(x)dx
\\
&\lesssim \sum_{j\geq 1}e^{-c4^j}\int_{\mathbb{R}^n}\int_0^{\infty}\int_{B(x,2^{j+2}\sqrt{2}r)}\big|H_K(y,r)\big|^2
\frac{dy \, dr}{r^{n+1}}w_0(x)dx
\\
&\lesssim \sum_{j\geq 1}2^{jnr}e^{-c4^j}\int_{\mathbb{R}^n}\Scal_{K+1,\hh}f(x)^2w_0(x)dx
\\
&\lesssim \int_{\mathbb{R}^n}\Scal_{K+1,\hh}f(x)^2w_0(x)dx,
\end{align*}
where $r> r_{w_0}$ is so that $w_0\in A_r$ and $0<2\leq 2r.$
This, together with \eqref{HG}, yields for all $2<\widetilde{q}<p_+(L)$, $\frac{\widetilde{q}}{2}\leq s<\infty$, and $w_0\in RH_{s'}$
\begin{align*}
\int_{\mathbb{R}^n}\left(\widehat{H}(x,u)^{2s}\right)^{\frac{1}{s}} w_0(x) dx\lesssim \int_{\mathbb{R}^n}\left(C(u,s,\widetilde{q})^{s}\Scal_{K+1,\hh}f(x)^{2s}\right)^{\frac{1}{s}}w_0(x)dx.
\end{align*}
Applying Lemma \ref{lemma:extrapol}, part $(b)$, to the pairs of functions $\big(\widehat{H}(x,u)^{2s},C(u,s,\widetilde{q})^{s}\Scal_{K+1,\hh}f(x)^{2s}\big)$ and with $q_0:=s$, we have for all $2<\widetilde{q}<p_+(L)$, $\frac{\widetilde{q}}{2}\leq s<\infty$, $1<q<\infty$, and $\widetilde{w}\in RH_{q'}$
\begin{align}\label{CU2}
\int_{\mathbb{R}^n}\widehat{H}(x,u)^{\frac{2s}{q}}\widetilde{w}(x) dx\lesssim C(u,s,\widetilde{q})^{\frac{s}{q}} \int_{\mathbb{R}^n}\Scal_{K+1,\hh}f(x)^{\frac{2s}{q}}\widetilde{w}(x)dx.
\end{align}
Consider now two cases:  $n\leq (2K+1)p_+(L)$ and $n> (2K+1)p_+(L)$.

\medskip

Assume first that $n\leq (2K+1)p_+(L)$. Fix $0<p<\infty$ and $w\in A_{\infty}$. Take $s>s_w\max\left\{\frac{p}{2},1\right\}$, $\max\left\{2,\frac{2sp_+(L)}{p_+(L)+2s}\right\}<\widetilde{q}<\min\left\{p_+(L),2s\right\}$, (if $p_+(L)=\infty$ take $\widetilde{q}:=2s$),
and $q:=\frac{2s}{p}.$ Then, $2<\widetilde{q}<p_+(L)$, $\frac{\widetilde{q}}{2}\leq s<\infty$, $1\leq s_{w}<q<\infty$, and $w\in RH_{q'}$. Hence, \eqref{CU2} yields
\begin{align*}
F(u)\lesssim u^{-\frac14}\,\left(\int_{\mathbb{R}^n}\widehat{H}(x,u)^{p}w(x) dx\right)^{\frac1p}
\lesssim
u^{-\frac14}\,
C(u,s,\widetilde{q})^{\frac{1}{2}}\|\Scal_{K+1,\hh}f(x)\|_{L^{p}(w)}.
\end{align*}
Note that from our choices of $s$ and $\widetilde{q}$ we have
$$
\theta:=4K-\frac{2n}{\widetilde{q}}+\frac{n}{s}+1>
4K-\frac{2n}{p_+(L)}+1>-1.
$$
Therefore, taking $-1<\widetilde{\theta}<\min\{\theta,0\}$,
we obtain
$$
C(u,s,\widetilde{q})=\int_{2\sqrt{u}}^{1}t^{\theta}\frac{dt}{t}\leq
\int_{2\sqrt{u}}^{1}t^{\widetilde{\theta}}\frac{dt}{t}\lesssim u^{\frac{\widetilde{\theta}}{2}},
$$
and hence
\begin{multline}\label{GSK}
I=
\int_{0}^{\frac{1}{4}}u^{\frac{1}{2}}F(u)\frac{du}{u}
\lesssim
\int_{0}^{\frac{1}{4}}u^{\frac{1}{4}}C(u,s,\widetilde{q})^{\frac{1}{2}}\frac{du}{u}\,\|\Scal_{K+1,\hh}f\|_{L^p(w)}
\\
\lesssim \int_{0}^{\frac{1}{4}}u^{\frac{1+\widetilde{\theta}}{4}}\frac{du}{u}\,\|\Scal_{K+1,\hh}f\|_{L^p(w)}
\lesssim \|\Scal_{K+1,\hh}f\|_{L^p(w)}.
\end{multline}

We next consider the case $n>(2K+1)p_+(L)$. Fix $w\in A_{\infty}$ and $p\in \mathcal{W}_w(0,p_+(L)^{K,*})$. Then $w\in RH_{\left(\frac{p_+(L)^{K,*}}{p}\right)'}$ and $0<p<\frac{p_+(L)n}{s_w(n-(2K+1)p_+(L))}$. There exists $\varepsilon_1>0$ small enough and $2<\widetilde{q}<p_+(L)$, such that
$$
0<p<\frac{\widetilde{q}n}{s_{w}(1+\varepsilon_1)(n-(2K+1)\widetilde{q})}.
$$
Besides, since $\widetilde{q}<\widetilde{q}n/(n-(2K+1)\widetilde{q})$,
there also exists $\varepsilon_2>0$ such that
$$
\widetilde{q}<\frac{\widetilde{q}n}{(1+\varepsilon_2)(n-(2K+1)\widetilde{q})}.
$$
Pick $\varepsilon_0:=\min\{\varepsilon_1,\varepsilon_2\}$, $s:=\frac{\widetilde{q}n}{2(1+\varepsilon_0)(n-(2K+1)\widetilde{q})}$,
and $q:=\frac{2s}{p}.$ Then, $2<\widetilde{q}<p_+(L)$, $\frac{\widetilde{q}}{2}\leq s<\infty$, $1\leq s_{w}<q<\infty$, and $w\in RH_{q'}$. Therefore, \eqref{CU2} yields
\begin{align*}
F(u)\lesssim u^{-\frac14}
\left(\int_{\mathbb{R}^n}\widehat{H}(x,u)^{p}w(x) dx\right)^{\frac1p}\lesssim
u^{-\frac14} C(u,s,\widetilde{q})^{\frac{1}{2}}\|\Scal_{K+1,\hh}f\|_{L^p(w)}.
\end{align*}
Once more, from our choices of $s$ and $\widetilde{q}$ we have
 $$
 \theta:=4K-\frac{2n}{\widetilde{q}}+\frac{n}{s}+1=-1+2\varepsilon_0\left(\frac{n}{\widetilde{q}}-(2K+1)\right)>-1+2\varepsilon_0\left(\frac{n}{p_+(L)}-(2K+1)\right)>-1.
 $$
Hence, taking $-1<\widetilde{\theta}<\min\{\theta,0\}$,
we obtain
\begin{align*}
C(u,s,\widetilde{q})=\int_{2\sqrt{u}}^{1}t^{\theta}\frac{dt}{t}\leq
\int_{2\sqrt{u}}^1t^{\widetilde{\theta}}\frac{dt}{t}\lesssim u^{\frac{\widetilde{\theta}}{2}}.
\end{align*}
Therefore,
\begin{multline*}
I=
\int_{0}^{\frac{1}{4}}u^{\frac{1}{2}}F(u)\frac{du}{u}
\lesssim
\int_{0}^{\frac{1}{4}}u^{\frac{1}{4}}C(u,s,\widetilde{q})^{\frac{1}{2}}\frac{du}{u}\,\|\Scal_{K+1,\hh}f\|_{L^p(w)}
\\
\lesssim \int_{0}^{\frac{1}{4}}u^{\frac{1+\widetilde{\theta}}{4}}\frac{du}{u}\,\|\Scal_{K+1,\hh}f\|_{L^p(w)}
\lesssim \|\Scal_{K+1,\hh}f\|_{L^p(w)}.
\end{multline*}
This and \eqref{GSK} give $I\lesssim \|\Scal_{K+1,\hh}f\|_{L^p(w)}$ for all $w\in A_{\infty}$ and $p\in \mathcal{W}_{w}(0,p_+(L)^{K,*})$ with $p\ge 1$.

%%%%%%%%%%%%%%%%%%%%%%%%%%%%%%%%%%%%%%%%%%%%%%%%%%%%%%%%%%%%%%%%%%%%%%%%%%%%%%%%%%
%%%%%%%%%%%%%%%%%%%%%%%%%%%%%%%%%%%%%%%%%%%%%%%%%%%%%%%%%%%%%%%%%%%%%%%%%%%%%%%%%%
%%%%%%%%%%%%%%%%%%%%%%%%%%%%%%%%%%%%%%%%%%%%%%%%%%%%%%%%%%%%%%%%%%%%%%%%%%%%%%%%%%

\medskip
To estimate $II$ we fix $\frac{1}{4}\leq u<\infty$ and observe that
\begin{align*}
\left|(e^{-\frac{t^2}{4u}L}
 -e^{-t^2L})f\right|\lesssim \int_{\frac{t}{2\sqrt{u}}}^{t}\left|r^2Le^{-r^2L}f\right|\frac{dr}{r}.
\end{align*}
Set $T_{r^2,K}:=(r^2L)^{K+1}e^{-r^2L}$ and pick $\widetilde{r}\geq \max\{\frac{p}{2},r_w\}$ so that $w\in A_{\widetilde{r}}$ and $1\le p\leq 2\widetilde{r}$. Then, applying Jensen's inequality, Fubini, the fact that we are integrating in $t<2\sqrt{u}r$, and Proposition \ref{prop:alpha}, we have
\begin{align*}
F(u)&\lesssim
\left(\int_{\mathbb{R}^n} \left(\int_{0}^{\infty}\left(\int_{\frac{t}{2\sqrt{u}}}^{t}\left(\int_{B(x,2t)}
|(t^2L)^KT_{r^2,0}
f(y)|^2 dy\right)^{\frac{1}{2}} \frac{dr}{r}\right)^2\frac{dt}{t^{n+1}}\right)^{\frac{p}{2}}w(x)dx\right)^{\frac{1}{p}}
\\
&
\lesssim
\left(\int_{\mathbb{R}^n} \left(\int_{0}^{\infty}
\int_{\frac{t}{2\sqrt{u}}}^{t}\int_{B(x,2t)}|(t^2L)^KT_{r^2,0}
f(y)|^2 dy \frac{dr}{r^2}\frac{dt}{t^{n}}\right)^{\frac{p}{2}}w(x)
dx\right)^{\frac{1}{p}}
\\
&
=
\left(\int_{\mathbb{R}^n} \left(\int_{0}^{\infty}
\int_r^{2\sqrt{u}r}\int_{B(x,2t)}|(t^2L)^KT_{r^2,0}
f(y)|^2 dy\frac{dt}{t^{n}}\frac{dr}{r^2}\right)^{\frac{p}{2}}w(x)
dx\right)^{\frac{1}{p}}
\\
&
\lesssim u^K \left(\int_{\mathbb{R}^n} \left(\int_{0}^{\infty}
\int_r^{2\sqrt{u}r}\int_{B(x,4\sqrt{u}r)}|T_{r^2,K}
f(y)|^2 dy \, dt\frac{dr}{r^{n+2}}\right)^{\frac{p}{2}}w(x)
dx\right)^{\frac{1}{p}}
\\
&\lesssim
u^{K+\frac{1}{4}}\left(\int_{\mathbb{R}^n} \left(\int_{0}^{\infty}\int_{B(x,4\sqrt{u}r)}|(r^2L)^{K+1}e^{-r^2L}
f(y)|^2 dy\frac{dr}{r^{n+1}}\right)^{\frac{p}{2}}w(x)
dx\right)^{\frac{1}{p}}
\\
&\lesssim u^{K+\frac{1}{4}+\frac{n\widetilde{r}}{2p}}\left(\int_{\mathbb{R}^n} \left(\int_{0}^{\infty}\int_{B(x,r)}|(r^2L)^{K+1}e^{-r^2L}
f(y)|^2 dy\frac{dr}{r^{n+1}}\right)^{\frac{p}{2}}w(x)
dx\right)^{\frac{1}{p}}
\\
&
=
u^{K+\frac{1}{4}+\frac{n\widetilde{r}}{2p}}\,\|\Scal_{K+1,\hh}f\|_{L^p(w)}
.
\end{align*}
Hence,
$$
II
=
\int_{\frac14}^{\infty} e^{-u}u^{\frac{1}{2}}F(u)\frac{du}{u}
\lesssim
\int_{\frac14}^{\infty} e^{-u}
u^{K+\frac{3}{4}+\frac{n\widetilde{r}}{2p}}
\frac{du}{u}\,\|\Scal_{K+1,\hh}f\|_{L^p(w)}
\lesssim
\|\Scal_{K+1,\hh}f\|_{L^p(w)}
$$
Gathering this estimate and the one obtained for $I$, from the observations made above, the proof of \eqref{est:gfrak} is complete.
\qed

\begin{remark}\label{remark:classesoffunctions}
We note that Theorems \ref{theor:control-SF-Heat} and \ref{theor:control-SF-Poisson} are restricted to functions $f\in L^2(\R^n)$. However, an inspection of the proof and a routine and tedious density argument allow us to extend these estimates to bigger classes of functions. For instance, we can take any function $f\in L^q(\widetilde{w})$ with $\widetilde{w}\in A_\infty$ and $q\in \mathcal{W}_{\widetilde{w}}(p_-(L), p_+(L))$. In that range the  Heat and the Poisson semigroups are uniformly bounded and satisfy off-diagonal estimates, hence the square functions under study are meaningfully defined. Moreover, $L$ has a bounded holomorphic functional calculus on $L^q(\widetilde{w})$ (see \cite{AuscherMartell:I}, \cite{AuscherMartell:II}, and \cite{AuscherMartell:III}). Further details are left to the interested reader.
\end{remark}


\begin{thebibliography}{10}



\bibitem{Auscher}
P. Auscher, {\em On necessary and sufficient conditions for $L^p$-estimates of Riesz transforms associated to elliptic operators on
$\mathbb{R}^n$ and related estimates}, Mem. Amer. Math. Soc. {\bf 186} (2007), no.
871.


\bibitem{Auscherangles}
P. Auscher, {\em Change of Angles in tent spaces,}
C. R. Math. Acad. Sci. Paris {\bf 349} (2011), no. 5-6, 297--301.

\bibitem{AHLMT} P. Auscher, S. Hofmann, M. Lacey, A. McIntosh, Ph. Tchamitchian, {\em The solution of the Kato square root
problem for second order elliptic operators on $\R^n$}, Ann. of Math. (2) \textbf{156} (2002), 633--654.

\bibitem{AuscherHofmannMartell}
P. Auscher, S. Hofmann, J.M. Martell, {\em Vertical versus conical square functions}. Trans. Amer. Math. Soc. \textbf{364} (2012), no. 10, 5469--5489.

\bibitem{AuscherMartell:I}
P. Auscher, J.M. Martell, {\em Weighted norm inequalities, off-diagonal estimates and elliptic operators. Part I: General operator theory and weihts},
Adv. Math.  \textbf{212}  (2007),  no. 1, 225--276.

\bibitem{AuscherMartell:II}
P. Auscher, J.M. Martell, {\em Weighted norm inequalities,
off-diagonal estimates and elliptic operators. Part II: off-diagonal
estimates on spaces of homogeneous type}, J. Evol. Equ.  \textbf{7}  (2007),  no. 2, 265--316.

\bibitem{AuscherMartell:III}
P. Auscher, J.M. Martell,
{\em Weighted norm inequalities, off-diagonal estimates
and elliptic operators.
Part III: harmonic analysis of elliptic operators}, J. Funct. Anal.  \textbf{241}  (2006),  no.~2, 703--746.



\bibitem{Auscher-McIntosh-Russ}  P. Auscher, A. McIntosh, E. Russ, {\em Hardy spaces of differential forms on
Riemannian manifolds}, J. Geom. Anal. {\bf 18} (2008), no. 1, 192--248.

\bibitem{Auscher-McIntosh-Morris} P. Auscher, A. McIntosh, A. Morris, {\em Calder\'on Reproducing Formulas and Applications to Hardy Spaces}. Preprint (2013).

\bibitem{Auscher-Russ} P. Auscher, E. Russ, {\em Hardy spaces and divergence operators on strongly Lipschitz domain of $R^n$}, J. Funct. Anal. {\bf 201} (2003), no. 1, 148--184.

\bibitem{BK1} S. Blunck, P. Kunstmann, {\em Calder\'on-Zygmund theory
for non-integral operators and the $H^\infty$-functional calculus},
Rev. Mat. Iberoamericana  \textbf{19}  (2003),  no. 3, 919--942.

\bibitem{BCKYY1} T. A. Bui, J. Cao, L. D. Ky, D. Yang, S. Yang, {\em Weighted Hardy spaces associated with operators satisfying reinforced off-diagonal estimates}, Taiwanese J. Math. {\bf 17} (2013), no. 4, 1127--1166.
 
\bibitem{BCKYY2}  T. A. Bui, J. Cao, L. D. Ky, D. Yang, S. Yang,  {\em Musielak-Orlicz-Hardy spaces associated with operators satisfying reinforced off-diagonal estimates}, Anal. Geom. Metr. Spaces 1 (2013), 69--129.

\bibitem{Anh-Duong}  T. A. Bui, X. T. Duong, {\em Weighted Hardy spaces associated to operators and boundedness of
singular integrals}. Preprint (2012). arXiv:1202.2063.





\bibitem{CoifmanMeyerStein}
R.R. Coifman, Y. Meyer, E.M. Stein, {\em Some new function spaces
and their applications to harmonic analysis}, J. Funct. Anal. {\bf 62} (1985), no. 2, 304--335.

\bibitem{CruzMartellPerez}
D.V. Cruz-Uribe, J.M. Martell, C. Perez {\em Weights extrapolation and the theory of Rubio de Francia}, Operator Theory: Advances and Applications, {\bf 215}. Birkh\"auser/Springer Basel AG, Basel, (2011).

\bibitem{Duo} J.~Duoandikoetxea, {\em Fourier Analysis},
Grad. Stud. Math. {\bf 29}, American Math.  Soc., Providence, RI, (2001).

\bibitem{Duong-Yan-BMO} X.T. Duong, L.X. Yan, {\em
New function spaces of BMO type, the John-Nirenberg inequality, interpolation, and applications},
Comm. Pure Appl. Math. {\bf 58} (2005), no. 10, 1375--1420.

\bibitem{Duong-Yan-H1} X.T. Duong, L.X. Yan, {\em Duality of Hardy and BMO spaces associated with operators with heat
kernel bounds}, J. Amer. Math. Soc. {\bf 18} (2005), no. 4, 943--973.

\bibitem{Fef-Ste}C. Fefferman, E.M. Stein, {\em $H^p$ spaces of several variables}. Acta Math. {\bf 129} (1972), no. 3-4, 137--193.


\bibitem{GC-w-HP} J. Garcia-Cuerva, {\em Weighted $H^p$ spaces},  Dissertationes Math. (Rozprawy Mat.) {\bf 162} (1979).

\bibitem{GC-extrapol} J. Garc{\'\i}a-Cuerva,
{\em An extrapolation theorem in the theory of $A_p$ weights},
Proc. Amer. Math. Soc. \textbf{87} (1983), no. 3, 422--426.

\bibitem{GCRF85} J.\,Garc{\'\i}a-Cuerva and J.\,Rubio de Francia, {\em Weighted norm
inequalities and related topics},
 North-Holland Mathematics Studies, {\bf 116}. North-Holland Publishing Co., Amsterdam, (1985).

\bibitem{Grafakos}
L. Grafakos, {\em Modern Fourier Analysis, second edition}, Graduate Texts in Mathematics, {\bf 250}. Springer, New York, (2009).


\bibitem{HLMMY}  S. Hofmann, G. Lu, D. Mitrea, M. Mitrea, L Yan, {\em Hardy spaces associated to non-negative self-adjoint operators satisfying Davies-Gaffney estimates},
Mem. Amer. Math. Soc. {\bf 214} (2011), no. 1007.


\bibitem{HofmannMartell}
S. Hofmann, J.M. Martell, {\em $L^p$ bounds for Riesz Transforms and Square Roots Associated to Second Order Elliptic Operators}, Publ. Mat. {\bf 47} (2003), no. 2, 497--515.

\bibitem{HofmannMayboroda}
S. Hofmann, S. Mayboroda, {\em Hardy and $BMO$ spaces associated to divergence
form elliptic operators}, Math. Ann. {\bf 344} (2009), no. 1, 37--116.

\bibitem{HofmannMayborodaMcIntosh} S. Hofmann, S. Mayboroda, A. McIntosh, {\em Second order elliptic operators with complex bounded measurable coefficients in $L^p$, Sobolev and Hardy spaces}, Ann. Sci. \'Ecole	Norm. Sup. (4) {\bf 44} (2011), no. 5, 723--800.

\bibitem{Lerner} A.K. Lerner, {\em On sharp aperture-weighted estimates for square functions}. Preprint (2013).


\bibitem{Liu-Song} S. Liu, L. Song, {\em An atomic decomposition of weighted Hardy spaces associated to self-adjoint operators},
J. Funct. Anal. {\bf 265} (2013), no. 11, 2709--2723.

\bibitem{Mar-Pri-2} J.M. Martell, C. Prisuelos-Arribas {\em Weighted Hardy spaces associated with elliptic operators. Part II: Characterizations of $H^1_L(w)$}. Preprint 2017, arXiv:1701.00920.
 


\bibitem{Pri} C. Prisuelos-Arribas, {\em Weighted Hardy spaces associated with elliptic operators. Part III: Characterizations of $H^p_L(w)$ and the weighted Hardy space associated with the Riesz transform}. In preparation.


\bibitem{RdF} J.L. Rubio de Francia, {\em Factorization
theory and $A_p$ weights}, Amer. J. Math. \textbf{106} (1984), no. 3,
533--547.

\bibitem{St70} E.M. Stein, {\em Singular integrals and differentiability
properties of functions}, Princeton Mathematical Series. no. 30,
Princeton University Press, Princeton, N.J. (1970).


\bibitem{Stein-Weiss} E.M. Stein, G. Weiss, {\em On the theory of harmonic functions of several variables. I. The theory of
$H^p$-spaces}, Acta Math. {\bf 103} (1960) 25--62.

\bibitem{ST}
    J-O. Str\"omberg, A. Torchinsky,
 {\em Weighted Hardy spaces},
Lecture Notes in Mathematics, {\bf 1381}. Springer-Verlag, Berlin, (1989).

\end{thebibliography}
\end{document}